\documentclass[11pt]{amsart}

\usepackage{amsmath,amsthm,amsfonts,latexsym,amscd,amssymb,enumerate,color,hyperref}
\usepackage{pdftricks}
\input{xypic}

\usepackage{fullpage}

\usepackage{rotating}
\usepackage{multirow}
\usepackage{graphicx}
\begin{psinputs}
 \usepackage[usenames,dvipsnames]{pstricks}
 \usepackage{epsfig}
  \usepackage{pst-grad} 
 \usepackage{pst-plot} 
\end{psinputs}
\usepackage{hyperref}
\usepackage{epstopdf}

\usepackage{pdflscape}
\usepackage{afterpage}
\usepackage{capt-of}

\hypersetup{
colorlinks,
allcolors=blue
}

\usepackage{booktabs}
\usepackage{enumitem}
\usepackage{multicol}
\usepackage{color}
\usepackage{pifont}

\usepackage{lineno}

\usepackage[normalem]{ulem}

\def \<{\langle}
\def \>{\rangle}

\newcommand{\Z}{\mathbb{Z}}
\newcommand{\R}{\mathbb{R}}
\newcommand{\CP}{\mathbb{C}P}
\newcommand{\RP}{\mathbb{R}P}
\newcommand{\HP}{\mathbb{H}P}

\newcommand{\SU}{\mathrm{SU}}
\newcommand{\U}{\mathrm{U}}
\newcommand{\UU}{\mathrm{U}}
\newcommand{\Sp}{\mathrm{Sp}}
\newcommand{\SO}{\mathrm{SO}}
\newcommand{\OO}{\mathrm{O}}
\newcommand{\Spin}{\mathrm{Spin}}
\newcommand{\F}{\mathrm{F}}

\newcommand{\GG}{G}
\newcommand{\G}{G}
\newcommand{\HH}{H}
\newcommand{\K}{K}
\newcommand{\RS}{S}
\newcommand{\T}{T}
\newcommand{\LL}{L}
\newcommand{\N}{N}
\newcommand{\M}{M}
\newcommand{\X}{X}

\newcommand{\W}{\mathrm{W}}
\newcommand{\D}{\mathrm{D}}

\newcommand{\SP}{\mathbb{S}} 
\newcommand{\PS}{\mathbf{P}} 
\newcommand{\CaP}{\mathbf{Ca}P}

\newcommand{\mytableextraspace}{\addlinespace[.4em]}

\DeclareMathOperator{\trace}{trace}

\DeclareMathOperator{\diag}{diag}
\DeclareMathOperator{\Proj}{Proj}

\DeclareMathOperator{\Isom}{Isom}

\DeclareMathOperator{\Susp}{Susp}

\DeclareMathOperator{\Imag}{Im}

\newtheorem{thm}{Theorem}[section]
\newtheorem{cor}[thm]{Corollary}
\newtheorem{lem}[thm]{Lemma}
\newtheorem{prop}[thm]{Proposition}

\newtheorem{theorem}{Theorem}

\theoremstyle{definition}
\newtheorem{defn}[thm]{Definition}

\newtheorem{rem}[thm]{Remark}

\newtheorem*{ack}{Acknowledgements}

\numberwithin{equation}{section}
\numberwithin{table}{section}

\begin{document}


\author[F.~Galaz-Garc\'ia]{Fernando Galaz-Garc\'ia$^*$}
\address[F.~Galaz-Garc\'ia]{Institut f\"ur Algebra und Geometrie, Karlsruher Institut f\"ur Technologie (KIT), Karlsruhe, Germany.}
\email{galazgarcia@kit.edu}
\thanks{$^{*}$ Supported in part by the Deutsche Forschungsgemeinschaft grant GA 2050 2-1 within the  Priority Program SPP2026 ``Geometry at Infinity''.}


\author[M.~Zarei]{Masoumeh Zarei} 
\address[M.~Zarei]{Department of Pure Mathematics, Faculty of Mathematical Sciences, Tarbiat Modares University, Tehran, Iran and Institut f\"ur Algebra und Geometrie, Karlsruher Institut f\"ur Technologie (KIT), Karlsruhe, Germany.}
\curraddr{Beijing International Center for Mathematical Research, Peking University, Beijing, China}
\email{mzarei@bicmr.pku.edu.cn}


\title[Cohomogeneity one Alexandrov spaces in low dimensions]{Cohomogeneity one Alexandrov spaces in low dimensions}
\date{\today}




\begin{abstract}
We classify closed, simply-connected  cohomogeneity-one Alexandrov spaces in dimensions $5$, $6$ and $7$. We show that every closed, simply-connected smooth $n$-orbifold, $2\leq n\leq 7$, with a cohomogeneity one action is equivariantly homeomorphic to a smooth good orbifold of cohomogeneity one.
\end{abstract}
\subjclass[2010]{53C23, 53C20, 57S10}

\keywords{cohomogeneity one, Alexandrov space, orbifold}

\setcounter{tocdepth}{1}

\maketitle


\tableofcontents

\section{Introduction}

Alexandrov spaces (with curvature bounded from below) are complete length spaces with a lower curvature bound in the triangle comparison sense; they generalize  Riemannian manifolds with a uniform lower sectional curvature bound. Instances of Alexandrov spaces include  Riemannian orbifolds (with a lower sectional curvature bound), orbit spaces of isometric  actions of compact Lie groups on Riemannian manifolds with sectional curvature bounded below, or Gromov-Hausdorff limits of sequences of $n$-dimensional Riemannian manifolds with a uniform lower bound on the sectional curvature. 


In dimensions two and three, the basic topological properties of Alexandrov spaces are fairly well-understood. Indeed,  two-dimensional Alexandrov spaces are topological two-manifolds, possibly with boundary (see \cite[Corollary 10.10.3]{Burago}); closed (i.e.\ compact and without boundary) three-dimensional Alexandrov spaces are either topological three-manifolds or are homeomorphic to quotients of smooth three-manifolds by orientation reversing involutions with isolated fixed points, and closed four-dimensional Alexandrov spaces are homeomorphic to orbifolds (see \cite{GGG2}). In higher dimensions, however, similar general results are lacking and considering spaces with large isometry groups provides a systematic way of studying Alexandrov spaces. This yields manageable families of spaces with a reasonably simple structure but flexible enough  to generate interesting examples on which to test conjectures or carry out geometric constructions. This framework has been successfully used in the smooth category to construct, for instance, Riemannian manifolds satisfying given geometric conditions, such as positive Ricci or sectional curvature (see \cite{De,GVZ,ziller5,GroveZiller}). 

One of the measures for the size of an isometric action of a compact Lie group $G$ on an Alexandrov space $X$ is its \emph{cohomogeneity}, defined as the dimension of the orbit space $X/G$. This quotient space, when equipped with the orbital distance metric, is itself an Alexandrov space with the same lower curvature bound as $X$. From the point of view of cohomogeneity, transitive actions are the largest one can have. These actions preclude any topological or metric singularities: by the work of Berestovski\u\i\  \cite{Ber}, homogeneous Alexandrov spaces are isometric to Riemannian manifolds. The next simplest case to consider is when the orbit space is one-dimensional, i.e.\ when the action is of  \emph{cohomogeneity one}. Alexandrov spaces of cohomogeneity one were first studied in \cite{GS}, where the authors obtained a structure result and classified these spaces (up to equivariant homeomorphism) in dimensions $4$ and below. Simple instances of these spaces are, for example, spherical suspensions of homogeneous spaces $X$ with sectional curvature bounded below by $1$, equipped with the canonical suspension action of the transitive action on $X$.

It was shown in \cite[Proposition 5]{GS} that the orbit space of an isometric cohomogeneity one $\G$-action on a closed, simply-connected Alexandrov space $X$ is homeomorphic to a closed interval $[-1,1]$ and there exist compact Lie subgroups  $H$ and $K^\pm$ of $\G$ such that $\HH\subseteq K^\pm \subseteq \G$ and $K^\pm/H$ are positively curved homogeneous spaces. The group $H$ is the principal isotropy group of the action and the groups $K^\pm$ are isotropy groups of points in the orbits corresponding to the boundary points $\pm 1$ of the orbit space.  The groups $K^\pm$ are called \emph{non-principal isotropy groups} and the orbits $G/K^\pm$ are called \emph{non-principal orbits}. We collect these groups in the quadruple $(G,H,K^{-},K^{+})$, called the \emph{group diagram} of the action. The space $X$ is the union of two bundles whose fibers are cones over the positively curved homogeneous spaces $K^\pm/H$. Conversely, any diagram
$(G,H,K^{-},K^{+})$, with $K^{\pm}/H$ positively curved homogeneous spaces, gives rise to a cohomogeneity one Alexandrov space. In the present article we complete the classification of these spaces in dimensions $5$, $6$ and $7$ (assuming simply-connectedness),  identify which of these spaces are smooth orbifolds, and show that the underlying space of such an orbifold is equivariantly homeomorphic to a good orbifold.

Closed, smooth manifolds of cohomogeneity one have been classified by Mostert \cite{Mostert, Mostert2} and Neumann \cite{Neumann} in dimensions $2$ and $3$, and by Parker \cite{Parker} in dimension $4$, without assuming any restrictions on the fundamental group. In dimensions $5$, $6$ and $7$, Hoelscher \cite{Hoelscher} obtained the equivariant classification of closed smooth cohomogeneity one manifolds assuming simply-connectedness. It is well-known that these manifolds admit invariant Riemannian metrics and are therefore Alexandrov spaces of cohomogeneity one. In the topological category, the corresponding classification results in dimensions at most $7$ follow from combining the smooth  classification with the classification of closed, simply-connected cohomogeneity one topological manifolds with a non-smooth cohomogeneity one action in dimensions at most $7$, obtained in \cite{GGZ}. It was also shown in \cite{GGZ} that closed, simply-connected cohomogeneity one topological manifolds decompose as double cone bundles whose fibers are cones over spheres or the Poincar\'e homology sphere, and hence they admit invariant Alexandrov metrics. Our first main result completes the equivariant classification of closed, simply-connected Alexandrov spaces in dimensions $5$, $6$ and $7$:


\begin{theorem}
\label{T:CLASSIF}
Let $X$ be a closed, simply-connected Alexandrov space of dimension $5, 6$ or $7$ with an (almost) effective cohomogeneity one isometric action of a compact connected Lie group. If the action is not equivalent to a smooth action on a smooth manifold, then it is given by one of the diagrams in Table~\ref{TB:GROUP_DGMS_D5} if $\dim X=5$, Table~\ref{TB:GROUP_DGMS_D6} if  $\dim X=6$, or Tables~\ref{TB:GROUP_DGMS_D7_S3xS3}--\ref{TB:GROUP_DGMS_D7_G2_SU(4)_SU(3)xS1_Spin(7)} if $\dim X=7$.
\end{theorem}

We point out that the diagrams $(G,H,K^{-},K^+)$ in Tables~\ref{TB:GROUP_DGMS_D5}--\ref{TB:GROUP_DGMS_D7_G2_SU(4)_SU(3)xS1_Spin(7)} contain, as particular cases, the diagrams of non-smoothable cohomogeneity one actions on closed, simply-connected topological manifolds in \cite{GGZ}; in this special situation the positively curved homogeneous spaces  $K^\pm/H$ are either spheres or the Poincar\'e homology sphere. Compared to the smooth and topological cases, the number of closed, simply-connected cohomogeneity one Alexandrov spaces that are not manifolds  increases substantially, due to the fact that at least one of the positively curved homogeneous spaces $K^\pm/H$ is no longer a sphere or the Poincar\'e homology sphere. In many cases, we can identify the spaces in Theorem~\ref{T:CLASSIF} as joins, suspensions, products or bundles of familiar spaces. Moreover, many of the spaces in Theorem~\ref{T:CLASSIF} are equivariantly homeomorphic to smooth cohomogeneity one orbifolds. Indeed, they admit a double cone bundle decomposition, where the cones are taken over spherical homogeneous spaces; this structure characterizes closed, smooth  orbifolds of cohomogeneity one whose orbit space is a closed interval (see \cite{Gonzalez}).  A natural question, then,  is whether  there exists a \emph{good}  orbifold structure on the underlying topological space $|Q|$ of a given orbifold group diagram, that is, whether $|Q|$ is equivariantly homeomorphic to a quotient of a cohomogeneity one smooth manifold. Our second main result gives a positive answer to this question in dimensions at most $7$:


\begin{theorem}
\label{T:GOOD_ORBIFOLDS}
Every closed, simply-connected smooth orbifold of dimension at most $7$ with an (almost) effective cohomogeneity one smooth action is equivariantly homeomorphic to a good cohomogeneity one smooth  orbifold.
\end{theorem}

As in the smooth and topological cases, the proof of Theorem~\ref{T:CLASSIF} follows from a case-by-case analysis of the possible group actions. 
Using dimension restrictions, one first determines the possible groups that can act.  One then considers each group action individually, taking into account the fact that the groups must satisfy restrictions imposed by the fact that the homogeneous spaces $K^\pm/H$ are positively curved. In this way, one obtains all the possible diagrams $(G,H,K^{-},K^{+})$, which determine the equivariant type of the Alexandrov space. Recognition results for specific types of actions help us identify the topological type of the space.
To prove Theorem~\ref{T:GOOD_ORBIFOLDS}, we first compute the orbifold fundamental group in each case. We then find a manifold with a cohomogeneity one action and a commuting action of the orbifold fundamental group whose quotient induces the diagram of the cohomogeneity one orbifold under consideration. 
\\

Our article is divided as follows. In Section~\ref{S:PRELIM} we collect background material on cohomogeneity one Alexandrov spaces and prove some results we will use in the proof of Theorem~\ref{T:CLASSIF}. The proof of this theorem is contained in Section~\ref{S:PROOF_ALEX_CLASSIF}. Finally, in Section~\ref{S:ORBIFOLDS} we recall some basic facts on orbifolds and prove Theorem~\ref{T:GOOD_ORBIFOLDS}. 

\begin{ack}
The authors would like to thank Wilderich Tuschmann at the Karlsruher Institut f\"ur Technologie (KIT),  Alexander Lytchak and Christian Lange at the Universit\"at K\"oln, and Burkhard Wilking at the Universit\"at M\"unster for their hospitality and for helpful conversations.
\end{ack}


\section{Preliminaries}
\label{S:PRELIM}
In this section,  we collect some background material which we will use in the proof of Theorem~\ref{T:CLASSIF}.

\subsection{Group actions}
 Let $X$ be a topological space and let $x$ be a point in $X$.  Given a topological (left) action $\G\times X\rightarrow X$ of a Lie group $\G$, we let $\G(x)=\{\,gx \mid g\in \G \,\}$ be the \emph{orbit}  of $x$ under the action of $\G$. The \emph{isotropy group} of $x$ is the subgroup $\G_x=\{\, g\in \G \mid gx=x\,\}$. Observe that $\G(x)\approx \G/\G_x$. We will denote the orbit space of the action by $X/G$ and let  $\pi:X\rightarrow X/G$ be the orbit projection map.  The \emph{(ineffective) kernel} of the action is the subgroup $\K=\bigcap_{x\in X}G_x$. The action is \emph{effective} if $\K$  is the trivial subgroup $\{e\}$ of $G$; the action is \emph{almost effective} if $\K$ is finite.

We will say that two $G$-spaces are \emph{equivalent} if they are equivariantly homeomorphic. From now on, we will suppose that $G$ is compact and connected, and assume that the reader is familiar with the basic notions of compact transformation groups (see, for example, Bredon \cite{Bredon}). We will assume all spaces to be connected, unless stated otherwise.  

\subsection{Alexandrov spaces}

A finite (Hausdorff) dimensional length space $(X, d)$ has curvature bounded  below by $k$ if every point $x\in X$ has a neighborhood $U$ such that, for any collection
of four different points $(x_0, x_1, x_2, x_3)$ in $U$, the following condition holds:
\begin{equation*}
 \angle _{x_1x_2} (k) + \angle _{x_2x_3} (k)+ \angle _{x_3x_1} (k) \leq 2\pi.
\end{equation*}
Here, $\angle _{x_ix_j} (k)$, called the \textit{comparison angle}, is the angle at $x_0(k)$ in the geodesic triangle in $M^2_k$, the simply-connected Riemannian $2$-manifold with constant curvature $k$, with  vertices \linebreak $(x_0(k), x_i(k), x_j (k))$, which are the isometric images of $(x_0, x_i, x_j )$. 
An \emph{Alexandrov space} is a complete length space with finite Hausdorff dimension and curvature bounded below by $k$ for some $k\in \mathbb{R}$. Recall that the Hausdorff dimension of an Alexandrov space is an integer and is equal to its topological dimension. The \emph{space of directions} of a general Alexandrov space $X^n$ of dimension $n$ at a point $x$ is, by definition, the completion of the space of geodesic directions at $x$.
We will denote it by $\Sigma_xX^n$. It is a compact Alexandrov space of dimension $n-1$ with curvature bounded below by $1$. We refer the reader to \cite{Burago,BGP} for the basic results on Alexandrov geometry. We will say that an Alexandrov space is \emph{closed} if it is compact and has no boundary.


\subsection{Group actions on Alexandrov spaces}
Let $X$ be an $n$-dimensional Alexandrov space. Fukaya and Yamaguchi proved in \cite[Theorem 1.1]{FY} that $\Isom(X)$, the isometry group of $X$, is a Lie group. Moreover, $\Isom(X)$ is compact, if $X$ is compact and connected (see \cite[p.~370, Satz I]{DW} or 
\cite[Corollary~4.10 and its proof in pp. 46--50]{KN}). As in the Riemannian case, the maximal dimension of $\Isom(X)$ is $n(n+1)/2$ and, if equality holds, $X$ must be isometric to a Riemannian manifold (see \cite[Theorems 3.1 and 4.1]{GGG}).


As for locally smooth actions (see \cite[Ch.~IV, Section 3]{Bredon}), for an isometric action of a compact Lie group $G$ on an Alexandrov space $X$ there also exists a maximal orbit type $G/\HH$  (see \cite[Theorem 2.2]{GGG}). This orbit type is the \emph{principal orbit type} and orbits of this type are called \emph{principal orbits}. A non-principal orbit is \emph{exceptional} if it has the same dimension as a principal orbit. 

The structure of the space of directions in the presence of an isometric action is given by the following proposition. 




\begin{prop}[\protect{\cite[Proposition 4]{GS}}]
\label{normal space}
Let $X$ be an Alexandrov space with an isometric $G$-action and fix $x\in X$ with $\dim(G/G_x) > 0$. Let $S_x\subseteq \Sigma_xX$ be the unit tangent space to the orbit $G(x)\simeq G/G_x$, and let $S_x^\perp=\{ v \in \Sigma_x X : \angle(v,w) = \pi/2\,\,\,  \text{for all } w\in S_x\}$ be the set of normal directions to $S_x$. Then the following hold:
\begin{itemize}
	\item[(1)] The set $S_x^\perp$ is a compact, totally geodesic Alexandrov subspace of $\Sigma_xX$ with curvature bounded below by 1, and the space of directions $\Sigma_xX$ is isometric to the join $S_x\ast S_x^\perp$ with the standard join metric.
	\item[(2)] Either $S_x^\perp$ is connected or it contains exactly two points at distance $\pi$.
\end{itemize}
\end{prop}




\subsection{Alexandrov spaces of cohomogeneity one}
In this subsection we collect basic facts on cohomogeneity one Alexandrov spaces and prove some preliminary results that we will use in the proof of Theorem~\ref{T:CLASSIF}. For cohomogeneity one actions on smooth or topological manifolds, we refer the reader to \cite{Hoelscher} or \cite{GGZ}, respectively.


\begin{defn}
Let $X$ be a connected $n$-dimensional Alexandrov space with an isometric action of a compact connected Lie group  $G$. The action is of \emph{cohomogeneity one} if the orbit space is one-dimensional or, equivalently, if there exists an orbit of dimension $n-1$. A connected Alexandrov space with an isometric action of cohomogeneity one is a \emph{cohomogeneity one Alexandrov space}.
\end{defn}

Cohomogeneity one Alexandrov spaces were first studied in \cite{GS}. Recall that the orbit space $X/G$ of an Alexandrov space $X$ by an isometric action of a group $G$ with closed orbits is again an Alexandrov space (see \cite[Proposition 10.2.4]{Burago}). Since one-dimensional Alexandrov spaces are topological manifolds, the orbit space of a cohomogeneity one Alexandrov space is homeomorphic to a connected $1$-manifold (possibly with boundary). When the orbit space is homeomorphic to $[-1,1]$, we denote  the isotropy groups corresponding to a point in the orbit mapped to $\pm 1$ by $\K^\pm$. By the Isotropy Lemma (see \cite[Lemma 2.1]{GGG}) and the fact that principal orbits are open and dense, the orbits that project to the interior $(-1,1)$ of the orbit space all have the same isotropy group $\HH$ (up to conjugacy) and $H$ is a subgroup of $K^\pm$. The subgroup $\HH$ is the principal isotropy group of the action and the corresponding orbits are the principal orbits.  Let us now show that $H$ is a proper subgroup of $K^\pm$. It suffices to show that if $\dim K^\pm=\dim H$, then $K^\pm\neq H$. Observe first that, in this case, $S^\perp = \SP^0$ with a transitive action of $K^\pm$ with isotropy $H$. Hence $K^\pm/H = \SP^0$, which shows that $K^\pm \neq H$. We call the orbits mapped to $\pm 1$ \emph{non-principal} orbits.


Let $X$ be a closed cohomogeneity one Alexandrov $G$-space. Since the orbit space $X/G$ must be a compact one-manifold, it  must be either a circle or a closed interval. When $X/G$ is a circle, $X$ is equivariantly homeomorphic to a fiber bundle over $\SP^1$ with fiber a principal orbit $G/H$. In particular, $X$ is a smooth manifold (see \cite[Theorem A]{GS}). Since we are interested in non-manifold Alexandrov spaces, we will focus our attention on the case where $X/G$ is a compact interval. 

A cohomogeneity one $G$-action on a closed Alexandrov space whose orbit space is an interval determines a group diagram 
\begin{equation*}
	(G,H, K^-, K^+),
\end{equation*}
where $K^\pm$ are isotropy subgroups at the non-principal orbits corresponding to the endpoints of the interval, and $H$ is the principal isotropy group of the action. The following theorem determines the structure of closed cohomogeneity-one Alexandrov spaces with orbit space an interval.


\begin{thm}[\protect{\cite[Theorem A]{GS}}]
\label{T:ALEX_STRUCTURE}
Let $X$ be a closed Alexandrov space with an effective isometric $G$-action of cohomogeneity one with principal isotropy $\HH$ and orbit space homeomorphic to $[-1,1]$. Then $X$ is the union of two
fiber bundles over the two singular orbits whose fibers are cones over positively curved homogeneous  spaces, that is, 
\[
X = G/K^-\times_{K^-}C(K^-/H)\bigcup_{G/H} G/K^+\times_{K^+}  C(K^+/H).
\]
 The group diagram of the action is given by $(G, H,  K^-, K^+)$, where $K^\pm/H$ are positively curved homogeneous spaces.
 Conversely, a group diagram $(G, H,  K^-, K^+)$, where $K^\pm/H$ are positively curved homogeneous spaces, determines a cohomogeneity one Alexandrov space.
\end{thm}

We will use the following proposition to identify equivalent actions. 


\begin{prop} [\protect{\cite[Proposition 9]{GS}}]
\label{P:Equivalent_actions}
If a cohomogeneity  one Alexandrov space is given by a group diagram $(G,H,K^{-},K^{+})$, then any of the following operations on the group diagram will result in an equivalent Alexandrov space:
\begin{enumerate}
 	\item Switching $\K^{-}$ and $\K^{+}$,
	\item Conjugating each group in the diagram by the same element of $G$,
	\item Replacing $K^{-}$ with $gK^{-}g^{-1}$  for $g\in N(\HH)_0$.
\end{enumerate}
Conversely, the group diagrams for two equivalent cohomogeneity one, closed Alexandrov space must be mapped to each other by some combination of these three operations.
\end{prop}


Let $G$ be a compact connected Lie group acting on a closed Alexandrov space  $X$ with cohomogeneity one and let $\pi: X\to  X/G=[0, 1]$
be  the projection map. A minimizing geodesic $\gamma: [0, d]\to X$ between non-principal orbits has the following properties (see \cite[Lemma~2.1]{GGG}):\\
\begin{itemize}
\item it goes through all principal orbits,
\item for all $t\in (0,d)$, $\HH=\G_{c(t)}\subset \G_{c(0)}, \G_{c(d)}$, and
\item the direction of $\gamma$ is horizontal.
\end{itemize}
We set $K^{-}=\G_{c(0)}$ and $K^{+}=\G_{c(d)}$. We call such a geodesic a \emph{normal geodesic}.


\begin{defn}
\label{Primitive Action}
We say that the cohomogeneity one Alexandrov space $X$ is \emph{non-primitive} if
it has some group diagram representation $(G, H,  K^-, K^+)$ for which there is a
proper connected closed subgroup $L\subset G$ with $K^{\pm} \subset L$. It then follows that
$(L, H,  K^-, K^+)$  is a group diagram which determines some cohomogeneity one
Alexandrov space $Y$.
\end{defn}


\begin{prop}[\protect{\cite[p.~96]{GS}}]\label{P:Primitive Action}
Take a non-primitive cohomogeneity one Alexandrov space $X$ with $L$ and $Y$ as in Definition~\ref{Primitive Action}. Then $X$ is equivalent to $(G\times Y)/L$, where $L$ acts on $G\times Y$ by $l\cdot(g,y)=(gl^{-1}, ly)$. 
Hence, there is a fiber bundle
\[
Y\to X\to G/L.
\]
\end{prop}


\begin{defn}
A cohomogeneity one action of a Lie group $G$ on an Alexandrov space
$X$ is called \emph{reducible} if there is a proper closed normal subgroup of $G$ that acts on $X$ with the same orbits.
\end{defn}

We now recall the following results which describe the reduction or extension of certain cohomogeneity one actions (cf. \cite[Section 1.11]{Hoelscher} and \cite[Section 2]{GS}). These results show why it is natural to consider only non-reducible actions.


\begin{prop}[\protect{\cite[Proposition 11]{GS}}]
\label{P: reduction}
Let $X$ be the cohomogeneity one Alexandrov space given by the group
diagram $(G, H, K^-, K^+)$ and suppose that $G=G_1\times G_2$ with $\Proj_2(H)=G_2$. Then
the subaction of $G_1\times1$ on $X$ is also of cohomogeneity one, has the same orbits as the action of $G$, and has
 isotropy groups $K^{\pm}_1=K^{\pm}\cap (G_1\times 1)$ and $H_1=H\cap (G_1\times 1)$.
\end{prop}

For the next proposition we will need the concept of a \emph{normal extension}, which
we now recall.


\begin{defn}
 Let $X$ be a cohomogeneity one Alexandrov space with group diagram \linebreak
$(G_1, H_1, K^-, K^+)$ and let $L$ be a compact, connected subgroup of $N(H_1)\cap
N(K^-)\cap N(K^+)$. Observe that the subgroup $L\cap H_1$ is normal in $L$ and define 
$G_2 := L/(L\cap H_1)$. We can then define an action of $G_1\times G_2$ on $X$ orbitwise by letting
\[
(\hat{g_1}, [l])\cdot g_1(G_1)_x = \hat{g_1}g_1l^{-1}(G_1)_x
\]
on each orbit $G_1/(G_1)_x$ for $(G_1)_x = H_1$ or $K^{\pm}$. Such an extension is called a
\emph{normal extension} of $G_1$.
\end{defn}

\begin{prop}[\protect{\cite[Proposition 12]{GS}}]
A normal extension of $G_1$ describes a cohomogeneity one action
of $G := G_1\times G_2$ on $X$ with the same orbits as $G_1$ and with group diagram
\[
(G_1\times G_2, (H_1\times 1)\Delta L, (K^-\times 1)\Delta L, (K^+ \times 1)\Delta L),
\]
where $\Delta L = \{(l, [l]) : l\in L\}$.
\end{prop}


\begin{prop}[\protect{\cite[Proposition 13]{GS}}]
For $X$ as in Proposition \ref{P: reduction}, the action by $G = G_1\times G_2$ occurs
as the normal extension of the reduced action of $G_1\times 1$ on $X$.
\end{prop}

By the above propositions, it is natural then to consider only non-reducible actions in the classification.


\subsection{Further tools} The following proposition, whose proof is as in \cite[ Proposition 1.25]{Hoelscher},  yields bounds on the dimension of a Lie group acting by cohomogeneity one in terms of the dimension of a principal isotropy subgroup.


\begin{prop}
 Let $X$ be a closed Alexandrov space of cohomogeneity one with group diagram $(G, H,  K^-, K^+)$. Suppose that $G$ acts non-reducibly on $X$ and that $G$ is the product of groups
\begin{equation*}
\label{P:BOUND_DIMH}
	G=\SU(4)^i\times (G_2)^j\times \Sp(2)^k\times \SU(3)^l\times (\RS^3)^m\times (\RS^1)^n.
\end{equation*}
 Then
\begin{equation*}
	\dim(H)\leq 10i+8j+6k+4l+m.
\end{equation*}
\end{prop}

We now state some useful results on the fundamental group of cohomogeneity one Alexandrov spaces. Their proofs follow as in the manifold case (see \cite[Section 1.6]{Hoelscher} and \cite[Section 4]{GGZ}).


\begin{prop}[Corollary to the van Kampen Theorem \protect{\cite[Proposition 1.8]{Hoelscher}}]
\label{fundamental group}
Let $X$ be the closed cohomogeneity one Alexandrov space given by the group diagram $(G, H, K^-, K^+)$ with $\dim(K^{\pm}/H)\geq 1$. Then 
\[
\pi_1(X) \cong \pi_1(G/H)/N^-N^+,
\]
 where 
 \[
 N^{\pm} = \ker\{\pi_1(G/H)\to \pi_1(G/K^{\pm})\} = \mathrm{Im}\{\pi_1(K^{\pm}/H) \to \pi_1(G/H)\}.
 \]
\end{prop}



\begin{cor}[\protect{\cite[Corollary 4.4]{GGZ}}]
\label{C:Simply-connected Orbits}
Let $X$ be the closed simply-connected cohomogeneity one Alexandrov space given by the group diagram  $(G, \HH,  \K^-, \K^+)$, with  $\dim (\K^\pm/\HH) \geq 1$, and $\K^-/\HH=\SP^l$, for $l\geq 2$. Then $G/\K^+$ is simply-connected and, if $G$ is  connected, then $\K^+$ is also connected. 
\end{cor}


\begin{lem}[\protect{\cite[Lemma 1.10]{Hoelscher}}]
\label{L:SIMP_CONN_COND}
Let $X$ be the closed cohomogeneity one Alexandrov space given by the group diagram $(G, H,  K^-, K^+)$. Denote $H^\pm= H\cap K^\pm_0$, and let $\alpha^i_\pm: [0,1]\to \K^\pm_0$ be curves that generate $\pi_1(\K^\pm/H)$, with $\alpha^i_\pm(0)=1\in G$. The space $X$ is simply-connected if and only if 
\begin{enumerate}
\item $H$ is generated as a subgroup by $H^-$ and $H^+$, and
\item $\alpha^{i}_-$ and $\alpha^{i}_+$ generate $\pi_1(G/H_0)$.
\end{enumerate}
\end{lem}

We will use the following results on transitive actions.


\begin{lem}[cf. \protect{\cite[Lemma 4.11]{GGZ}}] 
\label{L:M_TRANS}
Let $G_1$ be a  compact, connected, simply-connected, simple Lie group of dimension $n$. Assume that $G_1$ is, up to a finite cover, the only Lie group that acts transitively and (almost) effectively on a manifold $M$ with isotropy group $H$. Let $G_2$ be a compact, connected Lie group of dimension at most $n-1$. If $G_1\times G_2$ acts transitively on $M$, then the following hold:
\begin{enumerate}
	\item The $G_2$ factor acts trivially on $M$ and
	\item the isotropy group $K$ of the $(G_1\times G_2)$-action is  $H\times G_2$.
\end{enumerate}
\end{lem}


\begin{proof}
Let $L\subseteq G_1\times G_2$ be the kernel of the action of $G_1\times G_2$ on $M$. Then $(G_1\times G_2)/L$ is isomorphic to $G_1$. Hence, $\dim G_2= \dim L$. Since $L$ is a normal and connected subgroup of $G_1\times G_2$, $\Proj_1(L)$ is a normal connected subgroup of $G_1$. Thus $\Proj_1(L)$ is trivial, since $\dim G_2\leq n-1$. As a result, $\LL= 1\times G_2 $ and $K=H\times G_2$. 
\end{proof}


\begin{prop}[\protect{\cite[Ch. 1, \S 5 Proposition 7]{On}}]
\label{P:Trans. of G_0}
Let a Lie group $G$ act  transitively on a manifold $M$. Then $G_0$ acts transitively on any connected component of $M$. In particular, if $M$ is connected, then $G_0$ acts transitively on $M$, and $G=G_0G_x$, for all $x\in M$.
\end{prop}


 The following  two results  give restrictions on the groups that may act by cohomogeneity one on a closed Alexandrov space.
  The next proposition can be found in \cite[Proposition~1.19]{Hoelscher} for smooth actions. It was proven in \cite{GGZ} in the slightly more general case of topological actions on topological manifolds. The proof for Alexandrov spaces follows as in the topological case \cite[Proposition 4.7]{GGZ}, taking into account that, by the Principal Orbit Theorem for Alexandrov spaces \cite[Theorem 2.2]{GGG}, all principal isotropy groups are conjugate to each other and conjugate to a subgroup of   non-principal isotropy groups.
 
 
\begin{prop}[cf. \protect{\cite[Proposition 4.7]{GGZ}}]
\label{L:DIM_BOUND}
If  a compact connected Lie group $G$ acts (almost) effectively on an Alexandrov space with principal orbits of dimension $k$, then $k\leq \dim G \leq k(k+1)/2$.
\end{prop}

An argument as in the proof of \cite[Proposition 1.18]{Hoelscher} yields the following lemma:


\begin{lem}\label{P:ABELIAN_G}
Let $X$ be a closed, simply-connected Alexandrov space with an (almost) effective cohomogeneity one action of a compact Lie group $G$. Suppose that the following conditions hold:
\begin{itemize}
	\item $G=G_1\times T^m$ and $G_1$ is semisimple;
	\item $G$ acts non-reducibly;
	\item at least one of the homogeneous spaces $K^\pm/H$ is other than standard spheres.
\end{itemize} 
Then, $G_1\neq 1$ and $m\leq 1$. Moreover, if $m=1$, then   one of the homogeneous spaces $K^\pm/H$, say $K^{-}/H$, is a circle  and $K_0^-=H_0\cdot S^-$, where $S^-$ is a circle group with $\Proj_2(S^-)=T^1$ and $K^+_0\subset G_1\times 1$.  
\end{lem}


\subsection{Special actions and recognition results}

In this subsection we list some special types of cohomogeneity one actions and prove some recognition results that will allow us to identify such actions (cf. \cite[1.21]{Hoelscher}).


\begin{defn}[Product action] 
Let $G_1$ and $G_2$ be  Lie groups such that $G_1$ acts on an Alexandrov space $X$ with cohomogeneity one  and  $G_2$  acts on a homogeneous space $G_2/ L$ transitively.  We call the natural action of $G_1\times G_2$ on  $X\times G_2/L$ given by 
\[
(g_1, g_2)\cdot (x, gL)=(g_1x, g_2gL)
\]
the \emph{product action} of $G_1\times G_2$.
\end{defn}


\begin{prop}
\label{P:Product-action}
Suppose that $G_1$ acts on an Alexandrov space $X$ with cohomogeneity one and with group diagram 
$(G_1,H,K^-,K^+)$, and $G_2$ acts transitively on the homogeneous space $G_2/L$. Then the product action of $G_1\times G_2$ on $X\times G_2/L$ is of cohomogeneity one with group diagram
\begin{equation}\label{Eq_GD}
(G_1\times G_2,  H\times L,  K^-\times L,  K^+\times L).
\end{equation}
Conversely, a cohomogeneity one action of  $G_1\times G_2$ with the above group diagram, and $G_1/K^\pm$ positively curved  homogeneous spaces,  is  equivalent  to a product action of  $G_1\times G_2$ on $X\times G_2/L$, where $X$ is the cohomogeneity one Alexandrov space determined by the diagram 
$(G_1,H,K^-,K^+)$. 
\end{prop}


\begin{proof}
It is clear that the product action of $G_1\times G_2$ on $X\times G_2/L$ is of cohomogeneity one. Now we prove that its group diagram is as in \eqref{Eq_GD}. Let $\gamma$ be a normal geodesic   between the non-principal orbits $G_1/K^\pm$ in $X$ giving the group diagram $(G_1,H,K^-,K^+)$. If we fix a $G_2$-invariant metric on $G_2/L$, then, in the product metric on $X\times G_2/L$, the curve $\tilde{\gamma}=(\gamma,1)$ is a shortest geodesic between non-principal orbits. The resulting diagram is
\[
(G_1\times G_2,  H\times L,  K^-\times L,  K^+\times L),
\]
as claimed. The converse follows from Proposition~\ref{P:Equivalent_actions}.
\end{proof}


\begin{defn}[Join action]
Let $G_1$ and $G_2$ be two Lie groups which act on  Alexandrov spaces $X_1$ and $X_2$, respectively. The  action of $G_1\times G_2$ on  $X_1\ast X_2$ is called \textit{join action}, if $G_1\times G_2$ acts on  $X_1\ast X_2$ naturally, i.e. 
\[
(g_1, g_2)\cdot [(x, y, t)]=[(g_1x, g_2y, t)].
\]
\end{defn}


\begin{prop}\label{P:Join-action}
If two Lie groups $G_1$ and $G_2$ act transitively on  positively curved homogeneous spaces $M_1$ and $M_2$ with isotropy groups $H_1$ and $H_2$, respectively, then the join action of  $G=G_1\times G_2$ on  $M_1\ast M_2$ is of cohomogeneity one with the following diagram:

\[
(G_1\times G_2,  H_1\times H_2,  G_1\times H_2,  H_1\times G_2).
\]
Conversely, a cohomogeneity one action of  $G_1\times G_2$ with the above group diagram, and $G_i/H_i$  positively curved  homogeneous spaces, for $i=1,2$, is  equivalent  to the  join action of  $G$ on $(G_1/H_1)\ast (G_2/H_2)$. 
\end{prop}

\begin{proof}
Let $x\in M_1$ and $y\in M_2$  be  such that $H_1=(G_1)_x$ and $H_2=(G_2)_y$. The curve 

\begin{align*}
\gamma:[0, \pi/2]&\to X_1\ast X_2\\
t&\mapsto [x, y, t]
\end{align*}
is a shortest geodesic between $[x, y, 0]$ and $[x, y, \pi/2]$ which goes through all orbits. Furthermore, $G_{\gamma(0)}=H_1\times G_2$, $G_{\gamma(\pi/2)}=G_1\times H_2$, and  $t\in (0, \pi/2)$,   $G_{\gamma(t)}=H_1\times H_2$. Therefore, the action is of cohomogeneity one with the  given diagram. By Proposition~\ref{P:Equivalent_actions}, the converse is immediate.
\end{proof}


\begin{defn}[Suspension action]
Let $G$ be a Lie group which acts on an  Alexandrov space $X$.  The  action of $G$ on  $\Susp(X)$ is called \textit{suspension action}, if $G$ acts on  $\Susp(X)$ as follows:
\[
g\cdot [(x, t)]=[(gx, t)].
\]
\end{defn}


\begin{prop}\label{P:Suspension-action}
Let  $G$ act transitively on a positively curved homogeneous space $M$ with isotropy group $H$.  Then the suspension action of  $G$ on  $\Susp(M)$ is of cohomogeneity one with diagram
$(\GG,  \HH,  \GG,  \GG)$. Conversely, a cohomogeneity one action of  $G$ with the above group diagram, and $G/H$  a positively curved  homogeneous space,  is  equivalent  to the  suspension action of  $G$ on $\Susp(G/H)$. 
\end{prop}

\begin{proof}
Let $x\in M$ be such that $H=G_x$. The curve 

\begin{align*}
\gamma:[0, \pi]&\to\Susp M \\
t&\mapsto [x,  t]
\end{align*}
is a shortest geodesic between $[x, 0]$ and $[x, \pi]$ which goes through all orbits. Furthermore, $G_{\gamma(0)}=G$, $G_{\gamma(\pi)}=G$, and for $t\in (0, \pi)$,  $G_{\gamma(t)}=H$. Therefore, the action is of cohomogeneity one with  given diagram. By Proposition~\ref{P:Equivalent_actions}, the inverse is clear.
\end{proof}


\begin{prop}[Spin action]
\label{P:Spin-action}
Let $G$ be a compact, simply-connected Lie group which  acts almost effectively and by cohomogeneity one on a closed, simply-connected Alexandrov space $X^n$ with group diagram $(G, H, K^{-}, K^{+})$. If $\dim G = n(n-1)/2$, then $G$ is isomorphic to $\Spin(n)$ and the action is equivalent to the cohomogeneity one action of  $\Spin(n)$ on  $\Susp(\RP^{n-1})$, which is the suspension of the transitive action of $\Spin(n)$  on  $\RP^{n-1}$.
\end{prop}


\begin{proof} The proof of this proposition is analogous to Hoelscher's proof in \cite[Proposition 1.20]{Hoelscher} with slight changes. Namely,  since in our case $\K^{\pm}$ is not a sphere,  ${\HH}_0\neq\HH$. As the only proper subgroup of $\Spin(n)$  containing $\Spin(n-1)$ is $N_{\Spin(n)}(\Spin(n-1))$, we have $H=N_{\Spin(n)}(\Spin(n-1))$ and $K^{\pm}/H=\RP^{n-1}$.
\end{proof}


\subsection{Transitive actions on spheres.}
We conclude this section by recalling the  well known classification of almost effective transitive actions on spheres  (see \cite{Asoh} and the references therein). We will use this classification throughout our work.

\begin{thm}[\protect{\cite[Section 2.1]{Asoh}}]
\label{T:ASOH}
Suppose that a compact, connected Lie group  $G$ acts almost effectively and transitively
on the sphere $\SP^{n-1} (n\geq 2)$. Then the $G$-action on $\SP^{n-1}$ is equivalent to
the following linear action of $G$ on $\SP^{n-1}$ via the standard representation $\iota: G\to \SO(n)$
with an isotropy subgroup $H$.
\begin{itemize}
\item[(i)] If $n$ is odd, then $G$ is simple and $(G, n, \iota, H)$ are
\begin{align}
	& (\SO(n), n, \rho_n, \SO(n- 1)),\label{Eq_T_1}\\ 
	& (\mathrm{G}_2, 7, \phi_2, \SU(3)). \label{Eq_T_2}
\end{align}
\item[(ii)] If $n$ is even, then $G$ contains a simple normal subgroup  $G'$ such that
the restricted $G'$-action on $\S^{n-1}$ is transitive and $G/G'$ is of rank at most $1$, and
$(G, n, \iota, H)$ is
\begin{align}
	& (\SO(n), n, \rho_n, \SO(n-1)(n\neq 4), \label{Eq_T_3}\\ 
	& (\Spin(7), 8, \Delta_7, \mathrm{G}_2), \label{Eq_T_4}\\ 
	& (\UU(k), 2k, (\mu_k)_{\R}, \UU(k-1)), \label{Eq_T_5}\\ 
	& (\Sp(k), 4k, (\nu_k)_{\R}, \Sp(k-1)), \label{Eq_T_6}\\ 
	&(\Sp(k)\times \RS^{i}, 4k, (\nu_k\otimes \mu_1^{*} (\nu^{*}))_{\R}, \Sp(l- 1)\times \RS^i), \quad (i=1, 3) \label{Eq_T_7}\\ 
	&(\Spin(9), 16, \Delta_9, \Spin(7)),\label{Eq_T_8}\\ 
	& (\SU(k), 2k, (\mu_k)_{\R}, \SU(k-1)). \label{Eq_T_9}
\end{align}
\end{itemize}
\end{thm}


\section{Proof of Theorem~\ref{T:CLASSIF}}
\label{S:PROOF_ALEX_CLASSIF}


\subsection{Possible groups}
\label{SS:POSSIBLE_GPS}
We first list the Lie groups  that can act (almost) effectively and by cohomogeneity one on an Alexandrov space of dimension $5$, $6$ or $7$. This list is obtained as in the manifold case, and we refer the reader to \cite[Section 1.24]{Hoelscher} for more details.

Let $G$ be a compact connected Lie group acting (almost) effectively and by cohomogeneity one on an $n$-dimensional Alexandrov space $X^n$. 
It is well-known that every compact and connected Lie group has a finite cover of the form $G_{ss}\times T^k$, where $G_{ss}$ is semisimple and simply-connected, and $T^k$ is a torus. 
The classification of simply-connected semisimple Lie groups is also well-known and all the possibilities are listed in Table \ref{TB:ACTING_GROUPS} for dimensions $21$ and less.

If an arbitrary compact connected Lie group $G$ acts on an Alexandrov space $X$, then every cover $\widetilde{G}$ of $G$ still acts on $X$, although less effectively. Hence, allowing for a finite ineffective kernel, and because $G$ will always have dimension $21$ or less, we can assume that $G$ is a product of groups from Table \ref{TB:ACTING_GROUPS} with a torus $T^k$. 

          
\begin{table}[!htbp]
\begin{center}
\small{
\begin{tabular}{p{5.5cm}p{2.5cm}l}\toprule
Group 	& Dimension  & Rank \\  \mytableextraspace
\midrule
$\RS^3\cong \SU(2)\cong \Sp(1)\cong \Spin(3)$ & $3$		& $1$ \\ \mytableextraspace
\midrule
$\SU(3)$											&  $8$		& $2$\    \\ \mytableextraspace
$\Sp(2)\cong \Spin(5)$							& $10$	& $2$\\ \mytableextraspace
$\mathrm{G}_2$											& $14$	& $2$ \\ \mytableextraspace
\midrule
$\SU(4)\cong \Spin(6)$							& $15$	& $3$ \\ \mytableextraspace
$\Sp(3)	$										& $21$	& $3$\\  \mytableextraspace
$\Spin(7)$											& $21$	& $3$\\  \mytableextraspace
\bottomrule
\end{tabular}
}
\end{center}
\caption{\label{TB:ACTING_GROUPS}Compact,  connected,  simply-connected simple Lie groups in dimensions $21$ and less}
\end{table}

          
In Table~\ref{TB:SUBGROUPS} we list the proper, connected, non-trivial closed subgroups of the groups in Table~\ref{TB:ACTING_GROUPS}, in dimensions at most $15$, and of $T^2$; these are the dimensions that will be relevant in our case. These subgroups are well-known (see, for example, \cite{Dy} or \cite[Tables 2.2.1 and 2.2.2]{HoelscherThesis}). 

          
\begin{table}[!htbp]
\begin{center}
\small{
\begin{tabular}{p{2cm}p{9.7cm}}\toprule
Group 			& Subgroups  							  \\  \mytableextraspace
\midrule
$T^2$ 		 			& $\{(e^{ip\theta},e^{iq\theta})\}$	 	  \\[.1cm] \mytableextraspace
$\RS^3$					& $\{e^{x\theta} = \cos\theta+x\sin\theta\}$, where $x\in \mathrm{Im}(S^3)$. 		 \\[.1cm] \mytableextraspace

$\SU(3)$					& $\RS^1\subset T^2$, $T^2$, $\SO(3)$, $\SU(2)$ and $\U(2)$		    \\[.1cm] \mytableextraspace
$\Sp(2)$				&  $\U(2)$, $\Sp(1)\SO(2)$ and $\Sp(1)\Sp(1)$, in dimensions $4$ and higher.						  \\[.1cm] \mytableextraspace
$\mathrm{G}_2$			& $\SU(3)$, in dimensions $8$ and higher.		 \\[.1cm] \mytableextraspace
$\SU(4)$ &	 $\U(3)$ and $\Sp(2)$ in dimensions $9$ and higher.		\\[.1cm] \mytableextraspace
\bottomrule
\end{tabular}
}
\end{center}
\caption{\label{TB:SUBGROUPS} Groups and their subgroups playing a role in the classification}
\end{table}


\subsection{Possible normal spaces of directions}

As stated  in Theorem~\ref{T:ALEX_STRUCTURE}, for a cohomogeneity one action with  group diagram $(G, H, K^-, K^+)$, the homogeneous spaces $K^\pm/H$ are positively curved. The classification of simply-connected positively curved homogeneous spaces has been carried out by Berger \cite{Berger}, Wallach \cite{Wallach}, Aloff and Wallach \cite{AW}, Berard-Bergery \cite{BB} and Wilking \cite{Wi} (for a complete exposition of the classification, correcting some oversights in the literature, see the article by Wilking and Ziller \cite{WZ}). 
Combining this with the classification of homogeneous space forms due to Wolf \cite{Wolf}, and the fact that in even dimensions there can be at most  $\mathbb{Z}_2$ quotients, by Synge's theorem, it follows that the positively curved homogeneous spaces in dimensions $5$ and below are (diffeomorphic to) $\SP^0$, $\SP^1$, $\SP^2$, $\RP^2$, the three-dimensional spherical space forms, $\SP^4$, $\RP^4$, $\CP^2$ (noting that $\CP^2$ admits no  $\mathbb{Z}_2$ quotient) and, in dimension $5$, the five-dimensional spherical space forms. In dimension $6$, there appear $\SP^6$, $\RP^6$, $\CP^3$, $\CP^3/\Z_2$ and, finally,  the Wallach manifold $\W^6=\SU(3)/T^2$ and its $\Z_2$ quotient. We collect this information in Table~\ref{TB:SPACES_DIR}.
 
 
 \begin{table}[!htbp]
\begin{center}
\small{
\begin{tabular}{p{2cm}p{6cm}}	\toprule
Dimension 			& Space  	\\  \mytableextraspace
\midrule
	$0$	 			& $\SP^0$	 	  			\\[.1cm] \mytableextraspace
	$1$				& $\SP^1$ 		 			\\[.1cm] \mytableextraspace
	$2$				& $\SP^2$, $\RP^2$		    	\\[.1cm] \mytableextraspace
	$3$				& $3$-dimensional spherical space forms		\\[.1cm] \mytableextraspace
	$4$				& $\SP^4$, $\RP^4$, $\CP^2$		    				\\[.1cm] \mytableextraspace
	$5$				& $5$-dimensional spherical space forms		\\[.1cm] \mytableextraspace
	$6$				& $\SP^6$, $\RP^6$, $\CP^3$, $\CP^3/\Z_2$, $\W^6$, $\W^6/\Z_2$. 		    \\[.1cm] \mytableextraspace
\bottomrule
\end{tabular}
}
\end{center}
\caption{\label{TB:SPACES_DIR} Positively curved homogeneous spaces in dimensions at most $6$}
\end{table}

Let $X$ be a closed Alexandrov space of cohomogeneity one. If both $\K^\pm/\HH$ are spheres, then $X$ is  equivalent to a smooth manifold. These manifolds and their actions have been classified  by Mostert \cite{ Mostert} and Neumann \cite{Neumann} in dimensions $2$ and $3$, Parker \cite{Parker} in dimension $4$, and Hoelscher \cite{Hoelscher} in dimensions $5$, $6$ and $7$ (assuming $X$ is simply-connected). If both $\K^\pm/\HH$ are integral  homology spheres, then $X$ is equivalent to a topological manifold and $\K^\pm/\HH$ must be either a sphere or the Poincar\'e homology sphere $\PS^3$ (see \cite{GGZ}). These manifolds and their actions have been classified in \cite{GGZ} up to dimension $7$, assuming, as in the manifold case, simply-connectedness in dimensions $5$, $6$ and $7$. From now on we will assume that at least one of the homogeneous spaces $\K^\pm/\HH$ is not a  sphere, i.e.\ that the action is not equivalent to a smooth action on a smooth manifold.

 
\subsection{Classification in dimension $5$}
\label{SS:CLASSIF_D5}
 To find the group diagrams of cohomogeneity one actions on closed, simply-connected Alexandrov spaces in dimension $5$, we first determine the acting groups. By Proposition~\ref{L:DIM_BOUND}, $4\leq \dim \GG\leq 10$. Hence, by Table~\ref{TB:ACTING_GROUPS},  $\GG$ has the form $(\RS^3)^m\times \T^n$,  $\SU(3)\times \T^n$ or $\Spin(5)$. From Proposition \ref{P:ABELIAN_G}, we have $n\leq 1$. Since $\dim \HH=\dim \GG-4$, Proposition  \ref{P:BOUND_DIMH} gives the possible groups. These are, up to a finite cover:  \[\RS^3\times \RS^1, \RS^3\times \RS^3, \SU(3), \text{ or } \Spin(5).\]

Now we examine the action of each group case by case.\\


\noindent $\mathbf{G=S^3\times S^1.}$ In this case, $\dim\HH = 0$, so $\HH_0=\{1\}$. By Proposition \ref{P:ABELIAN_G}, and without loss of generality, we can assume that $\K^-/\HH=\SP^1$. Therefore, $\K^-_0=\{(e^{xp\theta}, e^{iq\theta})\mid  \theta\in \mathbb{R}\}\subseteq \RS^3\times \RS^1$, with $x\in \Imag(\mathbb{H})$, $q\neq 0$ and  $(p, q)=1$. Now we want to determine $K^+/H$. Since we have assumed that the action is non-smoothable, $K^+/H$ is not a sphere. Hence, the possible dimensions for $K^+/H$ are $2$, $3$ or $4$. 
 Since, by Proposition~\ref{P:Trans. of G_0}, $\K^{+}_0$  acts transitively on  $K^+/\HH$ ,  it cannot be $1$-dimensional. Further, by Proposition~\ref{P:ABELIAN_G}, $\K^{+}_0\subseteq S^3\times 1$. Therefore, $\K^{+}_0=S^3\times 1$ and  $K^+/H=\RS^3/\Gamma$, with $\Gamma \neq \{1\}$. Consequently,  by Proposition~\ref{P:Trans. of G_0}, we have that $\HH^+=\K^+_0\cap\HH=\Gamma\times 1$.

Let $p=0$. Then $\HH^-=\K^-_0\cap \HH=1\times \mathbb{Z}_k$. Thus by Lemma \ref{L:SIMP_CONN_COND}, $\HH=\langle \HH^+,\HH^-\rangle=\Gamma\times \mathbb{Z}_k$.  Therefore, by Proposition \ref{P:Trans. of G_0}, $\K^-=\K^-_0\HH=\Gamma\times \RS^1$ and $\K^+=\K^+_0\HH=\RS^3\times \mathbb{Z}_k$, and we  obtain the diagram
\begin{align}
(\RS^3\times \RS^1, \Gamma\times \mathbb{Z}_k, \Gamma\times \RS^1, \RS^3\times \mathbb{Z}_k).
\end{align}
By Proposition \ref{P:Join-action}, $\X$ is equivariantly homeomorphic to $(\RS^3/\Gamma)\ast \SP^1$.

Now let $p\neq 0$. After conjugation, we may assume that 
$\K^-_0=\{(e^{ip\theta}, e^{iq\theta})\mid  \theta\in \mathbb{R}\}$. Since $\Gamma\times 1\subseteq \HH\subseteq \K^-\subseteq \N_{\GG}(\K^-_0)=\RS^1\times \RS^1$, we have that $\Gamma=\mathbb{Z}_m$, for $m\geq 2$. Moreover,  
\[
\HH^-=\HH\cap\K^-_0=\mathbb{Z}_k:=\langle(e^{\frac{2\pi i}{k}p}, e^{\frac{2\pi i}{k}q})\rangle.
\]
Then, by Lemma~\ref{L:SIMP_CONN_COND},
\[
\HH=\langle\HH^-,\HH^+\rangle=\{(e^{\frac{2\pi i}{k}\frac{lk+mps}{m}}, e^{\frac{2\pi i}{k}qs})\mid 1\leq s\leq k, 1\leq l\leq m\}.
\] 
By Proposition \ref{P:Trans. of G_0}, we have then that 
\[
\K^+=\K^+_0\HH=\RS^3\times \mathbb{Z}_{k/(k, q)}
\]
 and 
\[
\K^-=\K^-_0\HH=(\mathbb{Z}_m\times 1)\K^-_0.
\] 
We now look for conditions on the parameters $p, q, m, k$.
By Proposition \ref{fundamental group}, and the long exact sequences of homotopy groups of the fiber bundles 
\[\K^{\pm}/\HH\to \GG/\HH\to \GG/\K^{\pm},\]
\[\K^-\to \GG\to \GG/\K^-,\]
one can see that  $\pi_0(\K^-)=\mathbb{Z}_m/\mathbb{Z}_q$. Thus  $q|m$.  
In addition,  since 
\[
H^-\cap H^+=\{(e^{\frac{2\pi i}{k}\frac{k}{(k,q)}ps} ,1) |\, 1\leq s\leq (k,q)\},
\] we have $(k, q)=q$, i.e. $q|k$. We can also assume that $H\cap (1\times \RS^1)=1$ to have a more effective action. This condition gives, in particular, that $(p, k)=1$. 
 Therefore, the  diagram is given by
\small{
\begin{align}
(\RS^3\times \RS^1, \{(e^{\frac{2\pi i}{k}\frac{lk+mps}{m}}, e^{\frac{2\pi i}{k}qs})\mid 1\leq s\leq k, 1\leq l\leq m\}, (\mathbb{Z}_m\times 1).\K^-_0, \RS^3\times \mathbb{Z}_{k/q}),                                                                                                      
\end{align} 
}
where  $(p, k)=1$ and $q|(m, k)$.
\\


\noindent $\mathbf{\GG=S^3\times S^3.}$ We have  $\dim\HH=2$. Since the  only connected $2$-dimensional subgroup of $G$ is its maximal torus, we have that $\HH_0=\T^2$. Therefore, $\K_0^{\pm}$,  which contains $\T^2$, must be $\RS^3\times \RS^1$ or $\RS^1\times \RS^3$. In particular, $\K^{\pm}/\HH$ is $2$-dimensional. Since at least one of the positively curved homogeneous spaces $K^{\pm}/\HH$ is not a sphere, we may assume, without loss of generality, that $\K^{+}/\HH=\RP^2$. The other homogeneous space $\K^{-}/\HH$ can be $\SP^2$ or $\RP^2$. 

First assume that $\K^{-}/\HH=\SP^2$. Then by Proposition \ref{fundamental group}, $\K^{+}$ is connected. Let  $\K^{+}=\RS^3\times \RS^1$. Recall that $\RS^3$ is, up to a finite cover,  the only Lie group that acts (almost) effectively and transitively on $\RP^2$. Then by Proposition \ref{L:M_TRANS}, $\HH=N_{\RS^3}(\RS^1)\times \RS^1$. Consequently, $\K^{-}$ has to be $N_{\RS^3}(\RS^1)\times \RS^3$ since $\K^{-}$ contains $\HH$ and $\K^{-}/\HH=\SP^2$. Therefore we have the diagram 
\begin{align}
(\RS^3\times \RS^3, N_{\RS^3}(\RS^1)\times \RS^1, N_{\RS^3}(\RS^1)\times \RS^3, \RS^3\times \RS^1),
\end{align}
which corresponds to  a join action. By Proposition \ref{P:Join-action} $\X$ is equivariantly homeomorphic to $\RP^2\ast \SP^2$.

Now let $\K^{-}/\HH=\RP^2$. Assume that $\K_0^{+}=\RS^3\times \RS^1$ and $\K_0^{-}=\RS^3\times \RS^1$. First notice that since  $\T^2\subseteq \K_0^{\pm}$, the circles in the second component of $\K_0^{\pm}$ are the same, so $\K_0^{-}=\K_0^{+}$. Since $\K^{\pm}$ acts transitively on $\RP^2$, so does $\K_0^{\pm}$. Furthermore, by Theorem~\ref{T:ASOH}, $\RS^3\times \RS^1$ does not act almost effectively on $\RP^2$. Thus, by Lemma~\ref{L:M_TRANS}, the second factor acts trivially and  $\HH\cap \K_0^{\pm}=N_{\RS^3}(\RS^1)\times \RS^1$. Since $X$ is simply-connected, by Lemma \ref{L:SIMP_CONN_COND}, $\HH= \langle H^+, H^-\rangle=N_{\RS^3}(\RS^1)\times \RS^1$. Therefore $\K^{\pm}$ are both connected and we obtain the diagram
\begin{align}
(\RS^3\times \RS^3, N_{\RS^3}(\RS^1)\times \RS^1, \RS^3\times \RS^1, \RS^3\times \RS^1). 
\end{align}
This action is non-primitive with $L=\RS^3\times \RS^1$ as in Definition~\ref{Primitive Action}. Hence,   by Proposition \ref{P:Primitive Action}, $X$ is  equivariantly homeomorphic to a $\Susp(\RP^2)$-bundle over $\SP^2$.

Assume now that $\K_0^{+}=\RS^3\times \RS^1$ and $\K_0^{-}=\RS^1\times \RS^3$. Thus $H^+=N_{\RS^3}(\RS^1)\times \RS^1$ and $H^-=\RS^1\times N_{\RS^3}(\RS^1)$. As before, the assumption that   $\X$ is simply-connected implies, by Lemma \ref{L:SIMP_CONN_COND}, that $\HH=\langle H^+, H^-\rangle= N_{\RS^3}(\RS^1)\times  N_{\RS^3}(\RS^1)$. Therefore we get the following diagram:
\begin{align}
(\RS^3\times \RS^3, N_{\RS^3}(\RS^1)\times N_{\RS^3}(\RS^1), N_{\RS^3}(\RS^1)\times \RS^3, \RS^3\times N_{\RS^3}(\RS^1)). 
\end{align}

This action is a join action and $\X$ is equivariantly homeomorphic to $\RP^2\ast \RP^2$.\\


\noindent  $\mathbf{G = SU(3)}$. In this case $\dim H=4$. By Table~\ref{TB:SUBGROUPS}, one can see that the only $4$-dimensional subgroup of $\SU(3)$ is $\U(2)$. Therefore, $H=H_0=\U(2)$, as $\U(2)$ is a maximal subgroup of $\SU(3)$. Since $X$ is simply-connected, the action does not have any exceptional orbits. Hence,  $\K^{\pm}$ must be $\SU(3)$.  Thus the diagram is 
\begin{align}
(\SU(3), \UU(2), \SU(3), \SU(3))
\end{align}
and, by Proposition~\ref{P:Suspension-action}, $X$ is equivalent to $\Susp(\CP^2)$. \\


\noindent  $\mathbf{G = Spin(5)}$. Since $\dim G=10$, by Proposition \ref{P:Spin-action}, the group diagram is 
\begin{align}
(\Spin(5), N_{\Spin(5)}(\Spin(4)), \Spin(5), \Spin(5) ),
\end{align}
and $X$ is equivariantly homeomorphic to $\Susp(\RP^4)$.\\




\subsection{Classification in dimension 6}
\label{SS:CLASSIF_D6}
 Proceeding as in  dimension $5$, we see that $5\leq \dim \GG\leq 15$ and $\dim \HH=\dim \GG-5$. It follows from Propositions~\ref{P:BOUND_DIMH} and \ref{P:ABELIAN_G} that $G$ is one of the following Lie groups:  
\[
 \RS^3\times \RS^3,\ \RS^3\times \RS^3\times \RS^1,\ \SU(3),\ \SU(3)\times \RS^1,\ \Sp(2),\ \Sp(2)\times \RS^1 \text{ or } \Spin(6).
 \] 
If $\GG=\Sp(2)$, then $\dim~\HH=5$. Since $\Sp(2)$ does not have a subgroup of dimension $5$, we can  rule it out. We now carry out the classification for the remaining groups in the list. \\

\noindent\textbf{Notational convention.}
The binary dihedral group $D_{2m}^*$ of order $4m$, $m\geq 3$, is a finite subgroup of $\RS^3$ (see \cite{Wolf}, Section~2.6). Throughout the rest of the paper, we consider it as the following subgroup:
\begin{equation}\label{EQ:BD_gr}
D^*_{2m}=\langle e^{\pi/m i}, j\rangle\subseteq \RS^3.
\end{equation}
If, in the right-hand side of \eqref{EQ:BD_gr}, we assume that $m=1$, then $\langle e^{\pi/m i}, j\rangle=\Z_4$. Therefore, we use the notation $D^*_{2m}$ for $m\geq 3$  (the binary dihedral group as in \cite{Wolf}), and, when $m=1$, $D^*_{2m}$  will correspond to the cyclic subgroup $\langle j \rangle$   of $\RS^3$ generated by $j$.
\\

 
\noindent $\mathbf{\GG = S^3\times S^3}$. In this case the principal isotropy group $\HH$ is $1$-dimensional. Thus $\HH_0=\T^1\subseteq \RS^3\times \RS^3$. After conjugation, we can assume that $\HH_0=\{(e^{ip\theta}, e^{iq\theta}) \mid \theta\in\mathbb{R}\}$ with $(p, q)=1$.  Exploring the subgroups of $\GG$ and the homogeneous spaces with positive curvature, we see that the normal space of directions to the singular orbits has to be a  sphere, a real projective plane or $\SP^3/\Gamma$ with $\Gamma\neq \{1\}$. 
\vspace{.2cm}

First, suppose that $\K^{+}/\HH=\RP^2$. Therefore, $\K^+_0$ is one of the subgroups $\RS^3\times 1$, $1\times \RS^3$ or $\Delta \RS^3$.  Let $\K^+_0=\RS^3\times 1$. Then $q=0$ and $H^+=\HH\cap \K^+_0=\N_{\RS^3}(\RS^1)\times 1$. We now consider the different possibilities for $K^-/H$, namely, $\SP^l$, $l\geq 1$, $\RP^2$, and $\SP^3/\Gamma$ with $\Gamma\neq \{1\}$. 

Let $\K^-/\HH=\SP^l$, $l\geq 2$. Then $\K^+$ is connected and $\HH=\N_{\RS^3}(\RS^1)\times 1$. Further, the only subgroup $\K^-$ of $\GG$ containing $\HH$ and satisfying $\K^-/\HH=\SP^l$ is $\N_{\RS^3}(\RS^1)\times \RS^3$.  Hence we have the following diagram:
\begin{align}
(\RS^3\times \RS^3, \N_{\RS^3}(\RS^1)\times 1, \N_{\RS^3}(\RS^1)\times \RS^3, \RS^3\times 1).
\end{align}
By Proposition \ref{P:Join-action},  $\X$ is equivariantly homeomorphic to $\RP^2\ast\SP^3$. 

Let $\K^-/\HH=\SP^1$. Then $\K^-_0=\T^2$. Since $\RS^1\times 1\subseteq \HH\cap \T^2\subseteq \RS^1\times \RS^1$, and $\HH^-=\HH\cap \T^2$ is a finite extension of $\RS^1\times 1$, we have $\HH\cap \T^2=\RS^1\times \mathbb{Z}_k$. Consequently $\HH=\langle H^-,H^+\rangle=\N_{\RS^3}(\RS^1)\times \mathbb{Z}_k$. Thus we obtain the diagram
\begin{align}
(\RS^3\times \RS^3, \N_{\RS^3}(\RS^1)\times \mathbb{Z}_k, \N_{\RS^3}(\RS^1)\times \RS^1, \RS^3\times \mathbb{Z}_k).
\end{align}
 This action is non-primitive with $\LL=\RS^3\times \RS^1$. Therefore, by Proposition~\ref{P:Primitive Action}, $\X$ is equivariantly homeomorphic to the total space of an $(\RP^2\ast\SP^1)$-bundle over $\SP^2$. 

Let $\K^-/\HH=\RP^2$. Hence, $\K^-$ is a $3$-dimensional subgroup of $\GG$, namely $\RS^3\times 1$, $1\times \RS^3$, or $\Delta \RS^3$. However, since $\RS^1\times 1=\HH_0\subseteq \K^-_0$, the group $\K^-_0$ must be $\RS^3\times 1$, and  $H^-=\K^-_0\cap \HH=\N_{\RS^3}(\RS^1)\times 1$. Therefore $\HH=\langle H^-,H^+\rangle=\N_{\RS^3}(\RS^1)\times 1$, and we get the following diagram:
\begin{align}
(\RS^3\times \RS^3, \N_{\RS^3}(\RS^1)\times 1, \RS^3\times 1, \RS^3\times 1).
\end{align}
This action is equivalent to the following action on $\Susp(\RP^2)\times \SP^3$:
\begin{align*}
(\RS^3\times\RS^3)\times(\Susp(\RP^2)\times\SP^3)&\to (\Susp(\RP^2)\times\SP^3)\\
((g, h), ([x,t], y))&\mapsto ([gxg^{-1}, t],  hy).
\end{align*}

Let $\K^-/\HH=\RS^3/\Gamma$ with $\Gamma\neq\{1\}$.  Therefore, $\K^-_0=\RS^3\times  \RS^1$, or $\K^-_0=\RS^1\times  \RS^3$. Assume $\Gamma\neq \mathbb{Z}_k$.  In this case, since $\RS^3$ is, up to a finite cover, the only Lie group which acts transitively and almost effectively  on $\RS^3/\Gamma$,  by Lemma \ref{L:M_TRANS}, $\HH\cap\K^-_0$ is $\Gamma\times \RS^1$ and $\RS^1\times\Gamma$, respectively. As $\HH\cap\K^-_0\subseteq \K^+=\RS^3\times \Gamma_1$, where  $\Gamma_1$ is a finite subgroup of $\RS^3$, we must have $\HH\cap\K^-_0=\RS^1\times\Gamma$.  The diagram is then given by
 \begin{align}
(\RS^3\times \RS^3, \N_{\RS^3}(\RS^1)\times \Gamma, \N_{\RS^3}(\RS^1)\times \RS^3, \RS^3\times \Gamma).
\end{align}
By Proposition \ref{P:Join-action}, $\X$ is equivariantly homeomorphic to $\RP^2\ast(\RS^3/\Gamma)$. 

Now let $\Gamma=\mathbb{Z}_k$. According to Theorem~\ref{T:ASOH}, $\RS^1\times \RS^3$ acts on $\RS^3/\mathbb{Z}_k$ in the following way:
\begin{align*}
(\RS^1\times\RS^3)\times\RS^3/\mathbb{Z}_k&\to \RS^3/\mathbb{Z}_k\\
((z, \nu), [x])&\mapsto [\nu x \bar{z}^p].
\end{align*} 
Thus $\HH\cap \K^-_0=\{(z, \lambda z^p)\mid z\in  \RS^1, \lambda\in \mathbb{Z}_k\}$. However, $\HH\cap\K^-_0$ is a subset of $\K^+=\RS^3\times \Gamma_1$, which yields $p=0$. Therefore we have the following diagram:
\begin{align}
(\RS^3\times \RS^3, \N_{\RS^3}(\RS^1)\times \mathbb{Z}_k, \N_{\RS^3}(\RS^1)\times \RS^3, \RS^3\times \mathbb{Z}_k).
\end{align}
By Proposition~\ref{P:Join-action}, $\X$  is equivariantly homeomorphic to $\RP^2\ast(\RS^3/ \mathbb{Z}_k)$. For $\K^-_0=\RS^3\times\RS^1$, we have $\HH\cap \K^-_0=\{( \lambda z^p,z)\mid z\in  \RS^1, \lambda\in \mathbb{Z}_k\}$, which is not a subset of  $\K^+=\RS^3\times \Gamma_1$. Therefore, this case does not occur. 
\vspace{.2cm}

We now repeat the above procedure for $\K^+_0=\Delta \RS^3$. In this case 
$\HH_0=\Delta \RS^1$ and 
\[
H^+=\HH\cap \K^+_0=\Delta \RS^1\cup (j, j)\Delta \RS^1.
\] 
We consider the different possibilities for $K^-/H$, namely, $\SP^l$, $l\geq 1$, $\RP^2$, and $\SP^3/\Gamma$ with $\Gamma\neq \{1\}$.

 If $\K^-/\HH=\SP^l$, $l\geq 1$, then, as before,  $l=1, 3$ only. First, suppose that $\K^-/\HH=\SP^3$. Therefore, $\K^+$ is connected and $\HH=\Delta \RS^1\cup (j, j)\Delta \RS^1$. Since $\K^-_0$ is $4$-dimensional,  after exchanging the factors of $G$ if necessary, we can assume that $\K^-_0=\RS^3\times\RS^1$. Hence $\K^-=\K^-_0\HH=\RS^3\times\N_{\RS^3}(\RS^1)$, and the following diagram is obtained 
\begin{align}
(\RS^3\times \RS^3, \Delta\RS^1\cup(j, j)\Delta\RS^1, \RS^3\times\N_{\RS^3}(\RS^1), \Delta\RS^3).
\end{align}

This action is equivalent to  the following action:
\begin{align*}
(\RS^3\times\RS^3)\times(\RS^3\ast \RP^2)&\to (\RS^3\ast \RP^2)\\
((g, h), [x, [y]])&\mapsto [gxh^{-1}, [hyh^{-1}]].
\end{align*} 
That is, $\X$  is equivariantly homeomorphic to $\RS^3\ast \RP^2$. 

Now let $\K^-/\HH=\SP^1$. Then $\K^-_0=\T^2$.  Since 
\begin{align*}
\HH\subseteq \N(\K^-)\cap\N(K^+) 	& =\pm\Delta\RS^3\cap(\N(\RS^1)\times\N(\RS^1))\\
							&=\pm\Delta\RS^1\cup(j,\pm j)\Delta\RS^1,
\end{align*}
 where $\pm\Delta\RS^1=\{(g, g)\}\cup\{(g, -g)\}$, we have two cases: $\HH=\Delta\RS^1\cup(j, j)\Delta\RS^1$ or $H=\pm\Delta\RS^1\cup(j,\pm j)\Delta\RS^1$. Thus  we have the  following  diagrams:
\begin{align}
(\RS^3\times \RS^3, \Delta\RS^1\cup (j, j)\Delta\RS^1, \T^2\cup (j, j)\T^2, \Delta\RS^3)
\end{align}
 and 
\begin{align}
(\RS^3\times \RS^3, \pm\Delta\RS^1\cup (j,\pm j)\Delta\RS^1, \T^2\cup (j, j)\T^2, \pm\Delta\RS^3).
\end{align}

Now assume that $\K^-/\HH=\RP^2$. Thus $K^-$ is a $3$-dimensional subspace containing $\Delta \RS^1\cup (j, j)\Delta \RS^1$, which gives in particular that $\K^-_0$ must be $\Delta \RS^3$, and $H^-=H\cap \K^-_0=\Delta \RS^1\cup (j, j)\Delta \RS^1$. Thus 
$H=\langle H^-,H^+\rangle=\Delta \RS^1\cup (j, j)\Delta \RS^1$ 
and the following diagram is obtained:
\begin{align}\label{Eq:Dim_6_m}
(\RS^3\times \RS^3, \Delta\RS^1\cup(j, j)\Delta\RS^1, \Delta\RS^3, \Delta\RS^3).
\end{align}

 Note that this action  is equivalent to the following action:
 \begin{align*}
(\RS^3\times\RS^3)\times(\Susp(\RP^2)\times\SP^3)&\to (\Susp(\RP^2)\times\SP^3)\\
((g, h), ([x,t], y))&\mapsto ([gxg^{-1}, t],  hyg^{-1}).
\end{align*}
   Thus $X$ is equivariantly homeomorphic to $\Susp(\RP^2)\times\SP^3$. 
   
Now, let $\K^-/\HH=\RS^3/\Gamma$ with $\Gamma\neq\{1\}$. Then $\K^-_0=\RS^3\times \RS^1$ or $\K^-_0=\RS^1\times \RS^3$. After exchanging the factors of $G$, if necessary, we can assume that $\K^-_0=\RS^3\times \RS^1$. If $\Gamma\neq \mathbb{Z}_k$, then $\HH\cap\K^-_0=\Gamma\times\RS^1$. Since $\Delta\RS^1=\HH_0\subseteq \HH\cap\K^-_0$, this cannot happen. Therefore $\Gamma= \mathbb{Z}_k$, and the action of $\RS^3\times \RS^1$ on $\RS^3/\mathbb{Z}_k$ is given by:
\begin{align*}
(\RS^3\times\RS^1)\times\RS^3/\mathbb{Z}_k&\to \RS^3/\mathbb{Z}_k\\
((\nu, z), [x])&\mapsto [\nu x \bar{z}^p].
\end{align*}

Thus $\HH\cap\K^-_0=\{(\lambda z^p, z)\mid z\in  \RS^1, \lambda\in \mathbb{Z}_k\}$. As $\Delta\RS^1\subseteq \HH\cap\K^-_0$, we have that $p=1$. Further, $\HH\cap\K^-_0\subseteq \N_{\RS^3\times \RS^3}(\Delta \RS^3)=\pm \Delta\RS^3$, which implies $\mathbb{Z}_k=\mathbb{Z}_2$. Therefore $\HH=\pm\Delta\RS^1\cup(j,\pm j)\Delta\RS^1$, and the following diagram is obtained
\begin{align}
(\RS^3\times \RS^3, \pm\Delta\RS^1\cup(j,\pm j)\Delta\RS^1, \RS^3\times\N_{\RS^3}(\RS^1), \pm\Delta\RS^3).
\end{align}
This action is equivalent to the action given by 
\begin{align*}
(\RS^3\times\RS^3)\times(\RP^2*\RP^3)&\to \RP^2*\RP^3\\
((g, h), [x, y,t])&\mapsto [gxg^{-1},  hy,t].
\end{align*} 
Thus $\X$ is equivariantly homeomorphic to $\RP^2\ast \RP^3$.
\vspace{.2cm}

Now assume that $\K^+/\HH=\RS^3/\Gamma$ with $\Gamma\neq\{1\}$. Thus $\dim~K^+=4$. Since the connected $4$-dimensional subgroups of $\RS^3\times \RS^3$ are $\RS^3\times \RS^1$ and
$\RS^1\times \RS^3$, we can assume, without  loss of generality, that $\K^+_0=\RS^3\times \RS^1$.
 The possibilities for $\K^-/\HH$ are $\SP^l$, $l\geq 1$, $\RP^2$, and $\RS^3/\Lambda$, where $\Lambda$ is a non-trivial finite subgroup of $\RS^3$. The case where $\K^-/\HH=\RP^2$ has been treated above, so we only examine the cases where $\K^-/\HH$ is a sphere or a $3$-dimensional spherical space form. 

First assume that $\K^-/\HH=\SP^1$. Then $\K^-_0=\T^2$. Since $\HH_0=\{(e^{pi\theta}, e^{qi\theta})\}\subseteq \T^2$, and $\HH_0=\{(e^{pi\theta}, e^{qi\theta})\}\subseteq \RS^3\times \RS^1$, the circle in the second component of $\K^-_0$ and $\K^+_0$ are the same, that is $\RS^1=\{e^{i\theta}\}$. This implies  that $\K^-_0\subseteq \K^+_0$. Therefore, $\HH=\langle H^-, H^+\rangle=H^+=\HH\cap\K^+_0$. Let $\Gamma\neq \mathbb{Z}_k\subseteq \{e^{i\theta}\}\subseteq \RS^3$. Then by Lemma  \ref{L:M_TRANS}, $\HH\cap\K^+_0=\Gamma\times\RS^1$, since by Theorem~\ref{T:ASOH}, $\RS^3$ is the only compact connected Lie group which acts almost effectively on $\RS^3/\Gamma$. Further, $\HH\subseteq \K^-\subseteq \N_{\RS^3\times \RS^3}(\T^2)=\N_{\RS^3}(\RS^1)\times \N_{\RS^3}(\RS^1)$. Among the finite subgroups of $\RS^3$, only $\mathbb{Z}_k\subseteq \{e^{i\theta}\}$ 
and $\D_{2m}^*$  are contained in $\N_{\RS^3}(\RS^1)=\RS^1\cup j\RS^1$. Thus 
 $\Gamma=\D_{2m}^*$,  which implies that  $\K^-=\T^2\cup (j, 1)\T^2$. Therefore, we have the following diagram:
\begin{align}
(\RS^3\times \RS^3, \D^*_{2m}\times\RS^1, \N_{\RS^3}(\RS^1)\times \RS^1, \RS^3\times \RS^1).
\end{align}


Now, suppose $\Gamma=\mathbb{Z}_k \subseteq \{e^{i\theta}\}$.  The transitive action of $\RS^3\times \RS^1$  on $\RS^3/\mathbb{Z}_k$ gives  
\[
\HH^+=\K^+_0\cap\HH=\{(e^{ip\theta}\lambda, e^{i\theta})\mid \theta\in \mathbb{R}, \lambda\in\mathbb{Z}_k\}.
\]
Thus we have the following diagram:
\begin{align}
(\RS^3\times \RS^3, \{(e^{ip\theta}\lambda, e^{i\theta})\mid \theta\in \mathbb{R}, \lambda\in\mathbb{Z}_k\}, \T^2, \RS^3\times \RS^1).
\end{align}

Assume now that $\K^-/\HH=\SP^l$,  $l\geq 2$. Hence by Corollary~\ref{C:Simply-connected Orbits}, $\K^+$ is connected.  First assume that $\Gamma\neq \mathbb{Z}_k$. As a result $\HH=\Gamma\times \RS^1$. For $l=2$, the only possibility for $\K^-_0$ is $1\times \RS^3$. Then we obtain the following diagram:
 \begin{align}
(\RS^3\times \RS^3, \Gamma\times \RS^1, \Gamma\times \RS^3, \RS^3\times \RS^1).
\end{align}
This action is equivalent to the join action on $(\RS^3/\Gamma)\ast\RS^2$. 

For $l\geq3$, there are no subgroups of $\GG$ such that $\K^-/\HH=\SP^l$; therefore we need not consider these cases.

Now let $\Gamma=\mathbb{Z}_k$. Then the isotropy subgroup of the transitive action of $\RS^3\times\RS^1$ on $\RS^3/\mathbb{Z}_k$ would be $\{(e^{ip\theta}\lambda, e^{i\theta})\mid \theta\in \mathbb{R}, \lambda\in\mathbb{Z}_k\}$. Assume that $l=2$. Then $\K^-_0=1\times \RS^3$ or $\K^-_0=\Delta\RS^3$. Therefore, $p=0$, or $p=1$, respectively. If $p=1$, then $\mathbb{Z}_k=\mathbb{Z}_2$ since $\HH\subseteq \N(\Delta\RS^3)=\pm\Delta\RS^3$. Thus we have the  following diagrams corresponding to $p=0$ and $p=1$, respectively:
 \begin{align}
(\RS^3\times \RS^3, \mathbb{Z}_k\times \RS^1, \mathbb{Z}_k\times \RS^3, \RS^3\times \RS^1),
\end{align}
and 
\begin{align}
(\RS^3\times \RS^3, \pm\Delta\RS^1, \pm\Delta\RS^3, \RS^3\times \RS^1).
\end{align}

The first action is equivalent to the join action on $(\RS^3/\mathbb{Z}_k)\ast\SP^2$. The second one is the join action on $\RP^3\ast\RS^2$  given by
\begin{align*}
(\RS^3\times\RS^3)\times(\RP^3*\SP^2)&\to \RP^3*\SP^2\\
(g, h)\cdot[x, y,t]&= [gxh^{-1},  hyh^{-1},t].
\end{align*}

For $l\geq 3$, there are no subgroups of $\GG$ such that $\K^-/\HH=\SP^l$.

Assume now that $\K^-/\HH=\RS^3/\Lambda$ with $\Lambda$ a non-trivial subgroup of $S^3$. Therefore, $\K^-_0=\RS^3\times \RS^1$ or $\K^-_0=\RS^1\times \RS^3$. First assume that $\K^-_0=\RS^3\times \RS^1$. Note that according to the classification of the transitive actions on $3$-dimensional space forms, $q\neq 0$, which gives that the circles in the second component of $K^{\pm}_0$ are the same, so $\K^-_0=\K^+_0$, and $\HH=\K^+_0\cap \HH=\K^-_0\cap \HH$. Thus $\Gamma=\Lambda$. Consequently, for $\Gamma\neq \mathbb{Z}_k$, we have the following diagram:
\begin{align}
(\RS^3\times \RS^3, \Gamma\times \RS^1, \RS^3\times \RS^1, \RS^3\times \RS^1).
\end{align}
This action is equivalent to the product action on $\Susp(\RS^3/\Gamma)\times\SP^2$. If $\Gamma=\mathbb{Z}_k$, the following diagram is obtained:
\begin{align}
(\RS^3\times \RS^3, \{((e^{ip\theta}\lambda, e^{i\theta})\mid \theta\in \mathbb{R}, \lambda\in\mathbb{Z}_k)\}, \RS^3\times \RS^1, \RS^3\times \RS^1).
\end{align}
For $p=0$, this action is equivalent to the product action on $\Susp(\RS^3/\mathbb{Z}_k)\times\SP^2$, and for $p\neq 0$, it is non-primitive. In particular, in the preceding diagram,
 if $\mathbb{Z}_k=\mathbb{Z}_2$ and $p=1$, then the action is as follows:
\begin{align*}
(\RS^3\times\RS^3)\times(\Susp(\RP^3)\times\SP^2)&\to\Susp(\RP^3)\times\SP^2\\
(g, h)\cdot([x, t], y)&= ([gxh^{-1}, t],  hyh^{-1}).
\end{align*}

Now let $\K^-_0=\RS^1\times \RS^3$. Hence $\HH_0=\Delta \RS^1$, and both $\Gamma$ and $\Lambda$ are  cyclic subgroups of $\RS^3$, say $\Z_k$ and $Z_l$, respectively. Then we have $H^+=\{((e^{i\theta}\lambda, e^{i\theta})\mid \theta\in \mathbb{R}, \lambda\in\mathbb{Z}_k)\}$ and $H^-=\{((e^{i\theta}\lambda, e^{i\theta})\mid \theta\in \mathbb{R}, \lambda\in\mathbb{Z}_l)\}$. Hence  $H^+$ and  $H^-$ are subgroups of both $K_0^+$ and $K^-_0$, which gives that $H=\langle H^+, H^-\rangle\subseteq K^{\pm}_0$. It follows then from Proposition~\ref{P:Trans. of G_0} that $K^+=K^+_0H=K^+_0$ and $K^-=K^-_0H=K^-_0$. Thus    $H^-=H=H^+$ and, in particular, $\Gamma=\Lambda$. The diagram is then given by
\begin{align}
(\RS^3\times \RS^3, \{((e^{i\theta}\lambda, e^{i\theta})\mid \theta\in \mathbb{R}, \lambda\in\mathbb{Z}_k)\}, \RS^1\times \RS^3, \RS^3\times \RS^1).
\end{align}
\vspace{.15cm}


\noindent  $\mathbf{\GG = S^3\times S^3\times S^1}.$ In this case, $\dim~\HH=2$ and $\HH_0\subseteq \RS^3\times \RS^3\times 1$, since the action is non-reducible. As the maximal torus of $\RS^3\times \RS^3$ is the only $2$-dimensional subgroup of $\RS^3\times \RS^3$, we have  $\HH_0=\T^2$. Further, by Proposition \ref{P:ABELIAN_G}, $\K^-/\HH=\SP^1$, and $\K^+_0\subseteq \RS^3\times \RS^3\times 1$. As a result, $\K^-_0=\T^3$. Since $\T^2=\HH_0\subseteq \K^+_0\subseteq \RS^3\times \RS^3\times 1$, we have $\K^+_0=\RS^3\times \RS^1$, $\K^+_0=\RS^1\times \RS^3$, or $\K^+_0=\RS^3\times \RS^3$. However, $\RS^3\times \RS^3$ does not act transitively on a $4$-dimensional  homogeneous space with positive curvature (see \cite{WZ}). Therefore, $\K^+_0=\RS^3\times \RS^1$ or $\K^+_0=\RS^1\times \RS^3$. Without loss of generality, we can assume that $\K^+_0=\RS^3\times \RS^1$. Thus $\K^+/\HH=\RP^2$. By the classification of the transitive actions on spheres, $\RS^3$, up to a finite cover, is the only Lie group which acts transitively and almost effectively on $\RP^2$. Therefore, $\K^+_0\cap\HH=\N_{\RS^3}(\RS^1)\times\RS^1\times 1$, and we obtain the following diagram:
\begin{align}
(\RS^3\times \RS^3\times \RS^1, \N_{\RS^3}(\RS^1)\times \RS^1\times \mathbb{Z}_k , \N_{\RS^3}(\RS^1)\times \RS^1\times \RS^1, \RS^3\times \RS^1\times \mathbb{Z}_k).
\end{align} 
By Proposition~\ref{P:Product-action}, this action is equivalent to the product  action on   $(\RP^2\ast\SP^1)\times \SP^2$.
\\


\noindent $\mathbf{\GG=SU(3)}.$ In this case, $\dim\HH=3$. Thus the only possibilities for $H_0$ are $\SO(3)$ and $\SU(2)$. If $\HH_0=\SO(3)$, then $\K^{\pm}=\SU(3)$ since $\SO(3)$ is a maximal connected subgroup and there are no exceptional orbits. This cannot happen since there are no  homogeneous spaces with positive curvature with an $\SU(3)$-action and $\SO(3)$ as the isotropy group (see \cite{WZ}). Hence $\HH_0=\SU(2)$. The  subgroups of $\GG$ which contain $\HH_0$ properly are $\UU(2)$ and $\SU(3)$. As $\UU(2)/\HH=\SP^1$, at least one of the singular isotropy groups, say $\K^+$, is equal to $\SU(3)$. Therefore, $\dim~\K^+/\HH=5$, which gives that $\K^+/\HH=\RS^5/\mathbb{Z}_k$. The classification of the transitive actions on spheres then shows that $\HH=\mathrm{S}(\UU(2)\mathbb{Z}_k)$. Depending on whether  $\K^-=\UU(2)$ or $\SU(3)$, we have the following two diagrams:
\begin{align}
(\SU(3), \mathrm{S}(\UU(2)\mathbb{Z}_k) , \UU(2), \SU(3)),
\end{align}
\begin{equation}\label{EQ: SU(3)-action, 2}
(\SU(3), \mathrm{S}(\UU(2)\mathbb{Z}_k) , \SU(3), \SU(3)).
\end{equation}
The space determined by diagram \eqref{EQ: SU(3)-action, 2} is equivalent to  $\Susp(\SP^5/\mathbb{Z}_k)$.
\\


\noindent $\mathbf{\GG=SU(3)\times S^1}.$ In this case, $\dim\HH=4$. By Proposition \ref{P:ABELIAN_G}, $\HH_0, \K^+_0\subseteq \SU(3)\times 1$ and $\K^-/\HH=\SP^1$. Therefore $\HH_0=\UU(2)\times 1$, $\K^+_0=\SU(3)\times 1$, and $\K^-_0=\UU(2)\times \RS^1$. Hence the following diagram is obtained:
\begin{align}
(\SU(3)\times\RS^1, \UU(2)\times\mathbb{Z}_k, \UU(2)\times \RS^1, \SU(3)\times\mathbb{Z}_k).
\end{align}
By Proposition \ref{P:Join-action}, the space determined by this diagram is equivalent to $\CP^2\ast\SP^1$. 
\\


\noindent $\mathbf{\GG=Sp(2)\times S^1}.$ In this case, $\dim\HH=6$. As above, by Proposition \ref{P:ABELIAN_G}, $\HH_0, \K^+_0\subseteq \Sp(2)\times 1$, and $\K^-/\HH=\SP^1$. Therefore $\HH_0=\Sp(1)\Sp(1)\times 1$, and  $\K^+_0=\Sp(2)\times 1$, for $\Sp(1)\Sp(1)$ is a  maximal connected subgroup of $\Sp(2)$. Thus $\dim~\K^+/\HH=4$, and therefore $\K^+/\HH=\RP^4$ (note that the other positively curved homogeneous space in dimension $4$ is $\CP^2$, which does not admit an $\Sp(2)$-transitive action (see \cite[Table~B]{WZ})). Hence we get the following diagram
\begin{align}
(\Sp(2)\times\RS^1, \N_{\Sp(2)}(\Sp(1)\Sp(1)))\times\mathbb{Z}_k, \N_{\Sp(2)}(\Sp(1)\Sp(1))) \times\RS^1, \Sp(2)\times\mathbb{Z}_k).
\end{align}
By Proposition \ref{P:Join-action}, $\X$ is equivalent to  $\RP^4\ast\SP^1$. 
\\


\noindent $\mathbf{\GG=Spin(6)}.$ In this case, since  $\dim~\GG=15=(6)(6-1)/2$, by Proposition~\ref{P:Spin-action}, we obtain the  diagram
\begin{align}
(\Spin(6), \N_{\Spin(6)}(\Spin(5)), \Spin(6), \Spin(6)).
\end{align}
and $\X$ is equivariantly  homeomorphic to $\Susp(\RP^5)$.\\




\subsection{Classification in dimension 7} 
\label{SS:CLASSIF_D7}
By Proposition \ref{L:DIM_BOUND}, we have $6\leq \dim \GG\leq 21$ and $\dim \HH=\dim \GG-6$. As before, Propositions  \ref{P:BOUND_DIMH} and \ref{P:ABELIAN_G} give us the possible acting groups:  
\begin{align*}
&\RS^3\times \RS^3,\ \RS^3\times \RS^3\times \RS^1,\ \SU(3),\ \RS^3\times \RS^3\times \RS^3,\ \SU(3)\times \RS^1,\  \Sp(2),\\ & \SU(3)\times \RS^3,\ \Sp(2)\times \RS^3,\ \GG_2,\ \SU(4),\ \SU(4)\times \RS^1,\ \Spin(7).
\end{align*} 
Now we examine each group case by case.
\\


\noindent $\mathbf{G=S^3\times S^3}.$ In this case, $\dim\HH=0$. Having looked at the classification of homogeneous spaces with positive curvature, and the subgroups of $\RS^3\times \RS^3$, one can see that the only  homogeneous spaces  with positive curvature  that can happen as the normal space of directions of singular orbits  are $3$-dimensional spherical space forms. 

Assume that $\K^+/\HH= \RS^3/\Gamma$, with $\Gamma$ a nontrivial finite subgroup of $\RS^3$. Then $\K^+$ is $3$-dimensional and, as a result, $\K^+_0$ can be $\RS^3\times 1$, $1\times \RS^3$ or  $\Delta_{g_0}\RS^3=\{(g, g_0gg_0^{-1})\mid g\in \RS^3\}$, for some fixed $g_0\in \RS^3$. 
\vspace{.2cm}

Suppose first  that $\K^+_0=\RS^3\times 1$. Then $\HH\cap\K^+_0=\Gamma\times 1$. Furthermore, $\K^-/\HH$ is one of the spaces $\SP^1$, $\SP^3$, or $\RS^3/\Lambda$ with $\Lambda$ a nontrivial finite subgroup of $\RS^3$. 

First assume that $\K^-/\HH=\SP^1$. Thus $\K^-_0=\{(e^{xp\theta}, e^{yq\theta})\mid  \theta\in \mathbb{R}, x, y\in \Imag (\mathbb{H})\cap \RS^3\}$. If $p=0$, then we have the diagram
 \begin{align}
(\RS^3\times \RS^3, \Gamma\times\mathbb{Z}_k, \Gamma\times \RS^1, \RS^3\times \mathbb{Z}_k), 
\end{align}
and $\X$ is equivariantly  homeomorphic to the total space of an $(\RS^3/(\Gamma\ast\SP^1))$-bundle over $\SP^2$. 

If $q=0$, then 
\[
\Gamma\times 1\subseteq \HH\subseteq N_{\RS^3\times\RS^3}(\RS^1\times 1)= N_{\RS^3}(\RS^1)\times \RS^3,
\]
 which implies that $\Gamma=\Z_k$ or $\Gamma= \D^{*}_{2m}$. For $\Gamma=\Z_k$, $H^+\subseteq K^-_0$. Therefore, $H=\langle H^+, H^-\rangle\subseteq K^-_0$, which gives, by Proposition~\ref{P:Trans. of G_0}, that $K^-=K^-_0$. Thus we get the following diagram:
\begin{align}
(\RS^3\times \RS^3, \mathbb{Z}_k\times 1, \RS^1\times 1, \RS^3\times 1).                                                                                                       
\end{align}
This action is equivalent to the product action of  $\RS^3\times \RS^3$ on $\X^4\times \SP^3$, where $\X^4$ is the $4$-dimensional Alexandrov space with the following diagram (see \cite{GS}):
 \[
(\RS^3, \mathbb{Z}_k, \RS^1, \RS^3).                                                                                                       
\]
Indeed, $X^4$ is equivariantly homeomorphic to $\CP^2/\Z_k$ (for more details see Subsection~\ref{SS: fixed point orbifolds}, Diagram \eqref{D:ORBI_ITEM_1}). 

For $\Gamma= \D^{*}_{2m}$, we have $\K^-=N_{\RS^3}(\RS^1)\times 1$, and we obtain the following diagram:
\begin{align}
(\RS^3\times \RS^3, \D^{*}_{2m}\times 1, N_{\RS^3}(\RS^1)\times 1, \RS^3\times 1).                                                                                                       
\end{align}
Similarly, this action is equivalent to the product action of  $\RS^3\times \RS^3$ on $\X^4\times \SP^3$, where $\X^4$ is given by 
\[
(\RS^3, \D^{*}_{2m}, N_{\RS^3}(\RS^1), \RS^3).                                                                                                       
\] 
Again, $X^4$ is equivariantly homeomorphic to $\CP^2/\Z_m$ (for more details see Subsection~\ref{SS: fixed point orbifolds}, Diagram \eqref{D:ORBI_ITEM_2}). 

If $pq\neq 0$, then,  since \[\Gamma\times 1\subseteq \HH\subseteq N_{\RS^3\times\RS^3}(\{(e^{xp\theta}, e^{yq\theta})\})=\{(e^{x\theta}, e^{y\phi})\}\cup (z, w)\{(e^{x\theta}, e^{y\phi})\},\]
where $z\in x^{\perp}\cap \Imag (\mathbb{H})\cap \RS^3$ and $w\in y^{\perp}\cap \Imag (\mathbb{H})\cap \RS^3$,
 we have $\Gamma= \mathbb{Z}_k$. Also, without loss of generality, we may assume that $\K^-_0=\{(e^{ip\theta}, e^{iq\theta})\}$. Therefore, we get the following diagram:
\begin{align}
\label{EQ:7, diag. 1}
(\RS^3\times \RS^3, \{(e^{\frac{lk+mps}{km}2\pi i}, e^{\frac{2\pi psi}{k}})\mid 1\leq s\leq k, 1\leq l\leq m\},  (\mathbb{Z}_m\times 1)\K^-_0, \RS^3\times \mathbb{Z}_{k/(k,q)}),
\end{align}
where $(k, q)=(q, m)$.

Now, assume that $\K^-/\HH=\SP^3$. As a result, $\K^+$ is connected and $\HH=\Gamma\times 1$. On the other hand, $\K^-_0$ is a $3$-dimensional subgroup containing $\Gamma\times 1$. Therefore, there are two possibilities:  $\K^-_0=1\times\RS^3$, and  $\K^-_0=\Delta_{g_0}\RS^3$. If $\K^-_0=1\times\RS^3$, then we obtain the following  diagram:
 \begin{align}
(\RS^3\times \RS^3, \Gamma\times 1, \Gamma\times \RS^3, \RS^3\times 1).                                                                                                       
\end{align}
By Proposition~\ref{P:Join-action}, $\X$ is equivariantly  homeomorphic to $\RS^3\ast (\RS^3/\Gamma)$.  

Now let  $K^-_0=\Delta_{g_0}\RS^3$. Since  $1\times\RS^3\subseteq \N(\HH)_0$,  by Proposition~\ref{P:Equivalent_actions} we can conjugate $\K^-$ by $(1, g_0^{-1})$ without changing the spaces. Moreover, $K^-\subseteq N(\Delta_{g_0}\RS^3)=\pm \Delta_{g_0}\RS^3$, so we can assume that $g_0=1$. Now, since $\K^-/\HH$ is simply-connected, the number of connected components of $\K^-$ and $\HH$ are the same. Since $\HH\neq 1$, and $\K^-$ has at most two components, we conclude that $\Gamma=\mathbb{Z}_2$. Thus, we get the following diagram:
\begin{align}
(\RS^3\times \RS^3,\mathbb{Z}_2\times 1, \pm \Delta\RS^3, \RS^3\times 1).                                                                                                       
\end{align}
This action is equivalent to the  following  action on $\RP^3\ast\SP^3$:
 \begin{align*}
(\RS^3\times\RS^3)\times(\RP^3\ast\SP^3)&\to \RP^3\ast\SP^3\\
(g, h)\cdot[x, y,t]& = [gxh^{-1},  hy,t].
\end{align*} 

If $\K^-/\HH=\RS^3/\Lambda$, then $K^-_0$ is equal to one of the subgroups  $\RS^3\times\ 1$,    $1\times \RS^3$ or $\Delta_{g_0}\RS^3$. First assume that $\K^-_0=\RS^3\times\ 1$. Then 
\[
\Gamma\times 1=\K^+_0\cap\HH=\K^-_0\cap\HH=\Lambda\times 1.
\]
Therefore, $\Gamma=\Lambda$ and by Lemma~\ref{L:SIMP_CONN_COND}, $H=\langle H^+, H^-\rangle=\Gamma\times 1$ and we obtain the diagram
 \begin{align}
(\RS^3\times \RS^3, \Gamma\times 1, \RS^3\times 1, \RS^3\times 1).                                                                                                       
\end{align}
This action is equivalent to the product action on $\Susp(\RS^3/\Gamma)\times \RS^3$. 

Now let $\K^-_0=1\times\ \RS^3$. In this case, $\Gamma\times 1=\K^+_0\cap\HH$, and $\K^-_0\cap\HH=\Lambda\times 1$, so by Lemma~\ref{L:SIMP_CONN_COND}, $\HH=\Gamma\times \Lambda$. Hence, we get the following diagram:
 \begin{align}
(\RS^3\times \RS^3, \Gamma\times \Lambda, \Gamma\times \RS^3, \RS^3\times \Lambda).                                                                                                       
\end{align}
Proposition~\ref{P:Join-action} implies that $\X$ is  equivariantly  homeomorphic to $(\RS^3/\Gamma)\ast(\RS^3/\Lambda)$.

  Finally, suppose that $\K^-_0=\Delta_{g_0}\RS^3$. Since $\Gamma\times 1\subseteq \K^-\subseteq \N(\Delta_{g_0}\RS^3)=\pm\Delta_{g_0}\RS^3$, and $\Gamma\neq 1$, then $\K^-$ has to be $\pm\Delta_{g_0}\RS^3$, and $\Gamma =\mathbb{Z}_2$. Also, the classification of transitive actions on spheres gives us that $\K^-_0\cap\HH=\Delta_{g_0}\Lambda$. Therefore $\HH=\pm \Delta_{g_0}\Lambda$, and the following diagram is obtained:
 \begin{equation}\label{EQ: 7, 2}
(\RS^3\times \RS^3, \pm \Delta_{g_0}\Lambda, \pm\Delta_{g_0}\RS^3, \RS^3\times g_0\Lambda g_0^{-1}).                                                                                                       
\end{equation}
According to Proposition~\ref{P:Equivalent_actions} and Equation~\eqref{EQ: 7, 2}, we can assume that  $g_0=1$. This action is equivalent to the
  following action and $\X$ is  then equivariantly  homeomorphic to  $\RP^3\ast(\RS^3/\Lambda)$:
 \begin{align*}
(\RS^3\times\RS^3)\times(\RP^3\ast\RS^3/\Lambda)&\to \RP^3\ast(\RS^3/\Lambda)\\
(g, h)\cdot[x, y, t]&= [gxh^{-1},  hy, t].
\end{align*}


\vspace{.2cm}
Now assume that $\K^+_0=\Delta_{g_0}\RS^3$. Thus $\HH\cap\K^+_0=\Delta_{g_0}\Gamma$. As before, $\K^-/\HH$ can be a circle, a $3$-sphere, or a $3$-dimensional spherical space form. 

First suppose that $\K^-/\HH=\SP^1$. Therefore,  $\K^-_0=\{(e^{xp\theta}, e^{yq\theta})\}$ and, after conjugation, we can assume that it is equal to $\K^-_0=\{(e^{ip\theta}, e^{iq\theta})\}$. Since $\K^+\subseteq N(\K^+_0)=\pm\Delta_{g_0} (\RS^3)$, there are two possibilities for $\K^+$: either $\K^+=\Delta_{g_0}\RS^3$ or $\K^+=\pm\Delta_{g_0} \RS^3$. 

Assume that $\K^+=\Delta_{g_0} \RS^3$. Therefore $\HH=\Delta_{g_0}\Gamma$. Let $q=0$. Then $K^-_0 = S^1\times 1$ and 
\[
\K^-=\K^-_0\HH=\{(za, g_0ag_0^{-1}) \mid a\in \Gamma,\ z\in S^1\,\}.
\]
Since $\Proj_1(\K^-)\subseteq \RS^1\cup j\RS^1$, we have $\Gamma=\mathbb{Z}_k$, or $\Gamma=\D^*_{2m}$. Thus $\K^-$ is equal to $\RS^1\times g_0\mathbb{Z}_k g_0^{-1}$ or $(\RS^1\times 1)\Delta_{g_0}\D^*_{2m}$, respectively. By conjugating the subgroups by $(1, g_0^{-1})$, we have the following diagrams:
  \begin{align}
  \label{D:D7_D2m_A}
(\RS^3\times \RS^3,  \Delta\mathbb{Z}_k,  \RS^1\times \mathbb{Z}_k, \Delta\RS^3),                                                                                                      
\end{align}
 \begin{align}
  \label{D:D7_D2m_B}
(\RS^3\times \RS^3,  \Delta\D^*_{2m},  (\RS^1\times 1)\Delta\D^*_{2m}, \Delta\RS^3).                                                                                                   
\end{align}

If $p=0$, we have,  similarly,
\begin{align}
(\RS^3\times \RS^3,  \Delta\mathbb{Z}_k,   \mathbb{Z}_k\times\RS^1, \Delta\RS^3),                                                                                                      
\end{align}
\begin{align}
(\RS^3\times \RS^3,  \Delta\D^*_{2m},   (1\times\RS^1)\Delta\D^*_{2m}, \Delta\RS^3).                                                                                                   
\end{align}
Observe that these two diagrams are the same as Diagrams~\eqref{D:D7_D2m_A} and \eqref{D:D7_D2m_B} up to exchanging the factors of $G$.

Now assume that $pq\neq 0$. Then $N(\K^-_0)=\{(e^{i\theta}, e^{i\phi})\}\cup \{(je^{i\theta}, je^{i\phi})\}$. Thus $\K^-=\{(e^{ip\theta}, e^{iq\theta})\}$ or $\K^-=\{(e^{ip\theta}, e^{iq\theta})\}\cup \{(je^{ip\theta}, je^{iq\theta})\}$. If $\K^-=\{(e^{ip\theta}, e^{iq\theta})\}$,  then $\Gamma=\mathbb{Z}_k$. We have 
\[
\{(a, g_0ag_0^{-1}) \mid a\in \mathbb{Z}_k\}=\Delta_{g_0}\mathbb{Z}_k=\{(e^{\frac{2\pi i}{k}p}, e^{\frac{2\pi i}{k}q})\}.
\]
Therefore, 
\[
g_0e^{\frac{2\pi i}{k}p}g_0^{-1}=e^{\frac{2\pi i}{k}q}.
\]
Since $e^{\frac{2\pi i}{k}p}$ and $e^{\frac{2\pi i}{k}q}$ are  elements  in the maximal torus $S^1=\{e^{\theta i\mid \theta\in\mathbb{R}}\}\subseteq S^3$, by \cite[Proposition~4.53]{Knapp}, they are conjugate in the Weyl group $W(S^3, S^1)=\{S^1, jS^1\}$. Thus one of the following cases occurs:
\begin{itemize}
\item
$g_0\in \{e^{i\theta}\}$ and $e^{\frac{2\pi i}{k}p}=e^{\frac{2\pi i}{k}q}$. Consequently $k|(p-q)$, and if $k$ is even, then $p,q$ are odd. By conjugating the isotropy groups by $(1, g_0^{-1})$, we have $\K^+=\Delta\RS^3$,   $\K^-=\{(e^{ip\theta}, e^{iq\theta})\}$ and $\HH=\Delta\mathbb{Z}_k$. 

\item 
 $g_0\in \{je^{i\theta}\}$ and $e^{\frac{2\pi i }{k}p}=e^{-\frac{2\pi i}{k}q}$  which gives $k|(p+q)$, and if $k$ is even, then $p$, $q$ are odd. Again, by conjugating the isotropy groups by $(1, g_0^{-1})$, we have $\K^+=\Delta\RS^3$,   $\K^-=\{(e^{ip\theta}, e^{-iq\theta})\}$ and $\HH=\Delta\mathbb{Z}_k$. 

\end{itemize}

Summing up, we have the following diagrams:
  \begin{align}
(\RS^3\times \RS^3,  \Delta\mathbb{Z}_k,  \{(e^{ip\theta}, e^{iq\theta})\}, \Delta\RS^3),                                                                                                      
\end{align}
where $k|(p-q)$, and if $k$ is even, then $p$, $q$ are odd, and we get
 \begin{align}
(\RS^3\times \RS^3,  \Delta\mathbb{Z}_k,  \{(e^{ip\theta}, e^{-iq\theta})\}, \Delta\RS^3),                                                                                                      
\end{align}
where, $k|(p+q)$, and if $k$ is even, then $p$, $q$ are odd.

If $K^-=\{(e^{ip\theta}, e^{iq\theta})\}\cup \{(je^{ip\theta}, je^{iq\theta})\}$,  similar arguments as above give rise to the following diagrams with the same conditions, respectively:
  \begin{align}
(\RS^3\times \RS^3,  \Delta\D^*_{2m},  \{(e^{ip\theta}, e^{iq\theta})\}\cup \{(je^{ip\theta}, je^{iq\theta})\}, \Delta\RS^3),
\end{align}
 \begin{align}
(\RS^3\times \RS^3,  \Delta\D^*_{2m},  \{(e^{ip\theta}, e^{-iq\theta})\}\cup \{(je^{ip\theta}, je^{-iq\theta})\}, \Delta\RS^3).
\end{align}

If $\K^+=\pm\Delta_{g_0}\RS^3$, then $\HH=\pm\Delta_{g_0}\mathbb{Z}_k$ or $\HH=\pm\Delta_{g_0}\D^*_{2m}$. By the same argument, we obtain the following diagrams:
 \begin{align}\label{D:D7_D2m_C}
(\RS^3\times \RS^3,  \pm\Delta\mathbb{Z}_k,  \RS^1\times \mathbb{Z}_k, \pm\Delta\RS^3),                                                                                                      
\end{align}
\begin{align}\label{D:D7_D2m_D}
(\RS^3\times \RS^3,  \pm\Delta\D^*_{2m},  (\RS^1\times 1)\Delta\D^*_{2m}, \pm\Delta\RS^3),                                                                                                 
\end{align}
  \begin{align}
(\RS^3\times \RS^3,  \pm\Delta\mathbb{Z}_k,   \mathbb{Z}_k\times\RS^1, \pm\Delta\RS^3),                                                                                                      
\end{align}
\begin{align}
(\RS^3\times \RS^3,  \pm\Delta\D^*_{2m},   (1\times\RS^1)\Delta\D^*_{2m}, \pm\Delta\RS^3).                                                                                                   
\end{align}
Observe that the last two diagrams are the same as Diagrams~\eqref{D:D7_D2m_C} and \eqref{D:D7_D2m_D} up to exchanging the factors of $G$.

In the following diagrams, $p$ is odd and $q$ is even, so $k$ has to be odd:
 \begin{align}
(\RS^3\times \RS^3,  \pm\Delta\mathbb{Z}_k,  \{(e^{ip\theta}, e^{iq\theta})\}, \pm\Delta\RS^3),                                                                                                      
\end{align}
\begin{align}
(\RS^3\times \RS^3,  \pm\Delta\mathbb{Z}_k,  \{(e^{ip\theta}, e^{-iq\theta})\}, \pm\Delta\RS^3),                                                                                                      
\end{align}
  \begin{align}
(\RS^3\times \RS^3,  \pm\Delta\D^*_{2m},  \{(e^{ip\theta}, e^{iq\theta})\}\cup \{(je^{ip\theta}, je^{iq\theta})\}, \pm\Delta\RS^3),
\end{align}
  \begin{align}
(\RS^3\times \RS^3,  \pm\Delta\D^*_{2m},  \{(e^{ip\theta}, e^{-iq\theta})\}\cup \{(je^{ip\theta}, je^{-iq\theta})\}, \pm\Delta\RS^3).
\end{align}

Now assume that $\K^-/\HH=\SP^3$. Therefore,  $\K^+=\K^+_0=\Delta_{g_0}\RS^3$  and $\HH=\Delta_{g_0}\Gamma$. Further, $\K^-_0$ is equal to $\RS^3\times 1$,  $1\times \RS^3$ or  $\Delta_{g_1}\RS^3$. For  $K^-_0=\RS^3\times 1, 1\times \RS^3$, we have  
 \begin{align}
(\RS^3\times \RS^3,  \Delta\Gamma,  \RS^3\times \Gamma, \Delta\RS^3),                                                                                                      
\end{align}
  \begin{align}
(\RS^3\times \RS^3,  \Delta\Gamma,   \Gamma\times\RS^3, \Delta\RS^3).
\end{align}
 
If $\K^-_0=\Delta_{g_1}\RS^3$, since $\K^-/\HH$ is simply-connected, $\pi_0(\K^-)=\pi_0(\HH)$. The number of connected components of $\K^-$ is at most 2, for $\K^-\subseteq \N(\K^-_0)=\pm \Delta_{g_1}\RS^3$. Thus $\HH=\langle (-1, 1) \rangle$. But then $\HH$ is not  a subgroup of $\K^+$. Therefore, this case cannot occur.  

Now, let $\K^-/\HH=\SP^3/\Lambda$ with $\Lambda$ a non-trivial subgroup of $S^3$. Again, we have three possibilities for $\K^-_0$: $\RS^3\times 1,\, 1\times \RS^3, \, \Delta_{g_1}\RS^3$. For $\K^-_0=\RS^3\times 1,\, 1\times \RS^3$, we obtain a diagram  equivalent to diagram \eqref{EQ: 7, 2} by Proposition \ref{P:Equivalent_actions}. Hence, suppose that $\K^-_0=\Delta_{g_1}\RS^3$.  If $K^+=\Delta_{g_0}\RS^3$, then $\HH=\Delta_{g_0}\Gamma$, so  $\K^-_0\cap\HH=\Delta_{g_1}\Lambda\subseteq \Delta_{g_0}\Gamma$ which implies $g_0^{-1}g_1\in C_{S^3}(\Lambda)$, where

\begin{equation}\label{EQ: 7, 3}
C_{S^3}(\Lambda)=\left\{ \begin{array}{ll}
\{\pm 1\} & \mbox{if $\Lambda\neq \mathbb{Z}_k$}, \\
\{e^{i\theta} \mid \theta\in \R\}& \mbox{if $\Lambda=\mathbb{Z}_k$}.
\end{array} 
\right.
\end{equation}
Moreover, $\{(-a, g_1ag_1^{-1})\mid a\in \Lambda\}\nsubseteq \Delta_{g_0}\Gamma$. Therefore, $\K^-=\Delta_{g_0z}\RS^3$, 
$\Delta_{g_1}\Lambda=\Delta_{g_0}\Gamma$, where $z\in C(\Gamma)$. If $\Gamma\neq \Z_k$, then by \eqref{EQ: 7, 3}, $z=\pm 1$ and in particular $\K^-=\Delta_{g_0}\RS^3$. Hence, after conjugating all subgroups by $(1, g_0^{-1})$, we obtain an equivalent diagram given by 
\begin{equation}\label{Eq_GD_7_1}
  (\RS^3\times \RS^3,  \Delta\Gamma,  \Delta\RS^3, \Delta\RS^3).                                                                                                   
  \end{equation} 
  If $\Gamma=\Z_k$, then we first conjugate all subgroups by $(1, g_0^{-1})$. Then, since $(1, z) \in N(H)_0$ by \eqref{EQ: 7, 3}, we can conjugate $\K^-$ by $(1, z)$ to obtain  diagram \eqref{Eq_GD_7_1}.

This action is equivalent to the following action on $\Susp(\RS^3/\Gamma)\times \SP^3$:
 \begin{align*}
(\RS^3\times\RS^3)\times(\Susp(\RS^3/\Gamma)\times \SP^3)&\to \Susp(\RS^3/\Gamma)\times \SP^3\\
(g, h)\cdot([x, t], y)&= ([gx, t],  gyh^{-1}).
\end{align*}

Now assume that $\K^+=\pm \Delta_{g_0}\RS^3$. Then, $\K^-=\pm\Delta_{g_1}\RS^3$ and $\pm\Delta_{g_0}\Gamma=H=\pm\Delta_{g_1}\Lambda$. Thus, we have $\Gamma=\Lambda$ and $g_1=g_0z$, for some $z\in N(\Gamma)$. We claim that if $z\in C_{S^3}(\Gamma)$, then this case cannot happen, since we have assumed that $X$ is simply-connected. Indeed, if $z\in C_{S^3}(\Gamma)$ then $\Delta_{g_0}\Gamma=\Delta_{g_0z}\Gamma=\Delta_{g_1}\Gamma$, and therefore, $H^-=H^+$. Since $X$ is simply-connected, by Proposition~\ref{fundamental group}, $H=\langle H^-, H^+\rangle=\Delta_{g_0}\Gamma$, which is a contradiction. Assume now that $z\notin C_{S^3}(\Gamma)$. A direct computation shows that either $\Gamma=\D^*_{2m}$, with $z=i$ or $z=j$, where in the latter case $2|m$, or $\Gamma=\Z_k$, with $z=j$ and $4|k$. As a result, we have the following diagram:
 \begin{align}
(\RS^3\times \RS^3,  \pm\Delta\Gamma,  \pm\Delta_z\RS^3, \pm\Delta\RS^3),                                                                                                      
\end{align}
where $\Gamma$ and $z$ are as above, respectively.
\\

\noindent $\mathbf{\GG=S^3\times S^3\times S^1}$. In this case, $\dim\HH=1$. By Proposition~\ref{P:ABELIAN_G}, $\K^-/\HH=\SP^1$ and both $\K^+_0$ and $\HH_0$ are subgroups of $\RS^3\times\RS^3\times 1$. Thus we can assume that $\HH_0=\{(e^{ip\theta}, e^{iq\theta})\}\times 1$ and $\Proj_3(\K^-_0)=\RS^1$. Since $\K^+\subseteq \RS^3\times\RS^3\times 1$, an examination of the subgroups of $\RS^3\times\RS^3$ shows that the only possibilities for $\K^+/\HH$ are $\RP^2$ and $\RS^3/\Gamma$ with $\Gamma$ finite and non-trivial.

 First assume that $\K^+/\HH=\RP^2$. Therefore, $\dim\K^+=3$. The possibilities for $K^+_0$ are $\RS^3\times 1\times 1,$ $1\times\RS^3\times 1$ and $\Delta_{g_0}\RS^3\times 1$. 
 
 Let $K^+_0=\RS^3\times 1\times 1$. Then $\HH_0=\RS^1\times 1\times 1$, $\HH\cap\K^+_0=\N(\RS^1)\times 1\times 1$, and $\K^-_0=\RS^1\times \T^1$, where $\T^1\subseteq \RS^3\times\RS^1$.  Since $\N(\RS^1)$ is  a maximal subgroup of $\RS^3$, $\HH\cap\K^-_0=\RS^1\times \mathbb{Z}_k$, where $\mathbb{Z}_k\subseteq \T^1$. Therefore, we obtain the following diagram:
   \begin{align}
(\RS^3\times \RS^3\times \RS^1, \N(\RS^1)\times\mathbb{Z}_k, \N(\RS^1)\times\T^1, \RS^3\times \mathbb{Z}_k).
\end{align}
By Proposition~\ref{Primitive Action}, this action is a non-primitive action with $\LL=\RS^3\times \RS^1\times\RS^1$ and $\X$ is equivariantly  homeomorphic to the total space of an $(\RP^2\ast\SP^1)$-bundle over $\SP^2$. Note that the action of $\RS^3\times \RS^1\times\RS^1$ on $\RP^2\ast\SP^1$ is in fact the normal extension of the action of $\RS^3\times \RS^1$ on $\RP^2\ast\SP^1$.  

Let $K^+_0=1\times S^3\times 1$. This case only differs from the previous one by an isomorphism of $G$ which exchanges the factors. Therefore the analysis is analogous to the one in the preceding case.

Now let $\K^+_0=\Delta\RS^3\times 1$. As a result,  $\HH_0=\Delta\RS^1\times 1$, $K^+_0\cap \HH=\N_{\Delta\RS^3}(\Delta\RS^1)\times 1$ and
\begin{align*}
\K^-_0&=(\Delta\RS^1\times 1)\{(e^{ia\theta}, e^{ib\theta}, e^{ic\theta})\}\\
&=\{(e^{i(\phi+a\theta)}, e^{i(\phi+b\theta)}, e^{ic\theta})\}\\
&= \{(e^{i\phi}, e^{i\phi}e^{ip\theta}, e^{ic\theta})\}.
\end{align*}
Since the action is non-reducible, $\Proj_3(\HH\cap\K^-_0)$ is a proper subgroup of $\RS^1$, namely $\Proj_3(\HH\cap\K^-_0)=\mathbb{Z}_k$. Therefore, $\HH\cap\K^-_0=\{(e^{i\phi}, e^{i\phi}e^{\frac{2\pi pli}{k}}, e^{\frac{2\pi cli}{k}})\mid 1\leq l\leq k\}$. The long exact sequence of homotopy groups corresponding to the fiber bundle 
\[
\K^-\to \GG\to \GG/\K^-
\]
shows that $\pi_0(\K^-)=\pi_1(\GG/\K^-)/\mathbb{Z}_c$ (note that $\pi_0(\K^-)$ is not trivial since $K^+_0\cap \HH=N_{\Delta\RS^3}(\Delta\RS^1)\times 1\subseteq \K^-$). By Proposition~\ref{fundamental group}, $\pi_1(\GG/\K^-)=\mathbb{Z}_2$ as $\X$ is simply-connected. Thus $c=1$. On the other hand, $\Z_k\subseteq \N(K^+_0)$, which gives $k|2p$, and since we can assume $\HH\cap (1\times 1\times \RS^1)=1$ to have a more effective action, $k=1, 2$. Therefore,   we obtain the following diagram:
  \begin{align}
(\RS^3\times \RS^3\times \RS^1, N_{\Delta\RS^3}(\Delta\RS^1)\mathbb{Z}_k, \T^2\cup (j, j, 1)\T^2,  \Delta\RS^3\mathbb{Z}_k),                                                                                                    
\end{align}
where $\T^2=\{(e^{i\phi}, e^{i\phi}e^{ip\theta}, e^{i\theta}) \mid \phi,\theta\in\mathbb{R}\}$, and $k=1, 2$. 
\\

Now let $\K^+/\HH=\RS^3/\Gamma$ with $\Gamma$ finite and non-trivial. Therefore,  $\K^+_0=\RS^3\times\RS^1\times 1$ or   $\K^+_0=\RS^1\times\RS^3\times 1$. 

Suppose that $\K^+_0=\RS^3\times\RS^1\times 1$ and $\Gamma=\mathbb{Z}_k$. Then $\HH\cap\K^+_0=\{(e^{\frac{2\pi li}{k}}e^{p\theta i}, e^{\theta i}, 1) \mid 1\leq l\leq k\}$. Since $\X$ is simply-connected, by Proposition \ref{fundamental group}, and the exact sequences of homotopy groups related to the fiber bundles 
\[\K^{\pm}/\HH\to \GG/\HH\to \GG/\K^{\pm},\]
\[\K^-\to \GG\to \GG/\K^-,\]
one can see that $\pi_0(\K^-)=\mathbb{Z}_k/\mathbb{Z}_c$. Therefore, $c|k$ and the following diagram is obtained:
  \begin{align}
(\RS^3\times \RS^3\times \RS^1, \{(e^{\frac{2\pi ari}{m}}e^{\frac{2\pi li}{k}}e^{p\theta i}, e^{\theta i}, e^{\frac{2\pi cri}{m}})\mid 1\leq r\leq m,1\leq l\leq k\},\\ \{(e^{\frac{2\pi li}{k}}e^{ip\theta}e^{ia\phi}, e^{i\phi}, e^{ic\theta}), 1\leq l\leq k\},  \RS^3\times\RS^1\times\mathbb{Z}_{\frac{k}{c}}).\notag
\end{align}

Now let $\Gamma\neq\mathbb{Z}_k$. By Lemma~\ref{L:M_TRANS}, $\K^+_0\cap\HH=\Gamma\times \RS^1\times 1$. Therefore,  $\HH_0=1\times \RS^1\times 1$  and, by Lemma~\ref{P:ABELIAN_G}, 
\begin{align*}
\K^-_0 	& =\HH_0\{(e^{ia\theta}, e^{ib\theta}, e^{ic\theta})\}\\
		& =\{(e^{ia'\theta}, e^{i\phi}, e^{ic'\theta})\},
\end{align*} for some integers $a'$, $c'$. If $a'=0$, then we have the following diagram:
  \begin{align}
(\RS^3\times \RS^3\times \RS^1, \Gamma\times\RS^1\times\mathbb{Z}_k, \Gamma\times\T^2,  \RS^3\times\RS^1\times\mathbb{Z}_k).
\end{align}
By Proposition~\ref{P:Product-action},  this action is a product action  and $\X$ is equivariantly  homeomorphic to $(\RS^3/\Gamma\ast\SP^1)\times\SP^2$.  If $a'\neq 0$, then $N(\K^-)=\{(e^{ia'\theta}, e^{i\phi}, e^{ic'\theta})\}\cup\{(e^{ia'\theta}, je^{i\phi}, e^{ic'\theta})\}$ which implies that $\Gamma=\mathbb{Z}_k$. Thus, this case cannot happen. 
\\


\noindent $\mathbf{\GG=S^3\times S^3\times S^3}$. In this case, $\dim\HH=3$. Since the action is non-reducible, $\Proj_i(\HH)\subsetneq \RS^3$, $i=1,2,3$. Therefore, $\HH_0$ must be a maximal torus of $\GG$. Further, by considering the subgroups of $\GG$ containing $\HH$, we only have $\RP^2$ and  $\SP^2$ as the normal spaces of directions of singular orbits.   

Assume, without loss of generality, that $\K^+/\HH=\RP^2$. Then there are two possibilities for $\K^-/\HH$, namely, $\K^-/\HH=\SP^2$ or $\K^-/\HH=\RP^2$.

Let $\K^-/\HH=\SP^2$. Therefore, $\K^+$ is connected and, after exchanging the factors of $G$ if necessary, we can assume that $\K^+=\K^+_0=\RS^3\times \T^2$. Thus $\HH=\HH\cap\K^+_0=\N_{\RS^3}(\RS^1)\times \T^2$. Also, since $\K^-/\HH$ is simply-connected, $\pi_0(\K^-)=\pi_0(\HH)=\mathbb{Z}_2$,  and their components intersect each other. Hence $\K^-=\N_{\RS^3}(\RS^1)\times \RS^3\times \RS^1$, and we get the following diagram:
 \begin{align}
(\RS^3\times \RS^3\times \RS^3, \N_{\RS^3}(\RS^1)\times \T^2, \N_{\RS^3}(\RS^1)\times \RS^3\times \RS^1,  \RS^3\times\T^2).
\end{align}
By Proposition~\ref{P:Product-action},  this action is a product action  and $\X$ is equivariantly  homeomorphic to $(\RP^2\ast\SP^2)\times\SP^2$.

 Now let $\K^-/\HH=\RP^2$.  Thus $\K^-_0$ is equal to  $\RS^3\times \T^2$ or $\RS^1\times \RS^3\times \RS^1$. If $\K^-_0=\RS^3\times \T^2$, then $\HH=\HH\cap\K^-_0=\HH\cap \K^+_0=\N_{\RS^3}(\RS^1)\times \T^2$, and the following diagram is obtained:
\begin{align}
(\RS^3\times \RS^3\times \RS^3, \N_{\RS^3}(\RS^1)\times \T^2, \RS^3\times\T^2,  \RS^3\times\T^2).
\end{align}
By Proposition~\ref{P:Product-action},  this action is a product action  and $\X$ is equivariantly  homeomorphic to $\Susp(\RP^2)\times(\SP^2\times \SP^2)$.
 If $\K^-_0=\RS^1\times \RS^3\times \RS^1$, then $\HH\cap\K^-_0=\RS^1\times  \N_{\RS^3}(\RS^1)\times\RS^1$, and so $\HH=\N_{\RS^3}(\RS^1)\times\N_{\RS^3}(\RS^1)\times\RS^1$. Hence we have the diagram
\begin{align}
(\RS^3\times \RS^3\times \RS^3, \N_{\RS^3}(\RS^1)\times\N_{\RS^3}(\RS^1)\times\RS^1, \N_{\RS^3}(\RS^1)\times \RS^3\times \RS^1,  \RS^3\times\N_{\RS^3}(\RS^1)\times \RS^1).                                                                                                      
\end{align}
By Proposition~\ref{P:Product-action},  this action is a product action  and $\X$ is equivariantly  homeomorphic to $(\RP^2\ast\RP^2)\times\SP^2$. 
\\



\noindent $\mathbf{G=SU(3)}$. In this case, $\dim\HH=2$, and so $\HH_0=\T^2$. The subgroups containing $\T^2$ are $\UU(2)$ and $\SU(3)$. 
\vspace{.2cm}

Assume first that $\K^{\pm}=\SU(3)$. Then the two following diagrams occur:
\begin{align}
(\SU(3), \T^2, \SU(3), \SU(3)),
\end{align}
\begin{align}
(\SU(3), \T^2\mathbb{Z}_2, \SU(3), \SU(3)),
\end{align}
The space  $\X$ is equivariantly  homeomorphic to $\Susp(\W^6)$ or $\Susp(\W^6/\mathbb{Z}_2)$, respectively, where $\W^6=\SU(3)/\T^2$ is the Wallach manifold.
\vspace{.2cm}

Suppose now that $\K^+=\SU(3)$ and  $\K^-_0=\UU(2)$. Since $\UU(2)$ is a maximal subgroup of $\SU(3)$, $\K^-=\K^-_0=\UU(2)$. We also have   $\N_{\UU(2)}(\T^2)/\T^2=\mathbb{Z}_2$, which gives $\HH=\T^2$ or $\T^2\mathbb{Z}_2$. As a result, the two following diagrams are obtained:
\begin{align}
(\SU(3), \T^2, \UU(2), \SU(3)),
\end{align}
\begin{align}
(\SU(3), \T^2\mathbb{Z}_2, \UU(2), \SU(3)).
\end{align}

Finally, assume that $\K^{\pm}=U(2)$ (up to conjugation in $\GG$). Let $\T^2=\diag(\SU(3))$. If $\K^{\pm}$ contains this $\T^2$, then it must
be a conjugate of $\UU(2)$ by an element of the  group $\N(\T^2)/\T^2$. Therefore, there are two possibilities for the pair $\K^+, \K^-$ up
to conjugacy of $G$: $\mathrm{S}(\UU(1)\UU(2))$, and $\mathrm{S}(\UU(2)\UU(1))$ (see \cite[Case~$4_7$]{Hoelscher}). On the other hand, since $\UU(2)/\HH_0=\SP^2$, $\HH$ must be $\T^2\mathbb{Z}_2$. However, $\mathrm{S}(\UU(1)\UU(2))\cap \mathrm{S}(\UU(2)\UU(1))=\T^2$, so $\K^-$, $\K^+$ should be the same. Thus we obtain the following diagram:
 \begin{align}
 (\SU(3), \T^2\mathbb{Z}_2, \UU(2), \UU(2)).
\end{align}
This action is a non-primitive action, and $\X$ is equivariantly  homeomorphic to the total space of  a 
$\Susp(\RP^2)$-bundle over  $\CP^2$.
\\


\noindent $\mathbf{\GG=SU(3)\times S^1}$. In this case, $\dim\HH=3$. By Proposition \ref{P:ABELIAN_G}, $\HH$, and $\K^+_0\subseteq \SU(3)\times 1$, $\K^-/\HH=\SP^1$. Since $\HH$ is $3$-dimensional, $\HH_0$ must be $\SO(3)\times 1$ or $\SU(2)\times 1$. If $\HH_0=\SO(3)\times 1$, then $\K^+$ has to be $\SU(3)\times 1$ since there is no exceptional  orbit. However, the classification of positively curved homogeneous spaces shows that this cannot occur. Therefore, $\HH_0=\SU(2)\times 1$. Since $\K^+/\HH$ is not a sphere, $\K^+=\SU(3)\times 1$, and $\K^+/\HH=\SP^5/\mathbb{Z}_k$. Thus $\K^+_0\cap \HH=\mathrm{S}(\U(2)\mathbb{Z}_k)$. On the other hand, $\K^-$ is a $4$-dimensional subgroup of $\GG$ containing $\mathrm{S}(\U(2)\mathbb{Z}_k)$ whose projection to $\RS^1$ is $\RS^1$. Hence it has to be $\mathrm{S}(\U(2)\mathbb{Z}_k)\times \RS^1$. As a result, we get the following diagram, and $\X$  is equivariantly  homeomorphic to $(\SP^5/\mathbb{Z}_k)\ast\SP^1$:
\begin{align}
(\SU(3)\times \RS^1, \mathrm{S}(\U(2)\mathbb{Z}_k)\times \mathbb{Z}_l, \mathrm{S}(\UU(2)\mathbb{Z}_k)\times \RS^1, \SU(3)\times \mathbb{Z}_l).
\end{align}


\noindent $\mathbf{\GG=SU(3)\times S^3}$.  In this case, $\dim\HH=5$. Since $\Proj_2(\HH_0)\subsetneq \RS^3$, we have $H_0=\UU(2)\times \RS^1$. Thus  $\UU(2)\subseteq\Proj_1(\K^{\pm}_0)\subseteq \SU(3)$ and $\RS^1\subseteq\Proj_2(\K^{\pm}_0)\subseteq \RS^3$. Since   $\UU(2)$ is a maximal subgroup of $\SU(3)$, and $\RS^1$ is a maximal connected subgroup of $\RS^3$, $\Proj_1(\K^{\pm}_0)=\UU(2)$ or $\SU(3)$, and $\Proj_2(\K^{\pm}_0)=\RS^1$ or $\RS^3$. But $\GG/\HH$ is not homeomorphic to a positively curved homogeneous space, so $\K^{\pm}$ are proper subgroups of $G$. Moreover, $\dim\K^{\pm}>\dim\HH$, for $\X$ does not have an exceptional orbit. Therefore, we have the following cases: $\K^{\pm}_0=\SU(3)\times S^1$, $\K^+_0=\SU(3)\times S^1$ and $\K^-_0=\U(2)\times S^3$, and $\K^{\pm}_0=\U(2)\times S^3$.
\vspace{.2cm}

Assume first that $\K^{\pm}_0=\SU(3)\times S^1$. Since $\UU(2)$ is maximal, $\HH=\HH_0=\UU(2)\times \RS^1$, and we obtain the following diagram:
 \begin{align}
 (\SU(3)\times \RS^3, \UU(2)\times \RS^1, \SU(3)\times\RS^1, \SU(3)\times\RS^1).
\end{align}
By Proposition~\ref{P:Product-action},  this action is a product action  and $\X$ is equivariantly  homeomorphic to $\Susp(\CP^2)\times\SP^2$. 
\vspace{.2cm}

Suppose now that $\K^+_0=\SU(3)\times S^1$ and $\K^-_0=\UU(2)\times S^3$. In this case, $\K^-/\HH$ can be either $\SP^2$ or $\RP^2$. Therefore, we have the following diagrams, respectively,
\begin{align}
(\SU(3)\times \RS^3, \UU(2)\times \RS^1, \SU(3)\times\RS^1, \UU(2)\times \RS^3),
\end{align}
\begin{align}
(\SU(3)\times \RS^3, \UU(2)\times \N_{\RS^3}(\RS^1), \SU(3)\times\N_{\RS^3}(\RS^1), \UU(2)\times \RS^3).
\end{align}
and $\X$ is equivariantly  homeomorphic to $\CP^2\ast\SP^2$  or  $\CP^2\ast\RP^2$, respectively.
\vspace{.2cm}

Finally, suppose that $\K^{\pm}_0=\UU(2)\times S^3$. Since $\K^{\pm}_0=\UU(2)\times \RS^3$ is a maximal subgroup of $\SU(3)\times \RS^3$, $\K^{\pm}=\K^{\pm}_0$. Thus $\K^{\pm}/\HH$  has to be $\RP^2$, and consequently, we get the following diagram:
\begin{align}
(\SU(3)\times \RS^3, \UU(2)\times \N_{\RS^3}(\RS^1), \UU(2)\times \RS^3, \UU(2)\times \RS^3).
\end{align}
By Proposition~\ref{P:Product-action},  this action is a product action  and $\X$ is equivariantly  homeomorphic to $\Susp(\RP^2)\times\CP^2$. \\

\noindent $\mathbf{\GG=Sp(2)}$.  In this case, $\dim\HH=4$. Hence $\HH_0=\UU(2)_{\max}$ (the maximal subgroup of $\Sp(2)$) or $\HH_0=\Sp(1)\SO(2)$. If $\HH_0=\UU(2)_{\max}$, then $\K^{\pm}$ have to be $\Sp(2)$, which is impossible, for $\Sp(2)/\UU(2)_{\max}$ does not admit a positively curved metric (see~\cite{WZ}). Thus $\HH_0=\Sp(1)\SO(2)$.  Since the only proper subgroup of $\GG$ containing $\HH_0$ is $\Sp(1)\Sp(1)$, we have the following cases: $\K^{\pm}_0=\Sp(2)$, $\K^+=\Sp(2)$ and $\K^-_0=\Sp(1)\Sp(1)$, and $\K^{\pm}_0=\Sp(1)\Sp(1)$.
\vspace{.2cm}

Assume first that $\K^{\pm}_0=\Sp(2)$. Therefore we have the following diagrams:
 \begin{align}
 (\Sp(2), \Sp(1)\SO(2), \Sp(2), \Sp(2)),
\end{align}
\begin{align}
(\Sp(2), \Sp(1)\SO(2)\mathbb{Z}_2, \Sp(2), \Sp(2)).
\end{align}
and $\X$ is equivariantly  homeomorphic to $\Susp(\CP^3)$ or to $\Susp(\CP^3/\mathbb{Z}_2)$, respectively.
\vspace{.2cm}

Suppose now that $\K^+=\Sp(2)$ and $\K^-_0=\Sp(1)\Sp(1)$. The space $\K^+/\HH$ is equal to either $\CP^3$ or $\CP^3/\mathbb{Z}_2$, so $\HH=\Sp(1)\SO(2)$‌ or $\HH=\Sp(1)\SO(2)\mathbb{Z}_2$, respectively. Therefore, we obtain the following diagrams:
 \begin{align}
 (\Sp(2), \Sp(1)\SO(2), \Sp(1)\Sp(1), \Sp(2)),
\end{align}
\begin{align}
(\Sp(2), \Sp(1)\SO(2)\mathbb{Z}_2, \Sp(1)\Sp(1), \Sp(2)).
\end{align}

\vspace{.2cm}
Finally, assume that $\K^{\pm}_0=\Sp(1)\Sp(1)$ (up to a conjugation). Since $ \K^{\pm}_0$ both contain $\HH_0=\Sp(1)\SO(2)$, they should be equal,  so $\HH\cap\K^+_0=\HH\cap\K^-_0$, which in turn implies that $\K^{\pm}$ is connected. On the other hand, $\K^{\pm}/\HH$ should be a positively curved homogeneous space not homeomorphic to a sphere. Therefore $\HH=\Sp(1)\SO(2)\mathbb{Z}_2$, and we get the following diagram:
\begin{align}
(\Sp(2), \Sp(1)\SO(2)\mathbb{Z}_2, \Sp(1)\Sp(1), \Sp(1)\Sp(1)).
\end{align}
This action is a non-primitive action with $\LL=  \Sp(1)\Sp(1)$, and $\X$ is equivariantly  homeomorphic to the total space of a $\Susp(\RP^2)$-bundle over   $\SP^4$.\\


\noindent  $\mathbf{\GG=Sp(2)\times S^3}$. Since the action is non-reducible, $\Proj_2(\HH_0)\subsetneq \RS^3$. Since $\dim\HH=7$, and the highest  dimension of  a proper subgroup of $\Sp(2)$ is $6$, $\Proj_2(\HH_0)$ must be equal to $\RS^1$. Thus $\HH_0=\Sp(1)\Sp(1)\times\RS^1$. Maximality of $\RS^1$ in $\RS^3$, and of $\Sp(1)\Sp(1)$ in $\Sp(2)$, gives rise to the following cases: $\K^{\pm}_0=\Sp(1)\Sp(1)\times S^3$, $\K^+_0=\Sp(1)\Sp(1)\times S^3$ and $\K^-_0=\Sp(2)\times S^1$, and $\K^{\pm}_0=\Sp(2)\times S^1$.
\vspace{.2cm}

Assume first that  $\K^{\pm}_0=\Sp(1)\Sp(1)\times S^3$. Since $\HH\cap\K^-_0=\HH\cap\K^+_0$, we have $\HH=\HH\cap\K^-_0=\HH\cap\K^+_0\subseteq \K^{\pm}_0$. Therefore, $\K^{\pm}$ are connected, and we get the following diagram:
 \begin{align}
 (\Sp(2)\times\RS^3, \Sp(1)\Sp(1)\times\N_{\RS^3}(\RS^1), \Sp(1)\Sp(1)\times\RS^3, \Sp(1)\Sp(1)\times\RS^3).
\end{align}
By Proposition~\ref{P:Product-action},  this action is a product action  and $\X$ is equivariantly  homeomorphic to  $\Susp(\RP^2)\times\SP^4$.
\vspace{.2cm}

Suppose now that $\K^+_0=\Sp(1)\Sp(1)\times S^3$ and $\K^-_0= \Sp(2)\times S^1$. We have the following diagrams:
 \begin{align}
 (\Sp(2)\times\RS^3, \Sp(1)\Sp(1)\times\N_{\RS^3}(\RS^1), \Sp(1)\Sp(1)\times\RS^3, \Sp(2)\times\N_{\RS^3}(\RS^1)),
\end{align}
corresponding to the join action of $\Sp(2)\times\RS^3$ on $\SP^4\ast\RP^2$;
 \begin{align}
 (\Sp(2)\times\RS^3, \Sp(1)\Sp(1)\mathbb{Z}_2\times\RS^1, \Sp(1)\Sp(1)\mathbb{Z}_2\times\RS^3, \Sp(2)\times\RS^1),
\end{align}
corresponding to the join action of $\Sp(2)\times\RS^3$ on $\SP^2\ast\RP^4$; and
 \begin{align}
 (\Sp(2)\times\RS^3, \Sp(1)\Sp(1)\mathbb{Z}_2\times\N_{\RS^3}(\RS^1), \Sp(1)\Sp(1)\mathbb{Z}_2\times\RS^3, \Sp(2)\times\N_{\RS^3}(\RS^1)).
\end{align}
which corresponds to the join action of $\Sp(2)\times\RS^3$ on $\RP^4\ast\RP^2$.
\vspace{.2cm}

Finally, assume that $\K^{\pm}_0=\Sp(2)\times S^1$.  In this case $\K^{\pm}$ are connected and $\HH= \Sp(1)\Sp(1)\mathbb{Z}_2\times\RS^1$. Thus we obtain the following diagram:
 \begin{align}
 (\Sp(2)\times\RS^3, \Sp(1)\Sp(1)\mathbb{Z}_2\times\RS^1, \Sp(2)\times \RS^1, \Sp(2)\times \RS^1).
\end{align}
By Proposition~\ref{P:Product-action},  this action is a product action  and $\X$ is equivariantly  homeomorphic to  $\Susp(\RP^4)\times\SP^2$.\\

\noindent  $\mathbf{\GG=\GG_2}$. Since $\dim\HH$ has to be $8$, we have $\HH=\SU(3)$. Thus $\K^{\pm}=\mathrm{G}_2$, for $\SU(3)$ is a maximal connected subgroup of $\mathrm{G}_2$ and there are no exceptional orbits. As a result, we have the following diagram and $\X$ is  equivariantly  homeomorphic to $\Susp(\RP^6)$:
 \begin{align}
 (\mathrm{G}_2, \N_{\mathrm{G}_2}(\SU(3)), \mathrm{G}_2, \mathrm{G}_2).
\end{align}
\\


\noindent  $\mathbf{\GG=SU(4)}$. In this case, $\dim\HH=9$, so $\HH_0=\UU(3)$. Because $\UU(3)$ is a maximal subgroup of $\GG$, we have $\K^{\pm}=\SU(4)$, and the following diagram is obtained:
 \begin{align}
 (\SU(4), \UU(3), \SU(4), \SU(4)).
\end{align}
This space is equivariantly  homeomorphic to $\Susp(\CP^3)$. 
\\


\noindent $\mathbf{\GG=SU(4)\times S^1}$. 
In this case,  $\dim\HH=10$. We have  $\K^+_0 ,\HH_0\subseteq \SU(4)\times 1$ and $\K^-/\HH=\SP^1$ by Proposition~\ref{P:ABELIAN_G}. Therefore, $\K^+_0=\SU(4)\times 1$,  $\HH_0=\Sp(2)\times 1$, $\K^-_0=\Sp(2)\times\RS^1$, and we get the following diagram:
 \begin{align}
 (\SU(4)\times \RS^1, (\Sp(2)\mathbb{Z}_2)\times\mathbb{Z}_k, (\Sp(2)\mathbb{Z}_2)\times\RS^1, \SU(4)\times\mathbb{Z}_k).
\end{align}
This action is equivalent to  the join action of $\SU(4)\times \RS^1$ on $\RP^5\ast\SP^1$.
\\


\noindent $\mathbf{\GG=Spin(7)}$. 
In this case, since $\dim\GG=21=(6\times 7)/2$€Œ, by Proposition~\ref{P:Spin-action}, we have 
\begin{align}
(\Spin(7), \N_{\Spin(7)}(\Spin(6)), \Spin(7), \Spin(7)),
\end{align}
and $\X$ is equivariantly  homeomorphic to $\Susp(\RP^6)$. 
\\

This concludes the proof of Theorem~\ref{T:CLASSIF}. \hfill $\square$


\section{Proof of Theorem~\ref{T:GOOD_ORBIFOLDS}}
\label{S:ORBIFOLDS}

In this section we prove Theorem~\ref{T:GOOD_ORBIFOLDS}, i.e.~we show that the underlying space of a closed, simply-connected cohomogeneity one smooth orbifold in dimension at most $7$ is  equivariantly homeomorphic to a good orbifold.  First, we recall basic definitions and facts about orbifolds (for more details see, for example, \cite{Adem}, \cite{Davis}).


\begin{defn}
\label{DEF:ORBIFOLD}
An \emph{orbifold chart} on a topological space $X$ is a quadruple $(\tilde{U}, G, U, \pi)$, where
\begin{itemize}
\item $U$ is an open subset of $X$,\\
\item $\tilde{U}$ is open in $\R^n$ and $G$ is a finite group of diffeomorphisms of $\tilde{U}$,\\
\item $\pi :\tilde{U}\to U$ is a map which can be factored as $\pi=\bar{\pi}p$, where $p:\tilde{U}\to\tilde{U}/G$
is the orbit map and $\bar{\pi}:\tilde{U}/G \to U$ is a homeomorphism.
\end{itemize}
\end{defn}

For $i = 1, 2$, suppose that $(\tilde{U_i}, G_i, U_i, \pi_i)$ are orbifold charts on $X$. The charts
are \emph{compatible} if given points $\tilde{u_i}\in \tilde{U_i}$ with $\pi_1(\tilde{u_1})=\pi_2(\tilde{u_2})$, there is a diffeomorphism
$h$ from a neighborhood of $\tilde{u_1}$ in $\tilde{U_1}$ onto a neighborhood of $\tilde{u_2}$ in $\tilde{U_2}$ so
that $\pi_1=\pi_2 h$ on this neighborhood.  An \emph{orbifold atlas} on $X$ is a collection $(\tilde{U_i}, G_i, U_i, \pi_i)$ of
compatible orbifold charts which cover $X$.


\begin{defn}\label{D:Orb. Atlas}
Let $(\tilde{U_i}, G_i, U_i, \pi_i)_{i\in I_1}$ and $(\tilde{U_i}, G_i, U_i, \pi_i)_{i\in I_2}$ be orbifold atlases over
a given topological space $X$. We say that they \emph{define the same orbifold structure} on
$X$ if the union atlas $(\tilde{U_i}, G_i, U_i, \pi_i)_{i\in I_1\cup I_2}$ satisfies the compatibility condition in Definition~\ref{DEF:ORBIFOLD}.
\end{defn}

Definition \ref{D:Orb. Atlas} determines an equivalence relation on the set of orbifold atlases over a
given topological space $X$.


\begin{defn}
 An $n$-dimensional \emph{smooth orbifold}, denoted by $Q$, is a second-countable,
Hausdorff topological space $|Q|$, called the \emph{underlying topological} space of $Q$, together with
an equivalence class of orbifold atlases on $Q$.
\end{defn}


\begin{defn}
Let $q\in |Q|$ and  $(\tilde{U}, G, U, \pi)$ be 
any local chart around $q=\pi(x)$.  We define the \emph{local group} at $q$ as
\[ G_q= \{g\in G  \,| \,gx = x\}.\]
This group is uniquely determined up to conjugacy in $G$.
\end{defn}

The notion of local group is used  to define the singular set of the orbifold.


\begin{defn}
For an orbifold $Q$, we define its \emph{singular set} as
\[S_Q =\{q\in Q\,|\, G_q \neq 1\}.\]
We call  $R_Q =Q\setminus S_Q$ the \emph{regular set} of $Q$.
\end{defn}


\begin{defn}
A \emph{covering orbifold} of an orbifold $Q$ is an orbifold $\tilde{Q}$ with a
projection map $\mathsf{p}\colon|\tilde{Q}| \to |Q|$ between the underlying spaces, such that each point $q\in |Q|$
has a neighborhood $U=\tilde{U}/G$ (where $\tilde{U}$ is an open subset of $\R^n$) for which each 
component  of $\mathsf{p}^{-1}(U)$ is homeomorphic to $\tilde{U}/G_i$, where $G_i\subseteq G$ is some subgroup. The
homeomorphism must respect the projections.
\end{defn}


\begin{defn}
 An orbifold is \emph{good} (or \emph{developable}) if it has some covering orbifold which is
a manifold. Otherwise it is \emph{bad} (or \emph{non-developable}).
\end{defn}


\begin{defn}
An orbifold $Q$  is called \emph{simply-connected} if it is connected and does not admit a
non-trivial orbifold covering, i.e. if $\mathsf{p}:|\tilde{Q}| \to |Q| $ is a covering with $|\tilde{Q}|$ connected,
then $\mathsf{p}$ is a homeomorphism.
\end{defn}


\begin{prop}[\protect{\cite[Proposition~13.2.4.]{Thurston}}]
Any connected orbifold $Q$ admits a simply-connected orbifold covering
$\pi\colon\tilde{Q} \to Q$. This has the following universal property:
if
$q \in |Q|-S_Q$ is a base point for $Q$ and $\tilde{q}$ is a base point for $\tilde{Q}$ which projects to $q$, then
 for
any other covering orbifold
\[
\pi'\colon\tilde{Q}' \to Q 
\]
with base point $\tilde{q}'$, there is a lift $\sigma\colon \tilde{Q}\to \tilde{Q}'$ of $\pi'$ to a covering map of
$\tilde{Q}'$.
\end{prop}


\begin{defn}
 The \emph{orbifold fundamental group} $\pi_1^{orb}(Q)$ of an orbifold $Q$ is the group
of deck transformations of the universal orbifold cover $\pi:\tilde{Q} \to Q$.
\end{defn}

The following proposition gives a necessary and sufficient condition for an orbifold to be good in terms of its local groups (see \cite[Page~7]{Davis}). For each $q\in |Q|$, let $G_q$ denote
the local group at $q$.  We can identify $G_q$ with the orbifold fundamental group of a neighborhood
of the form $\tilde{U}_q/G_q$ where $\tilde{U}_q$ is a ball in some linear representation. We say that $G_q$ is
the \emph{local fundamental group} at $q$. The inclusion of the neighborhood into $Q$ induces
a homomorphism $G_q\to \pi_1^{orb}(Q)$.


\begin{prop}[\protect{\cite[Proposition~1.18]{Davis}}]
\label{P:Injection_of_Local_Gr}
An orbifold $Q$ is good if and only if  each local group injects into the orbifold fundamental group, i.e. for each $q\in |Q|$, the homomorphism $G_q\to \pi_1^{orb}(Q)$ is injective.
\end{prop}


\begin{defn}
In an orbifold $Q$, a \emph{stratum} of type $(H)$ is the subspace of
$|Q|$ consisting of all points with local group isomorphic to $H$.
\end{defn}

\begin{prop}[\protect{\cite[Remark 1.23, and p. 9]{Davis}}]
\label{P:Strata_Complement}
Let  $Q_{(2)}$ denote the complement
of the strata of codimension $>2$ of an orbifold $Q$. Then $\pi_1^{orb}(Q)=\pi_1^{orb}(Q_{(2)})$.
\end{prop}

We state the proof of the following theorem as in \cite{Watts}, since it has an algorithm to recover the fundamental group of an orbifold from the fundamental group of the regular part.
\begin{thm}[\protect{\cite[Theorem 5.5]{Watts}}]
\label{T:Recovering_Fund_Group}
Let $Q$ be a connected orbifold. Then a presentation
for the orbifold fundamental group can be constructed using the topology, stratification,
and the orders of points in codimension $2$ strata.
\end{thm}

\begin{proof}[Proof \textrm{(See \cite{Watts})}]
Let $\hat{Q}$ be the differential subspace of $Q$ consisting of
codimension $0$ and codimension $1$ strata. Fix a base point $q$ in the codimension $0$ stratum.
Let $G$ be the (topological) fundamental group of $\hat{Q}$ with respect to $q$.
\begin{itemize}
\item[(1)] For each codimension $1$ stratum $S_i$, and for each homotopy class $\mu$ of paths starting
at $q$ and ending in $S_i$ attach a generator $\beta_{i, \mu}$ to $G$ with relation $\beta_{i, \mu}^2=1$. \\
\item[(2)] For each codimension $2$ stratum $T_j$ not in the closure of a codimension $1$ stratum,
let $\alpha_j$ be an element of $G$ represented by a loop starting at $q$ and going around $T_j$.
Then add the relation $\alpha_j^k=1$
 to $G$ where $k$ is the order of any point in $T_j$.\\
\item[(3)] For each codimension $2$ stratum $R$ in the closure of a codimension $1$ stratum, for each
 pair of codimension $1$ strata $S_i$, $S_i'$ with $R$ in their closures, and for each pair $\beta_{i, \mu}, \beta_{i', \mu'}$ (where $\mu\neq \mu'$) as constructed in Item (1) above, add the relation $(\beta_{i, \mu}\beta_{i', \mu'})^k=1$,
where $2k$ is the order of any point in $R$.
\end{itemize}
The resulting group is the orbifold fundamental group of $Q$.
\end{proof}

The structure of closed cohomogeneity one smooth orbifolds is given by the following theorem.

\begin{thm}[\protect{\cite[Theorem 2.1]{Gonzalez}\label{T:C1 orbifold structure}}]
Let $Q$ be a closed, connected, smooth orbifold with an (almost) effective
smooth action of a compact, connected Lie group $G$ with principal isotropy group $H$. If
the action is of cohomogeneity one, then the orbit space $Q/G$ is homeomorphic to a circle
or to a closed interval and the following statements hold.
\begin{itemize}
\item[(i)] If the orbit space is a circle, then $Q$ is equivariantly diffeomorphic  to a $G/H$-bundle
over a circle with structure group $N(H)/H$, where $N(H)$ is the normalizer of $H$ in
$G$. In particular, $Q$ is a manifold and its fundamental group is infinite.\\
\item[(ii)] If the orbit space is homeomorphic to an interval, say $[-1,1]$, then:
\begin{itemize}
\item[(a)] there are two non-principal orbits, $\pi^{-1}(\pm 1)=G/K^{\pm}$ where $\pi: Q\to Q/G$ is
the natural projection and $K^{\pm}$ is the isotropy group of the $G$-action at a point
in $\pi^{-1}(\pm 1)$.\\
\item[(b)] The orbifold singular set of $Q$ is either empty, a non-principal orbit or both
non-principal orbits.\\
\item[(c)] The orbifold $Q$ is equivariantly diffeomorphic to the union of two orbifiber bundles
over the two non-principal orbits whose fibers are cones over spherical space forms, that is,
\[
Q\simeq G\times_{K^-} C(\SP^-/\Gamma^-)\cup_{G/H}G\times_{K^+} C(\SP^+/\Gamma^+),
\]
where $S^{\pm}$ denotes the round sphere of dimension $(\dim~Q -\dim~G/K^{\pm} -1)$ and $\Gamma^{\pm}$
is a finite group acting freely and by isometries on $S^{\pm}$. The action is determined
by a group diagram $(G, H, K^-, K^+)$ with group inclusions $H\leq K^{\pm}\leq G$ and
where $K^{\pm}/H$ are spherical space forms $\SP^{\pm}/\Gamma^{\pm}$.\\
\item[(d)] Conversely, a group diagram $(G, H, K^-, K^+)$ with group inclusions $H\leq K^{\pm}\leq G$ and where $K^{\pm}/H$ are spherical space forms, determines a cohomogeneity one
orbifold as in part (c).
\end{itemize}
\end{itemize}
\end{thm}

Note that when the normal space of directions to a singular orbit is $\SP^1$, i.e. when the orbit is of codimension $2$, there are different orbifold structures on the underlying topological space corresponding to the diagram $(G, H, K^-, K^+)$. Namely, the local group at a point in this singular orbit can be trivial, or a finite cyclic group. For example, consider the group diagram $(\RS^1, 1, \RS^1, \RS^1)$ of a topological action  of $\RS^1$ on $\SP^2$. There are at least three orbifold structures on $\SP^2$: the usual smooth structure, the teardrop structure, and the rugby ball structure. 

When there is no   orbit of codimension $2$, or we choose the local group at a point on the orbit of codimension $2$ to be trivial, then we can obtain the orbifold fundamental group of the orbifold $Q$ given by the group diagram $(G, H, K^-, K^+)$ as follows:


\begin{prop}\label{P:Orbifold fundamental group}
Let $Q$ be a cohomogeneity one orbifold given by the group diagram $(G, H, K^-, K^+)$. 
\begin{itemize}
\item[(i)] If the normal spaces of directions at both singular orbits are not spheres, then  $\pi_1^{orb}(Q)=\pi_1(G/H)$.\\
\item[(ii)] If there is an orbit, say $G/K^-$, such that $K^-/H$  is a sphere, and the local group at a point in this orbit is trivial, then $\pi_1^{orb}(Q)=\pi_1(G/K^-)$. In particular, if $G$ is simply-connected, then $Q$ is a bad orbifold.
\end{itemize}
\end{prop}


\begin{proof}
We first prove part (i). If the normal spaces of directions at both singular orbits are not spheres, then the regular subset of the orbifold coincides with the union of the regular orbits. Since by assumption, the codimension of each of the singular orbits is at least $3$, we have $Q_{(2)}=R_{Q}$. Thus, by Proposition \ref{P:Strata_Complement}, $\pi_1^{orb}(Q)=\pi_1(R_Q)$. Since $R_Q=G/H\times (-1, 1)$, we have $\pi_1^{orb}(Q)=\pi_1(G/H)$.

Now we prove part (ii). In this case the regular part of the orbifold consists of the space of regular orbits and the singular orbit with sphere as its normal space of directions. The singular orbit with non-sphere normal space of directions is the other stratum of the orbifold, whose codimension is at least $3$. Therefore, $Q_{(2)}=R_Q=G\times_{K^-}D^-$. By the long exact sequence of homotopy groups of the disk bundle 
\[
\D^-\to G\times_{K^-}D^-\to G/K^-,
\]
we have $\pi_1^{orb}(Q)=\pi_1(G/K^-)$.
\end{proof}


\subsection{Cohomogeneity one orbifolds with exactly one fixed point.}
\label{SS: fixed point orbifolds}
In this subsection we prove that if the group acting on $Q$ is simply-connected and the action has exactly one fixed point, then there exists a good orbifold structure on the underlying topological space.

Let $G$ be a compact simply-connected Lie group which acts almost effectively, non-reducibly and with cohomogeneity one on a smooth simply-connected orbifold $Q$  with group diagram $(G, H, K, G)$. Assume that $K\subsetneq G$ so that the  action of $G$ on $Q$ has only one fixed point. Since the normal space of directions is a spherical space form, $G/H=\SP^{n-1}/\Gamma$, for $\Gamma$ a finite subgroup of  $\SO(n)$. On the other hand, since $G/H$ is a principal orbit, $G$ acts on $\SP^{n-1}/\Gamma$ almost effectively. Therefore, we can lift the almost effective action of $G$ on $\SP^{n-1}/\Gamma$  to an almost effective and transitive action of $G$ on $\SP^{n-1}$. These actions are well known and are given in Theorem~\ref{T:ASOH}.

Note that in the group diagram of the $G$-action on $Q$, the group  $K$ is a proper subgroup of $G$ which cannot be a finite extension of $H$ either, because there are no exceptional orbits. Thus we can rule out the actions \eqref{Eq_T_1}--\eqref{Eq_T_4} in Theorem~\ref{T:ASOH}, for in these cases $H$ is a maximal connected subgroup of $G$. Further, since we assume that the action is non-reducible, we rule out  cases  \eqref{Eq_T_5} and  \eqref{Eq_T_7}. in Theorem~\ref{T:ASOH}. Thus, we need only consider the actions \eqref{Eq_T_6},  \eqref{Eq_T_8} and \eqref{Eq_T_9} in Theorem~\ref{T:ASOH}. 

For each of the corresponding groups we will find the possible cohomogeneity one diagrams and explore the existence of a good orbifold structure on the underlying topological space determined by the diagram. 
Our strategy is as follows. Observe first that a cohomogeneity one action on an orbifold $Q$ lifts to a cohomogeneity one action on its universal orbicover. Therefore, if $Q$ is a good cohomogeneity one orbifold with finite orbifold fundamental group, then its universal orbicover $\tilde{Q}$ is a compact cohomogeneity one smooth manifold and, provided that $\dim\tilde{Q}\leq 7$, it must be one of those listed in \cite{Hoelscher}. In our case we will compute the orbifold fundamental group of $Q$. To do so, we need to pay especial attention to the local groups of $Q$. Then we will show that the underlying topological space $|Q|$ corresponds to some diagram given by the quotient  of a cohomogeneity one action on a manifold. 
\\


\noindent $\mathbf{\GG=Sp(n)}$. This case corresponds to action \eqref{Eq_T_6} in Theorem~\ref{T:ASOH}. We need the following Lemma:


\begin{lem}[\protect{\cite[Lemma 2.1]{TT}}]
\label{L:Sp(n)_SUBGPS}
Assume $n\geq 3$. Let $K$ be a closed connected subgroup of $\Sp(n)$
such that $\dim\Sp(n)/K\leq {4n-1}$. Then, up to an inner automorphism of $\Sp(n)$, $K$ coincides
with one of
\[
\Sp(n-1), \, \U(1)\times \Sp(n-1),\, \Sp(1)\times \Sp(n-1)\text{ or } \Sp(n),
\]
embedded in the standard way.
\end{lem}
 
Since $H_0\subsetneq K_0\subsetneq G$ and $H_0=\Sp(n-1)$, Lemma~\ref{L:Sp(n)_SUBGPS} implies that $K_0=\U(1)\times \Sp(n-1)$ or $\Sp(1)\times \Sp(n-1)$. Recall that $K\subset N_{G}(K_0)$. Since we have
\[
N_{\Sp(n)}(\U(1)\times \Sp(n-1))=N_{\Sp(1)}(\U(1))\times \Sp(n-1)\]
and  
\[
N_{\Sp(n)}(\Sp(1)\times \Sp(n-1))=\Sp(1)\times \Sp(n-1),
\]
we must have that $K$ is $N_{\Sp(1)}(\U(1))\times \Sp(n-1)$, $\Sp(1)\times \Sp(n-1)$ or $\U(1)\times \Sp(n-1)$. Recall now that $H$ is a finite extension of $H_0=\Sp(n-1)\subset \Sp(n)$. Using the fact that $H\subsetneq K$, we get the following diagrams:  

\begin{align}
	& (\Sp(n), \Sp(n-1)\times \Z_k, \Sp(n-1)\times \U(1), \Sp(n)), \label{D:ORBI_ITEM_1}	\\ 
	& (\Sp(n), \Sp(n-1)\times \D_{2m}^*, \Sp(n-1)\times N_{\Sp(1)}(\U(1)), \Sp(n)),\label{D:ORBI_ITEM_2}\\ 
	& (\Sp(n), \Sp(n-1)\times \Gamma, \Sp(n-1)\times \Sp(1), \Sp(n)), \text{ where $\Gamma$ is a finite subgroup of $\Sp(1)$}.\label{D:ORBI_ITEM_3} 
\end{align}
Now we explore the orbifold structures on the underlying topological space of each diagram.
\\


\noindent \textbf{Diagram \eqref{D:ORBI_ITEM_1}.}  Let  $\CP^{2n}=\SP^{4n+1}/S^1$, where $\SP^{4n+1}\subseteq \mathbb{C}\times \mathbb{H}^n$, $\mathbb{C}=\{a+ib \mid a,b \in \mathbb{R}\}\subseteq \mathbb{H}$, and $S^1\subseteq \mathbb{C}$ acts on $\SP^{4n+1}$ by right multiplication. Then we define the actions of  $\Sp(n)$ and $\Z_k$ on $\CP^{2n}$ as 
\begin{equation}\label{ACTION_SP(n_CPn)}
A\cdot[x_0, x_1, \ldots, x_n]=[x_0, A(x_1, \ldots, x_n)],\ A\in \Sp(n),
\end{equation} 
and 
\[
\lambda\cdot[x_0, x_1, \ldots, x_n]=[\lambda x_0, x_1, \ldots, x_n],\ \lambda\in \Z_k,
\]
 respectively. The action of $\Sp(n)$ is of cohomogeneity one (see \cite[Proposition~1.23]{Hoelscher}) and commutes with the $\Z_k$-action. Therefore,  $\Sp(n)$ acts on $\CP^{2n}/\Z_k$ with cohomogeneity one with group diagram \eqref{D:ORBI_ITEM_1} above. Hence by Proposition~\ref{P:Equivalent_actions}, $Q$ is equivariantly homeomorphic to $\CP^{2n}/\Z_k$. 
Note that in this case one of the normal space of directions is $\SP^{4n-1}/\Z_k$ and the other one is $\SP^1$. Therefore, we have a codimension $2$ stratum, namely, the orbit $G/K$. If we choose the local group on this orbit to be trivial, then by Proposition~\ref{P:Orbifold fundamental group}, $\pi_1^{orb}(Q)=\pi_1(G/K)=0$. Therefore, $\CP^{2n}/\Z_k$ admits a bad orbifold structure. If we choose the local group at points on $G/K$ to be $\Z_k$, then $\CP^{2n}$ is the universal orbicover of $\CP^{2n}/\Z_k$, giving that $Q$ is a good orbifold.
\\


\noindent \textbf{Diagram \eqref{D:ORBI_ITEM_2}.} We consider the cohomogeneity one action of $\Sp(n)$ on $\CP^{2n}$ defined in \eqref{ACTION_SP(n_CPn)}, and then define the action of $D^{*}_{2m}=\langle e^{\pi i/m}, j\rangle$ via
\[
e^{\pi i/m}\cdot[x_0, x_1, \ldots, x_n]= [e^{\pi i/m}x_0, x_1, \ldots, x_n],
\]
\[
j\cdot[x_0, x_1, \ldots, x_n]=[-jx_0j, x_1j, \ldots, x_nj].
\]
By the definition of $\Sp(n)$, the action of $\Sp(n)$ and the action of $D^*_{2m}$ on $\CP^{2n}$ commute. Thus we have a cohomogeneity one action of $\Sp(n)$ on $\CP^{2n}/D^{*}_{2m}$ with the same diagram as \eqref{D:ORBI_ITEM_2}. By Proposition~\ref{P:Equivalent_actions}, $Q$ is  equivariantly homeomorphic to $\CP^{2n}/D^{*}_{2m}$. Again, we can choose two different orbifold structures on $\CP^{2n}/D^{*}_{2m}$. First we let the local group at points on the orbit $G/K$ be trivial. Then by Proposition~\ref{P:Orbifold fundamental group}, $\pi_1^{orb}(Q)=\pi_1(G/K)=\Z_2$. Since $D^{*}_{2m}$ cannot be embedded in $\Z_2$, $\CP^{2n}/D^{*}_{2m}$ admits a bad orbifold structure by Proposition~\ref{P:Injection_of_Local_Gr}. Now, assume that the local group at the points on  the orbit $G/K$ is $\Z_m$. Then $\CP^{2n}$ is the universal orbicover of $\CP^{2n}/D^{*}_{2m}$.
\\


\noindent \textbf{Diagram \eqref{D:ORBI_ITEM_3}.} In this case, the local groups at both non-principal orbits are $\Gamma$, and there are no strata of codimension $2$.  Therefore, we only have one cohomogeneity one orbifold structure on the underlying topological  space of this diagram. Also, the regular subset of the orbifold coincides with the union of principal orbits. Hence, by Proposition~\ref{P:Orbifold fundamental group}, $\pi_1^{orb}(Q)=\pi_1(G/H)=\Gamma$. We consider the following two actions on $\HP^n$. For $(x_0: x_1: \ldots: x_n)\in\HP^n$, $A\in \Sp(n)$, and $\gamma\in \Gamma\subseteq \Sp(1)$,  define
\[
A\cdot(x_0: x_1: \ldots: x_n)= (x_0: A(x_1, \ldots, x_n)),
\]
\[
\gamma\cdot(x_0: x_1: \ldots: x_n)=(x_0\gamma: x_1: \ldots: x_n).
\]
These two actions commute, and  give us a cohomogeneity one action of $\Sp(n)$
on $\HP^n/\Gamma$ with group diagram \eqref{D:ORBI_ITEM_2}. Thus by Proposition~\ref{P:Equivalent_actions}, $Q$ is  equivariantly homeomorphic to  $\HP^n/\Gamma$. Since there is only one cohomogeneity one orbifold structure on $Q$, it is in fact diffeomorphic to $\mathbb{HP}^n/\Gamma$, which is a good orbifold.
\\


\noindent  $\mathbf{\GG=\SU(n)}$. This case corresponds to action \eqref{Eq_T_9} in Theorem~\ref{T:ASOH}. The only  connected proper  subgroup of $G$ containing $\SU(n-1)$ is $\mathrm{S}(\U(n-1)\U(1))$ (see, for example, \cite[Case II in p.~631]{GroveZiller}). 
Since $\mathrm{S}(\U(n-1)\U(1))$ is maximal in $G$, it does not have a finite extension in $G$. Therefore we only have the following group diagram:
\begin{align}\label{Eq_GD_SU(n)}
(\SU(n), \mathrm{S}(\U(n-1)\Z_k), \mathrm{S}(\U(n-1)\U(1)), \SU(n)).
\end{align}
We define  actions of $\SU(n)$ and $\Z_k$ on $\CP^n$ as follows:
\[
A\cdot(x_0: x_1:\dots: x_n)= (x_0: A(x_1, \ldots, x_n)),\, \, A\in \SU(n), \,(x_0: x_1: \cdots: x_n)\in \CP^n,
\]
\[
\lambda\cdot(x_0: x_1: \cdots: x_n)=(\lambda x_0: x_1: \cdots: x_n),\, \, \lambda\in \Z_k, \,(x_0: x_1: \cdots: x_n)\in \CP^n.
\]
Clearly, these two actions commute. As a result, we have a cohomogeneity one action on $\CP^n/\Z_k$ with group diagram \eqref{Eq_GD_SU(n)}. Therefore, by Proposition~\ref{P:Equivalent_actions}, $Q$ is equivariantly homeomorphic to $\CP^n/\Z_k$. Since the orbit $\SU(n)/\mathrm{S}(\U(n-1)U(1))$ is of codimension $2$, we can choose different orbifold structures on $\CP^n/\Z_k$. If we let the local group at this orbit  be trivial, then by Proposition~\ref{P:Orbifold fundamental group}, $\pi_1^{orb}(Q)=0$. Hence $\CP^n/\Z_k$ is a simply-connected orbifold. We now choose the local group to be $\Z_k$. Then $\CP^n$ is the universal orbicover of $\CP^n/\Z_k$, that is, in this case $\CP^n/\Z_k$ is a good orbifold.
\\


 \noindent $\mathbf{\GG=\Spin(9)}$. This case corresponds to action \eqref{Eq_T_8} in Theorem~\ref{T:ASOH}.  We have the following diagram:
 \begin{align}\label{Eq_GD_Spin(9)}
  (\Spin(9), N_{\Spin(9)}(\Spin(7)), \Spin(8), \Spin(9)).
  \end{align}
 Since the codimension of the non-principal orbits  of this action is not $2$, we have only one cohomogeneity one orbifold structure on the underlying topological  space. The local groups of both orbits are $\Z_2$ and, by Proposition~\ref{P:Orbifold fundamental group}, $\pi_1^{orb}(Q)=\pi_1(G/H)=\Z_2$. Now, we consider 
the following  action of $\Spin(9)$ on  the Cayley plane $\CaP^2$ (for more details see \cite{Iwata}). \\
Let $\mathfrak{J}$ be
the set of all $3\times 3$ Hermitian matrices over the Cayley number field $\mathbf{Ca}$. A matrix $A\in \mathfrak{J}$ has the form 
\[X(\xi, \mu)=\left( \begin{array}{ccc}
\xi_1 & u_3 & \bar{u}_2 \\
\bar{u}_3 & \xi_2 & u_1 \\
u_2 & \bar{u}_1 & \xi_3 \end{array} \right),\] 
where $\xi_i\in \R$ and $u_i\in \mathbf{Ca}$, for $i=1, 2, 3$. 

Let
\begin{align*}
& E_1=\left( \begin{array}{ccc}
1 & 0 & 0 \\
0 & 0 & 0 \\
0 & 0 & 0 \end{array} \right),
  &E_2=\left(\begin{array}{ccc}
0 & 0 & 0 \\
0 & 1 & 0 \\
0 & 0 & 0  \end{array} \right),
\qquad& E_3=\left(\begin{array}{ccc}
0 & 0 & 0 \\
0 & 0 & 0 \\
0 & 0 & 1  \end{array} \right), \\
&F_1^{u}=\left( \begin{array}{ccc}
0 & 0 & 0 \\
0 & 0 & \bar{u} \\
0 & u & 0 \end{array} \right),
 &F_2^{u}=\left(\begin{array}{ccc}
0 & 0 & \bar{u} \\
0 & 0 & 0 \\
u & 0 & 0 \end{array} \right), 
\qquad &F_3^{u}=\left( \begin{array}{ccc}
0 & u & 0\\
\bar{u} & 0 & 0 \\
0 & 0 & 0
\end{array} \right).
\end{align*}

Then, the set $\{E_1, E_2, E_3, F_1^{e_i}, F_2^{e_i}, F_3^{e_i} \mid i = 0, 1,\ldots, 7\}$ constitutes an $\mathbb{R}$-basis
of $\mathfrak{J}$, where $\{e_1,\ldots,e_7\}$ is the standard basis of $\mathbf{Ca}$. The
Jordan product  is defined on $\mathfrak{J}$ by
\[
X\circ Y=\frac{1}{2}[XY+YX], \quad X, Y \in  \mathfrak{J}.
\]
An $\mathbb{R}$-isomorphism $\varphi: \mathfrak{J}\to \mathfrak{J}$ is called an \emph{automorphism} of $\mathfrak{J}$, if it preserves the Jordan product, i.e.
\[
\varphi(X\circ Y)=\varphi(X)\circ \varphi(Y),
\]
for all $X, Y \in \mathfrak{J}$. It is well-known that the group of automorphisms of
$\mathfrak{J}$ is the exceptional Lie group $\F_4$ . The Cayley projective plane $\CaP^2$,
defined by
\[
\{X\in \mathfrak{J}|\quad X\circ X=X, \trace X=1\},
\] 
is identified with the left coset space $\F_4/\Spin(9)$, where
\[
\Spin(9)=\{\varphi\in \F_4 \mid  \varphi(E_1)=E_1\}.
\]
$\Spin(9)$ contains
\[
\Spin(8)=\{\varphi\in \F_4\mid \varphi(E_i)=E_i, \, i=1, 2, 3\}
\]
and $\Spin(8)$ contains
\[
\Spin(7)=\{\varphi\in \Spin(8)\mid \varphi(F^1_3)=F^1_3\}.
\]
Through the inclusion $\Spin(9)\subseteq \F_4$, $\Spin(9)$ acts on $\CaP^2$ with cohomogeneity one (see, for example \cite[Example~1]{Iwata}). The orbits of the action are given by 
\[
A_s=\{X(\xi, \mu)\in \CaP^2\mid \xi_1=s\},
\]
for $0\leq s\leq 1$. In particular,  the following hold:
\begin{itemize}
\item $A_1=\{E_1\}$ is a fixed point.
\item $A_0$ is an $8$-dimensional sphere. The isotropy group at $E_2$ is $\Spin(8)$.
\item $A_s$, for $0< s<1$, is the principal orbit which is a $15$-dimensional sphere. The isotropy group at $(E_1+E_2+F_3^1)$ is $\Spin(7)$.
\end{itemize}
Now define an $\mathbb{R}$-linear transformation $\sigma$ of $\mathfrak{J}$ by
\[
\sigma \left( \begin{array}{ccc}
\xi_1 & u_3 & \bar{u}_2 \\
\bar{u}_3 & \xi_2 & u_1 \\
u_2 & \bar{u}_1 & \xi_3 \end{array} \right)=\left( \begin{array}{ccc}
\xi_1 & -u_3 & -\bar{u}_2 \\
-\bar{u}_3 & \xi_2 & u_1 \\
-u_2 & \bar{u}_1 & \xi_3 \end{array} \right).
\]
Then $\sigma\in Z(\Spin(9))$ and $\sigma^2=1$ (see, for example, \cite[Section 2.9]{Yo}). 
Thus $\sigma$ is an involution commuting with the $\Spin(9)$ action on $\CaP^2$ described above. This gives a cohomogeneity one action of $\Spin(9)$ on $\CaP^2/\Z_2$ with group diagram \eqref{Eq_GD_Spin(9)}. Therefore, by Proposition~\ref{P:Equivalent_actions}, $Q$ is equivariantly homeomorphic to $\CaP^2/\Z_2$. Since $Q$ admits only one orbifold structure compatible with the cohomogeneity one action, $Q$ is indeed equivariantly diffeomorphic to $\CaP^2/\Z_2$, i.e. $Q$ is a good orbifold.


\subsection{Orbifold structures on non-primitive cohomogeneity one spaces}
We now  inquire whether or not  one can endow an underlying space of  a   non-primitive diagram whose primitive part admits a good structure, with a good structure.  The next proposition shows that it is indeed the case under certain mild restrictions. 


\begin{prop}
\label{P:Good_S_N_Primitive}
Let $(G, H, K^-, K^+)$ be the group diagram of a non-primitive cohomogeneity one action, where the underlying topological space $|Q|$ is homeomorphic to an orbifold. Let  $L\subset G$ be as in Definition \ref{Primitive Action}, and $(L, H, K^-, K^+)$ be the primitive diagram such that its underlying space is $L$-equivariantly homeomorphic to a global quotient, say $M/\Gamma$. Let $\tilde{L}$ be a covering group of $L$ which acts on $M$ as a lifting action of the $L$-action on $M/\Gamma$.  If $\tilde{L}$ is a subgroup of $G$, then $|Q|$ admits a good orbifold structure. 
 \end{prop}
 
 
 \begin{proof}
 By Proposition~\ref{P:Primitive Action}, $Q$ is equivariantly homeomorphic to 
$G\times_{L}(M/\Gamma)$. 
 Since $\tilde{L}$ is a covering group of $L$ which acts on $M$ as a lifting action of $L$-action on $M/\Gamma$, the actions of $\tilde{L}$ and $\Gamma$ commute. 
 Let $(\tilde{L}, \tilde{H},\tilde{K}^-, \tilde{K}^+)$ be the diagram of the $\tilde{L}$-action on $M$.
The cohomogeneity one manifold given by the group diagram $(G, \tilde{H},\tilde{K}^-, \tilde{K}^+)$ is equivariantly diffeomorphic to $G\times_{\tilde{L}}M$. Extend the action  of $\Gamma$ on $M$ to $G\times_{\tilde{L}}M$ by letting it act trivially on $G$. This is well-defined since the actions of $\tilde{L}$ and $\Gamma$ commute. We endow $G\times_{\tilde{L}}M$ with an invariant Riemannian metric and we assume that  $c(t)$  is a normal geodesic of $G\times_{\tilde{L}}M$ giving  diagram $(G, \tilde{H},\tilde{K^-}, \tilde{K^+})$. Then $\pi\circ c$ is a normal geodesic of $(G\times_{\tilde{L}}M)/\Gamma$  giving the diagram $(G, H, K^-, K^+)$, where $\pi$ is the canonical projection.   Therefore, $G\times_{L}(M/\Gamma)$ is equivariantly homeomorphic to $(G\times_{\tilde{L}}M)/\Gamma$.
 \end{proof}


\begin{cor}
\label{C:Good_S_N_Primitive_J}
Let $(G, H, K^-, K^+)$ be the group diagram of a non-primitive cohomogeneity one action, where the underlying topological space $|Q|$ is homeomorphic to an orbifold. Let  $L=L_1\times L_2\subset G$ be as in Definition \ref{Primitive Action}. If $L$ acts on  $(\SP^{n_1}/\Gamma_1)\ast (\SP^{n_2}/\Gamma_2)$ via the natural join action with group diagram $(L, H, K^-, K^+)$, where $L_i$ acts on $\SP^{n_i}$ and its action  commutes with the action of $\Gamma_i$, $i=1, 2$, then $|Q|$ admits a good orbifold structure.
 \end{cor}
 
 
 \begin{proof}
 By Proposition~\ref{P:Primitive Action}, $|Q|$ is equivariantly homeomorphic to 
$G\times_{L}((\SP^m/\Gamma_1)\ast (\SP^n/\Gamma_2))$. 
 Define the action of $\Gamma_1\times\Gamma_2$  on $\SP^{n_1}\ast\SP^{n_2}$ as the join action, and apply  Proposition \ref{P:Good_S_N_Primitive} with $\tilde{L}=L_1\times L_2$.
 \end{proof}

 We have a similar situation when there is a suspended action instead of the join action.


\begin{cor}\label{C:Good_S_N_Primitive_S}
Let $(G, H, K^-, K^+)$ be the group diagram of a non-primitive cohomogeneity one action, where the underlying topological space $|Q|$ is homeomorphic to an orbifold. Let  $L\subset G$ be as in Definition \ref{Primitive Action}. If $L$ acts on  $\Susp(\SP^n/\Gamma)$ by the suspended action with group diagram $(L, H, K^-, K^+)$, where $L$ acts on $\SP^n$ and its action  commutes with the action of $\Gamma$,  then $|Q|$ admits a good orbifold structure.
 \end{cor}
 
 
 \begin{proof}
By Proposition~\ref{P:Primitive Action}, $|Q|$ is equivariantly homeomorphic to 
$G\times_{L}\Susp(\SP^n/\Gamma)$.
Similarly, define the action of $\Gamma$ on $\Susp (\SP^n)$ as a suspended action and apply Proposition~\ref{P:Good_S_N_Primitive} with $\tilde{L}= L$.
 \end{proof}

 
\subsection{Proof of Theorem~\ref{T:GOOD_ORBIFOLDS}} In what follows we exploit the classification of cohomogeneity one actions on smooth manifolds in low dimensions (see \cite{Hoelscher}) to find possible smooth coverings for our orbifolds. Since the cohomogeneity one action on an orbifold lifts to a cohomogeneity one action on the orbifold universal covering, taking as departure point the information that we get from the group diagram about the orbifold  fundamental group and the local groups at non-principal points, we look for a finite  group acting on the  cohomogeneity one   smooth manifolds commuting with the acting group. In the cases that we have a representative of the group diagram and the action, we describe such actions; otherwise we use the following general argument.

 
 Let $(\tilde{G}, H, K^-, K^+)$ be a group diagram for a cohomogeneity one action on a smooth manifold $M$, and let $\Gamma\subseteq N=N_{\tilde{G}}(H)\cap N_{\tilde{G}}(K^-)\cap N_{\tilde{G}}(K^+)$ be a finite subgroup of $\tilde{G}$. The group $\Gamma$ acts on $M$ \emph{orbitwise}, i.e. $\gamma\cdot gL=g\gamma^{-1}L$, for $L=H, K^-, K^+$. This action is well-defined and  commutes with the $\tilde{G}$-action. Let $\rho\colon \tilde{G}\to G$ be a  group covering. If, in addition, $\{\Phi^g| \, g\in \ker~\rho\}\subseteq \Gamma$, where $\Phi^g$ is the action map, then  $G$ acts on $M/\Gamma$ with cohomogeneity one (cf. \cite[Proposition~3.2]{ZKA}). If $c(t)$ is a normal geodesic of the cohomogeneity one action on $M$ used to determine a group diagram, then $\pi\circ c$ is a normal geodesic for $M/\Gamma$. Thus we can find the group diagram of $M/\Gamma$ using the group diagram of $M$.


 \begin{rem}
 \label{REM:DEF_OF_N}
 In the sequel, we let $N=N_{\tilde{G}}(H)\cap N_{\tilde{G}}(K^-)\cap N_{\tilde{G}}(K^+)$, where $(\tilde{G}, H, K^-, K^+)$ is the group diagram of the smooth manifold under consideration. We also refer to the orbitwise action defined above only as the \emph{orbitwise action}, without mentioning  the explicit action.
 \end{rem}

 Now, we prove Theorem~\ref{T:GOOD_ORBIFOLDS} case by case. 
 
 \begin{rem}
 \label{REM:GOOD_ORBIFOLD_ACTIONS}
 Note that if the action is a join, suspension, product, or a non-primitive action (as in Corollaries~\ref{P:Good_S_N_Primitive}, \ref{C:Good_S_N_Primitive_J}, and  \ref{C:Good_S_N_Primitive_S}), or has a fixed point, then it is already clear  what the good structure should be. Therefore, we only explore a good structure on the spaces corresponding to actions not equivalent to these actions. 
\end{rem}

 

 \subsection*{Dimensions~$2$ and $3$.} By Perelman's Conical Neighborhood Theorem, and the fact that the only closed $1$-dimensional Alexandrov space is a circle, any $2$-dimensional Alexandrov space is a topological manifold. Therefore, the  $2$-sphere is the only closed, $2$-dimensional simply-connected Alexandrov space. In dimension $3$, the $3$-sphere and $\Susp(\RP^2)$ are the only closed, simply-connected Alexandrov spaces of cohomogeneity one (see \cite{GS}). Clearly, $\Susp(\RP^2)$ has a good orbifold structure.
 
 
 \subsection*{Dimension $4$}
  In dimension $4$, every space of directions is homeomorphic either to $\Susp(\RP^2)$ or to a spherical space form (see \cite{GGG} or \cite{HS}). Hence every $4$-dimensional Alexandrov space is homeomorphic to an orbifold. In this dimension, however, not every closed orbifold is good. For example, the so-called \emph{weighted complex projective spaces} are bad $4$-dimensional orbifolds (see \cite[p.~27]{Adem}). Closed, $4$-dimensional Alexandrov spaces of cohomogeneity one were equivariantly classified in \cite{GS}. In Table \ref{TB:GROUP_DGMS_D4} we have collected the diagrams corresponding to the simply connected spaces. Observe that every action in the table is a join, suspension or one-fixed-point action. Therefore, by Remark \ref{REM:GOOD_ORBIFOLD_ACTIONS} above, the underlying space of each diagram admits a good structure.

 \subsection*{Dimension~$5$}
 The only diagram that we should consider in dimension $5$ is
\begin{align}\label{D_Dim_5_GS}
(\RS^3\times \RS^1, \{(e^{\frac{lk+mps}{km}2\pi i}, e^{\frac{2\pi qsi}{k}})\}, (\mathbb{Z}_m\times 1)\K^-_0, \RS^3\times \mathbb{Z}_{k/q}),
\end{align}
where
$\K^-_0=\{(e^{ip\theta}, e^{iq\theta}) \mid \theta\in\mathbb{R}\}$ and $q|(m,k)$ and $(p, k)=1$.  In this case, the local groups at the points on $G/K^+$ are $\Z_m$, and we assume the local groups at the points on $G/K^-$ to be trivial. Thus, by Proposition~\ref{P:Orbifold fundamental group}, $\pi_1^{orb}(Q)=\pi_1(G/K^-)=\Z_m$.
Now consider the following diagram coming from the smooth classification (see \cite[Case $1_5A$]{Hoelscher}):
\begin{align}\label{D_Dim_5_M}
(\RS^3\times \RS^1, \Z_k, \{(e^{ip\theta}, e^{i\theta})\}, (\RS^3\times 1)\mathbb{Z}_{k}),
\end{align}
where $(p, k)=1$. The explicit description  of this action is as follows (see \cite[Diagram ($Q^5_C$) in p.\ 172]{Hoelscher}), where we consider $\SP^5$ as the unit  sphere in $\mathbb{H}\times \mathbb{C}$:
\begin{align*}
(\RS^3\times \RS^1)\times \SP^5&\to \SP^5\\
((g, z),(x, w))&\mapsto (gx\bar{z}^p, z^kw).
\end{align*}
Since in Diagram~\eqref{D_Dim_5_GS}, $(p, k)=1$, we can choose the same $p, k$ in Diagram~\eqref{D_Dim_5_M}. Let 
\[
\Gamma=\Z_m\times 1\subseteq \{e^{i\theta}\}\times \RS^1=N_G(H)\cap N_G(K^-)\cap N_G(K^+),
\]
 and let
\begin{align*}
\rho\colon \RS^3\times \RS^1&\to\RS^3\times \RS^1\\
(x, z)&\mapsto (x, z^q),
\end{align*}
 be a group covering of $\RS^3\times \RS^1$. Since $q|(m,k)$, it is easy to see that $\{\Phi^g \mid \, g\in \ker~\rho\}\subseteq \mathbb{Z}_m\times 1$. Therefore, $\RS^3\times \RS^1$ acts on $\SP^5/\Z_m$ as follows:
 \[(g, h)\cdot[(x, w)]=[(g\cdot x, \tilde{h}\cdot w)],\]
 where $(g, h)=\rho(g, \tilde{h})$.
  The isotropy groups of this action are precisely the ones in Diagram~\ref{D_Dim_5_GS}. Therefore, $Q$ is equivariantly diffeomorphic to $\SP^5/\Z_m$.

  
 \subsection*{Dimension~6} 
 In this dimension we have six cases to consider.
\\

 \noindent \textbf{Case 6.1.} For the diagram
 \[
 (\RS^3\times \RS^3, \Delta\RS^1\cup(j, j)\Delta\RS^1, \T^2\cup (j, j)\T^2, \Delta\RS^3),
 \]
  the local group at the points on $(\RS^3\times \RS^3)/\Delta\RS^3$ is $\Z_2$ and we let the local group at the points on the other non-principal orbit be trivial. Then,  by Proposition~\ref{P:Orbifold fundamental group}, $\pi_1^{orb}(Q)=\pi_1(G/K^-)=\Z_2$.
 Consider the counterpart diagram in the smooth classification (\cite[Case $2_6C$]{Hoelscher}, for $n=1$):
  \[
  (\RS^3\times \RS^3, \Delta\RS^1, \T^2, \Delta\RS^3).
  \]
  This action is equivalent to the natural action of $G=\SO(4)\subseteq \SO(5)$ on $M=\SO(5)/(\SO(2)\SO(3))$. 
  
   Let 
  \[
c(t)=\left( \begin{array}{ccccc}
\cos t & 0 & 0& 0 &\sin t \\
0 & 1 & 0& 0 &0 \\
0 & 0 & 1& 0 &0\\
0 & 0 & 0& 1&0\\
-\sin t & 0 & 0& 0 &\cos t
\end{array} \right),
\]
for $0\leq t\leq \pi/2$, and let $L=\SO(2)\SO(3)$. Then $c(t)L$ is a normal geodesic 
 in $M$  (see, for example, \cite[p.~88]{HoelscherThesis}) and the corresponding isotropy groups are as follows:
  \[
G_{c(0)L}=\left( \begin{array}{c|c|c}
\SO(2)  & 0 & 0 \\
\hline
0 & \SO(2) &0 \\
\hline
0 & 0 & 1
\end{array} \right),
\]
  \[
G_{c(\pi/4)L}=\left( \begin{array}{c|c|c}
\mathrm{Id} & 0 & 0 \\
\hline
0 & \SO(2) &0 \\
\hline
0 & 0 & 1
\end{array} \right),
\]
\[
G_{c(\pi/2)L}=\left( \begin{array}{ccccc}
a_{1} & 0 & a_2& a_3 &0 \\
0 & 1 & 0& 0 &0 \\
b_1 & 0 & b_2& b_3 &0\\
c_1 & 0 & c_2& c_3&0\\
0 & 0 & 0& 0 &1
\end{array} \right),
\]
where the latter is isomorphic to $\SO(3)$. Let 
\begin{equation}
\label{EQ:ORB_D6.1_SIGMA}
\sigma=\left( \begin{array}{ccccc}
-1 & 0 & 0& 0 &0 \\
0 & 1 & 0& 0 &0 \\
0& 0 & -1& 0 &0\\
0 & 0 & 0& 1&0\\
0 & 0 & 0& 0 &1
\end{array} \right).
\end{equation}
Then $\sigma^2=1$ and $\sigma\in N_G(G_{c(t)L})$, for all $t$. Hence $\sigma$ acts on $M$ orbitwise, fixing $G(c(\pi/2)L)$, and commutes with the $G$-action. As a result, $G$ acts on $M/\Z_2$ with 
\begin{align*}
	G_{[c(0)L]} & =G_{c(0)L}\cup \sigma G_{c(0)L},\\
	G_{[c(\pi/4)L]} & =G_{c(\pi/4)L}\cup \sigma G_{c(\pi/4)L},\\
	G_{[c(\pi/2)L]}& =G_{c(\pi/2)L}
\end{align*}
 as isotropy groups. Therefore $Q$ is equivariantly diffeomorphic to $M/\Z_2$.
\\


\noindent \textbf{Case 6.2.} This case corresponds to diagram
\begin{align}\label{D:Dim_6_GS_1}
(\RS^3\times \RS^3, \pm\Delta\RS^1\cup(j,\pm j)\Delta\RS^1, \T^2\cup (j, j)\T^2, \pm\Delta\RS^3).
\end{align}
Here, depending on the local group that we choose for the codimension $2$ orbit, we have different manifold diagrams. 

Assume first that the local group at the points of the codimension $2$ orbit is $\Z_2$. Then  
\[
\pi_1(R_Q)=\pi_1(G/H)=\Z_2\oplus \Z_2.
\] 
Thus, by Theorem~\ref{T:Recovering_Fund_Group}, $\pi_1^{orb}(Q)=\Z_2\oplus \Z_2$, since the local group of the codimension $2$ stratum is $\Z_2$. 

Now consider the diagram and the manifold in the Case 6.1. Let $\sigma$ be as in \eqref{EQ:ORB_D6.1_SIGMA} and let 
\[
\tau=\left( \begin{array}{ccccc}
-1 & 0 & 0& 0 &0 \\
0 & 1 & 0& 0 &0 \\
0& 0 & 1& 0 &0\\
0 & 0 & 0& 1&0\\
0 & 0 & 0& 0 &1
\end{array} \right).
\]
Then $\{id, \sigma, \tau,\sigma\tau\} = \Z_2\oplus \Z_2$.

Let $\tau$ act on $M$ by 
\[
\tau\cdot xc(t)L=\tau xc(t)\tau L, 
\]
where $x$ is an element of $\SO(5)$.
 Note that $\tau c(t)\tau L=c(t)L$ if and only if $t=0$. That is, $\sigma$ fixes the orbit $G(c(\pi/2)L)$ and $\tau$ fixes $G(c(0)L)$. The action of $\Z_2\oplus \Z_2$ on $M$ commutes with the action of $G$. Thus $G$ acts on $M/(\Z_2\oplus \Z_2)$ with the same diagram as \eqref{D:Dim_6_GS_1}. That is, $Q$ is equivariantly diffeomorphic to $M/(\Z_2\oplus \Z_2)$.
 
Now let the local group of the codimension $2$ orbit be trivial. Then by Proposition \ref{P:Orbifold fundamental group}, 
\[
\pi_1^{orb}(Q)=\pi_1(G/K^-)=\Z_2.
\]
Here, we consider the following diagram in the smooth classification (\cite[Table~F]{GWZ}):
\[
(\SO(4), \Z_2\SO(2), \SO(2)\SO(2), \OO(3)).
\] 
This is the diagram of  the natural action of $\SO(4)$ on $\CP^3=\SU(4)/\U(3)$. 
Let
\[
\alpha=\left( \begin{array}{cccc}
1 & 0 & 0& 0  \\
0 & -1 & 0& 0  \\
0& 0 & 1& 0 \\
 0 & 0& 0 &-1
\end{array} \right)
\]
and note that $\alpha\in N_G(H)\cap N_G(K^-)\cap N_G(K^+)$.
The element $\alpha$ acts orbitwise on $\CP^3$ and its action commutes with the $\SO(4)$-action. Since  $\alpha\in \OO(3)$, it fixes the orbit $\SO(4)/\OO(3)\approx \RP^3$ and it does not fix any other point (see \cite{Masuda}). Consequently, $\SO(4)$ acts on $\CP^3/\Z_2$ with the same diagram as \eqref{D:Dim_6_GS_1}. Therefore, $Q$ is equivariantly diffeomorphic to $\CP^3/\Z_2$. 
 
 Note that this argument in particular shows that $(\SO(5)/\SO(2)\SO(3))/(\Z_2\oplus \Z_2)$ is homeomorphic to $\CP^3/\Z_2$, since they are both the underlying spaces of the same diagram. 
 \\

 
 \noindent \textbf{Case 6.3.} This case corresponds to  diagram
 \begin{align}\label{D:Dim_6_GS_2}
 (\RS^3\times \RS^3, \D^*_{2m}\times\RS^1, \N_{\RS^3}(\RS^1)\times \RS^1, \RS^3\times \RS^1).
 \end{align}
 If we let the local group of the codimension $2$ orbit be trivial, then by Proposition~\ref{P:Orbifold fundamental group}, $\pi_1^{orb}(Q)=\Z_2$. The local group of the other singular orbit is $D^*_{2m}$, which gives rise to a bad structure by Proposition~\ref{P:Injection_of_Local_Gr}.  Now assume that the local group of the codimension $2$ orbit is $\Z_m$. The  action of $\RS^3\times\RS^3$ on $\SP^2\times\CP^2$ given by
 \begin{align*}
 (\RS^3\times\RS^3)\times (\SP^2\times\CP^2)&\to \SP^2\times\CP^2\\
 (g, h)\cdot(x, y)&= (gxg^{-1}, hy).
 \end{align*}
  yields the smooth  diagram we need, i.e. 
 \[
  (\RS^3\times \RS^3, 1\times\RS^1, \RS^1\times \RS^1, \RS^3\times \RS^1).
  \]
 
 We let $D^*_{2m}=\langle e^{\frac{\pi}{m}i}, j\rangle$ act on $\CP^2$ via
 \[
 e^{\frac{\pi}{m}i}\cdot(z_0: z_1: z_2)=(e^{\frac{\pi}{m}i} z_0: z_1: z_2),
 \]
 \[
 j\cdot(z_0: z_1: z_2)=(\bar{z}_0: -\bar{z}_2: \bar{z}_1).
 \]
 This action commutes with the cohomogeneity one action of $\RS^3\approx\left( \begin{array}{c|c}
1 & 0   \\
\hline
0 &\SU(2) 
\end{array} \right)$
on $\CP^2$ and fixes only the  orbit with isotropy group $\RS^3$. If we take the action of $D^*_{2m}$ on $\SP^2$ to be trivial, then the actions of $\RS^3\times\RS^3$ and $D^*_{2m}$ commute. Therefore, $\RS^3\times\RS^3$ acts on $\SP^2\times(\CP^2/D^*_{2m})$ with the same diagram as \eqref{D:Dim_6_GS_2}. 
 \\
 

 
 \noindent \textbf{Case 6.4.} In this case, we have the  non-primitive diagram
 \begin{align}\label{D:Dim_6_GS_3}
 (\RS^3\times \RS^3, \{(e^{ip\theta}\lambda, e^{i\theta})\mid \theta\in \mathbb{R}, \lambda\in\mathbb{Z}_k\}, \T^2, \RS^3\times \RS^1).
\end{align}
The corresponding primitive diagram is 
\begin{align}\label{D:Dim_6_GS_Non_P_1}
(\RS^3\times \RS^1, \{(e^{ip\theta}\lambda, e^{i\theta})\mid \theta\in \mathbb{R}, \lambda\in\mathbb{Z}_k\}, \T^2, \RS^3\times \RS^1),
\end{align}
which corresponds to an action of $\RS^3\times \RS^1$ on $\CP^2/\Z_k$. More precisely, $\RS^3\times \RS^1$ acts on $\CP^2$ as follows (see \cite[Subsection~3.4]{H}):
 \begin{align*}
 (\SU(2)\times\RS^1)\times \CP^2&\to \CP^2\\
 (A, z)\cdot (z_0: z_1: z_2)&\mapsto \diag(1, z^p A)(z_0: z_1: z_2).
 \end{align*}
 Let $\Z_k\subseteq \RS^1$ act on $\CP^2$ by 
 \[
 \gamma\cdot(z_0: z_1: z_2)=(\gamma z_0: z_1: z_2).
 \]
 The group $\Z_k$ fixes both a point, $(z: 0: 0)$, and $\CP^1=\{(0: z_1: z_2)\}$ (cf. \cite[Chapter VII, Section 3]{Bredon}), which are indeed the non-principal orbits of the cohomogeneity one action of $\RS^3\times\RS^1$. Further, the action of $\Z_k$ commutes with the action of $\RS^3\times \RS^1$. Therefore, $\RS^3\times \RS^1$ acts on $\CP^2/\Z_k$ with the same diagram as \eqref{D:Dim_6_GS_Non_P_1}. As a result, the underlying space, $|Q|$, of Diagram~\ref{D:Dim_6_GS_3}, is equivariantly homeomorphic to the total space of a $(\CP^2/\Z_k)$-bundle over $\SP^2$. Thus by Proposition~\ref{P:Good_S_N_Primitive}, $|Q|$ admits a good structure. In fact, if we choose the local group at the codimension $2$ orbit to be $\Z_k$, 
then by Theorem~\ref{T:Recovering_Fund_Group}, $\pi_1^{orb}(Q)=\Z_k$, and the good structure is just the natural good structure on 
\[ 
(\RS^3\times \RS^3)\times_{(\SU(2)\times\RS^1)}(\CP^2/\Z_k).
\]
 Note that if
   we let the local group at the codimension $2$ orbit  be trivial, then by Proposition~\ref{P:Orbifold fundamental group}, $\pi_1^{orb}(Q)=0$ since $\K^-=\T^2$ is connected. In this case, $Q$ is a bad orbifold.  
 \\

 
 \noindent \textbf{Case 6.5.} The diagram of this case is also a non-primitive diagram:
 \begin{align}\label{D:Dim_6_GS_4}
 (\RS^3\times \RS^3, \{(e^{ip\theta}\lambda, e^{i\theta})\}, \RS^3\times \RS^1, \RS^3\times \RS^1).
\end{align}
The corresponding primitive diagram is 
 \begin{align}\label{D:Dim_6_GS_Non_P_2}
 (\RS^3\times \RS^1, \{(e^{ip\theta}\lambda, e^{i\theta})\}, \RS^3\times \RS^1, \RS^3\times \RS^1),
\end{align}
which is equivalent to an action of $\RS^3\times \RS^1$ on $\SP^4/\Z_k$. Namely,  let $\Psi\colon\RS^3\times\RS^3\to \SO(4)\subseteq \SO(5)$ be the double cover of $\SO(4)$ and define the action of $\RS^3\times \RS^1$ on $\SP^4$ as follows (see \cite[Subsection~3.5]{H}):
\begin{align*}
&\RS^3\times \RS^1\times \SP^4\to \SP^4\\
& ((g, z), x)\mapsto \diag(1, \Psi(g, z^p))x.
 \end{align*}
 Let $\Z_k=1\times \Z_k\subseteq \RS^3\times\RS^1$ act  on $\SP^4$ via $(1, h)\cdot x=\diag(1, \Psi(1, h))x$. This action commutes with the $(\RS^3\times \RS^1)$-action and fixes  the non-principal orbits, which are two points. Therefore, $\RS^3\times \RS^1$ acts on $\SP^4/\Z_k$ with the same diagram as \eqref{D:Dim_6_GS_Non_P_2}. This means that the underlying space, $|Q|$, of  Diagram \eqref{D:Dim_6_GS_4} is equivariantly homeomorphic to the total space of an $(\SP^4/\Z_k)$-bundle over $\SP^2$. Thus, by Proposition~\ref{P:Good_S_N_Primitive}, $Q$ is a good orbifold and $\pi_1^{orb}(Q)=\Z_k$, by Proposition~\ref{P:Orbifold fundamental group}. 
 \\

  \noindent \textbf{Case 6.6.} In this case we have the following group diagram:
 \begin{align}\label{D:Dim_6_GS_5}
 (\RS^3\times \RS^3, \{(e^{i\theta}\lambda, e^{i\theta})\}, \RS^1\times \RS^3, \RS^3\times \RS^1).
 \end{align}
 We will show that $|Q|$ is homeomorphic to $\CP^3/\Z_k$. 
 
 Consider  the  cohomogeneity one action of $\RS^3\times \RS^3\approx \SU(2)\SU(2)\subseteq \SU(4)$ on $\CP^3$ (see \cite[Page~172, $Q^6_D$]{Hoelscher}) with  group diagram
 \begin{align}\label{D:Dim_6_Manifold_1}
 (\RS^3\times \RS^3, \Delta \RS^1, \RS^1\times \RS^3, \RS^3\times \RS^1).
 \end{align}
Now define an action  of $\Z_k\subseteq \RS^1$ on $\CP^3$ as follows:
\[
\lambda\cdot(z_0: z_1: z_2: z_3)=(\lambda z_0: \lambda z_1: z_2: z_3).
\]
Note that $\mathrm{Fixed}(\CP^3, \Z_k)=\{(z_0: z_1 : 0 : 0)\}\sqcup\{(0: 0 : z_2 : z_3)\}=\CP^1\sqcup\CP^1$ (see \cite[Chapter VII, Section 3]{Bredon}), which correspond to the non-principal orbits. Moreover, the action commutes with the $(\RS^3\times \RS^3)$-action on $\CP^3$, which consequently gives an $(\RS^3\times \RS^3)$-action on $\CP^3/\Z_k$ with the same diagram as \eqref{D:Dim_6_GS_5}. Therefore, $|Q|$ is homeomorphic to $\CP^3/\Z_k$. In this case, we do not have a codimension $2$ orbit and the local groups at both non-principal orbits are $\Z_k$. Thus by Proposition~\ref{P:Orbifold fundamental group},  $\pi_1^{orb}(Q)=\Z_k$. Therefore, the smooth manifold given by the diagram \eqref{D:Dim_6_Manifold_1}, i.e.\ $\CP^3$, is its universal orbicover. Hence,  $Q$ is equivariantly diffeomorphic to $\CP^3/\Z_k$ with the natural good structure.


\subsection*{Dimension $7$.}  Unlike dimension $6$, in which there exists a diffeomorphism type classification  of closed, smooth, simply-connected cohomogeneity one manifolds, there  is no such classification in dimension $7$. Therefore, here we do not give an explicit description  of smooth universal covers of our good orbifolds. In fact, in most cases, we only use the smooth \emph{equivariant} classification and show that a finite subgroup of $N_G(H)\cap N_G(K^-)\cap N_G(K^+)$ acts on the smooth manifold given by the diagram  $(G, H, K^-, K^+)$, commutes with the given $G$-action, and has the \emph{desired} fixed point set.
\\

  \noindent \textbf{Case 7.1.} In this case we have the following  non-primitive  diagram:
 \begin{align}\label{D:Dim_7_GS_1}
(\RS^3\times \RS^3, \{(e^{\frac{lk+mps}{km}2\pi i}, e^{\frac{2\pi qsi}{k}})\}, (\mathbb{Z}_m\times 1)\K^-_0, \RS^3\times \mathbb{Z}_{k/(k,q)}),
 \end{align}
 where $(m, q)=(k, q)$. The corresponding primitive diagram is 
  \begin{align}\label{D:Dim_7_GS_Non_P_1}
(\RS^3\times \RS^1, \{(e^{\frac{lk+mps}{km}2\pi i}, e^{\frac{2\pi qsi}{k}})\}, (\mathbb{Z}_m\times 1)\K^-_0, \RS^3\times \mathbb{Z}_{k/(k,q)}).
 \end{align}
Consider the  diagram
\begin{equation}\label{D:Dim_7_Manifold_1}
(\RS^3\times \RS^3, \{(e^{\frac{2\pi pli}{k}}, e^{\frac{2\pi qli}{k}})\}, \{(e^{p\theta i}, e^{q\theta i})\}, \RS^3\times \mathbb{Z}_k),
\end{equation}
from the smooth classification \cite[Case $1_7B$, $N_C^7$]{Hoelscher}, with $(k, q)=1$. This is also a non-primitive diagram  with
 \begin{equation} \label{D:Dim_7_Manifold_Non_P_1}
 (\RS^3\times \RS^1, \{(e^{\frac{2\pi pli}{k}}, e^{\frac{2\pi qli}{k}})\}, \{(e^{p\theta i}, e^{q\theta i})\}, \RS^3\times \mathbb{Z}_k),
 \end{equation}
 as its primitive diagram. This primitive diagram corresponds to  an $(\RS^3\times \RS^1)$-action on $\SP^5/\Z_q$.

Let the local group at the orbit of codimension $2$ be trivial. Then, by Proposition~\ref{P:Orbifold fundamental group}, 
 \[
 \pi_1^{orb}(Q)=\pi_1(G/K^-)=\Z_{m/(m, q)}.
 \]
Thus $Q$ is good if and only if $(m, q)=1$, as the local group $\Z_m$ should inject into the orbifold fundamental group by Proposition~\ref{P:Injection_of_Local_Gr}. Indeed, if $(m, q)=1=(q, k)$, we can choose Diagram~\eqref{D:Dim_7_Manifold_Non_P_1} with the same parameters as \eqref{D:Dim_7_GS_Non_P_1}. Let 
\[
\Gamma=\Z_m\times 1\subseteq \RS^1\times\RS^1\subseteq N_G(H)\cap N_G(K^-)\cap N_G(K^+).
\]
Then $\Gamma$ acts orbitwise on $\SP^5/\Z_q$ and commutes with the  $(\RS^3\times \RS^1)$-action on $\SP^5/\Z_q$. 
Therefore,  $\RS^3\times \RS^1$ acts on $(\SP^5/\Z_q)/\Z_m$ with the same diagram as \eqref{D:Dim_7_GS_Non_P_1}. By Proposition~\ref{P:Good_S_N_Primitive}, $Q$ is a good orbifold.  

Now suppose that the local group at the orbit of codimension $2$ is $\Z_k$. Let 
\begin{align} \label{D:Dim_7_Manifold_Non_P_1_Special}
(G, H, K^-, K^+) = (\RS^3\times \RS^1, 1, \{(e^{p\theta i}, e^{q\theta i})\}, \RS^3\times 1) 
 \end{align}
and let  $\Gamma =\{(e^{\frac{lk+mps}{km}2\pi i}, e^{\frac{2\pi qsi}{k}})\}$. Note that
\begin{align*}
\Gamma 
		& =\Z_m\Z_k\\
		& \subseteq \RS^1\times\RS^1\\
		&\subseteq N_G(H)\cap N_G(K^-)\cap N_G(K^+).
\end{align*}
 Hence $\Gamma$ acts orbitwise on the manifold given by $(G, H, K^-, K^+)$ and commutes with the $G$-action.
 Thus $G$ acts on $(\SP^5/\Z_q)/(\Z_m\Z_k)$ with the same diagram as \eqref{D:Dim_7_GS_Non_P_1}. By Proposition~\ref{P:Good_S_N_Primitive}, $Q$  admits a good structure. 
 \\

 
 \noindent \textbf{Case 7.2.} The diagram for this case is 
 \begin{align}\label{D:Dim_7_GS_2}
 (\RS^3\times \RS^3, \Delta\mathbb{Z}_k, \RS^1\times \mathbb{Z}_k, \Delta\RS^3).
 \end{align}
 
Let the local group of the  codimension $2$ orbit be trivial. Then, by Proposition~\ref{P:Orbifold fundamental group}, $\pi^{orb}_1(Q)=\Z_k$. Consider the  diagram
\begin{align}\label{D:Dim_7_Manifold_2}
 (\RS^3\times \RS^3, \Z_n, \{(e^{ip\theta}, e^{iq\theta})\}, \Z_n\Delta\RS^3),
 \end{align}
 of the smooth classification (see \cite[Case $1_7B$, $P_D^7$]{Hoelscher}),
 with $n=1$ and $p, q$ arbitrary, or $n=2$ and $p$ even or $q$ even (but not both). In our case we assume $n=1$ and $q=0$.  Recall that 
 $N=N_{G}(H)\cap N_{G}(K^-)\cap N_{G}(K^+)$, where  $(G, H, K^-, K^+)$ is the group diagram of the smooth manifold we are currently considering. For $\Gamma=\Delta \Z_k\subseteq N$, we let $\Gamma$ act on $M$ orbitwise. Then we get an action on $M/\Gamma$ giving rise to  Diagram  \eqref{D:Dim_7_GS_2}. Therefore, $Q$ admits  a good structure. Note that if the local group of the orbit of codimension $2$  is $\Z_l$, where $(l, k)=1$, then by Proposition~\ref{P:Orbifold fundamental group}, $Q$ is simply-connected and, therefore, bad.
 \\

 
 \noindent \textbf{Case 7.3.} Here we consider the following diagram:
 \begin{align}\label{D:Dim_7_GS_3}
 (\RS^3\times \RS^3, \Delta D^*_{2m}, (\RS^1\times 1)\Delta D^*_{2m}, \Delta\RS^3).
 \end{align}
 If we let the local group of the orbit of codimension $2$ be trivial, then by Proposition~\ref{P:Orbifold fundamental group}, $\pi^{orb}_1(Q)=D_{2m}^*$. For  Diagram~\eqref{D:Dim_7_Manifold_2}, we let $\Delta D_{2m}^*\subseteq N$ act on $M$ orbitwise. Then $\RS^3\times \RS^3$ acts on $M/D_{2m}^*$ with the same diagram as \eqref{D:Dim_7_GS_3}. Thus, $Q$ is good. 
 \\

 
  \noindent \textbf{Case 7.4.} In this case we have  diagram
 \begin{align}\label{D:Dim_7_GS_4}
 (\RS^3\times \RS^3, \Delta\mathbb{Z}_k, \{(e^{ip\theta}, e^{iq\theta})\}, \Delta\RS^3)
 \end{align}
 with  $k|(p-q)$ and, if $k$ is even, then $p$ and $q$ are odd.
Suppose first that the local group at the codimension $2$ orbit is trivial. Then by Proposition~\ref{P:Orbifold fundamental group}, $\pi_1^{orb}(Q)=0$, i.e. $Q$ is simply-connected and is  therefore bad. 
Now let  the local group be $\Z_k$. Then $\pi_1(R_Q)=\pi_1(G/H)=\Z_k$, which gives, in particular, that $\pi_1^{orb}(Q)=\Z_k$ by Proposition~\ref{T:Recovering_Fund_Group}. Consider the smooth  diagram~\eqref{D:Dim_7_Manifold_2} above with the same $p, q$ as diagram~\eqref{D:Dim_7_GS_4}. Let $\Gamma=\Delta Z_k\subseteq N\cap \{(e^{ip\theta}, e^{iq\theta})\}$ act orbitwise on $M$. Then $\RS^3\times \RS^3$ acts on $M/\Gamma$ giving the same diagram as \eqref{D:Dim_7_GS_4}. Hence, $Q$ is a good orbifold. 
 \\

 
 \noindent \textbf{Case 7.5.} We have the diagram
 \begin{align}\label{D:Dim_7_GS_5}
 (\RS^3\times \RS^3, \Delta\mathbb{Z}_k, \{(e^{ip\theta}, e^{-iq\theta})\}, \Delta\RS^3)
 \end{align}
 with  $k|(p+q)$ and, if $k$ is even, then $p$ and $q$ are odd. The same argument as in the preceding case gives the result.
 \\


 
 \noindent \textbf{Case 7.6.} In this case we consider the following diagram:
 \begin{align}\label{D:Dim_7_GS_6}
(\RS^3\times \RS^3, \Delta D^*_{2m}, \{(e^{ip\theta}, e^{iq\theta})\}\cup \{(je^{ip\theta}, je^{iq\theta})\}, \Delta\RS^3),
 \end{align}
 where  $m|(p-q)$ and, if $m$ is even, then $p$ and $q$ are odd. Consider Diagram~\eqref{D:Dim_7_Manifold_2} with $n=1$. If the local group at the codimension $2$ orbit is trivial, then, by Proposition~\ref{P:Orbifold fundamental group}, $\pi_1^{orb}(Q)=\Z_2$, and hence $Q$ is a bad orbifold since $D_{2m}^*$ does not inject into $\Z_2$. Assume then that the local group at the codimension $2$ orbit is $\Z_m$. Choose $\Gamma=\Delta D_{2m}^*=\langle ((e^{\frac{\pi ip}{m}}, e^{\frac{\pi iq}{m}}), (j, j))\rangle\subseteq N$, and act on $M$ orbitwise. Then we have a $\RS^3\times \RS^3$-action on $M/\Gamma$ with the same diagram as \eqref{D:Dim_7_GS_5}. Thus $Q$ is a good orbifold.
 \\
 
 
  \noindent \textbf{Case 7.7.} Here we consider the diagram
 \begin{align}\label{D:Dim_7_GS_7}
(\RS^3\times \RS^3, \Delta\D^*_{2m}, \{(e^{ip\theta}, e^{-iq\theta})\}\cup \{(je^{ip\theta}, je^{-iq\theta})\}, \Delta\RS^3)
 \end{align}
with  $m|(p+q)$ and, if $m$ is even, then $p$ and $q$ are odd. The same argument as above gives the result.
  \\

  
   \noindent \textbf{Case 7.8.} We now explore  good structures on the following diagram:
 \begin{align}\label{D:Dim_7_GS_8}
(\RS^3\times \RS^3, \pm\Delta\mathbb{Z}_k, \RS^1\times \mathbb{Z}_k, \pm\Delta\RS^3). 
 \end{align}
 If the local group at the orbit of codimension $2$  is trivial, then by Proposition~\ref{P:Orbifold fundamental group}, $\pi_1^{orb}(Q)=\Z_k$.  Consider Diagram~\eqref{D:Dim_7_Manifold_2} with $n=2$ and $q=0$, and choose $\Gamma=\Delta \Z_k\subseteq N$, which acts orbitwise on $M$. Then $\RS^3\times \RS^3$ acts on $M/\Gamma$ giving the same diagram as \eqref{D:Dim_7_GS_8}. Therefore $Q$ is a good orbifold. 
 
 If the orbit of codimension $2$ has $\Z_2$ as its local group,  consider Diagram~\eqref{D:Dim_7_Manifold_2} with $n=1$ and $q=0$, and choose  $\Gamma=\Delta\Z_k\cdot\langle (-1, 1)\rangle \subseteq N$. Thus $\RS^3\times \RS^3$ acts on $M/\Gamma$ with the same diagram as \eqref{D:Dim_7_GS_8}, i.e. $Q$ is good. 
 \\
 
 
    \noindent \textbf{Case 7.9.} We consider the following diagram:
 \begin{align}\label{D:Dim_7_GS_9}
 (\RS^3\times \RS^3, \pm\Delta\D^*_{2m}, (\RS^1\times 1)\D^*_{2m}, \pm\Delta\RS^3). 
 \end{align}
 If the local group at the orbit of codimension $2$  is trivial, then by Proposition~\ref{P:Orbifold fundamental group}, $\pi_1^{orb}(Q)=D^*_{2m}$.  Consider Diagram~\eqref{D:Dim_7_Manifold_2} with $n=2$ and $q=0$, and choose $\Gamma=D^*_{2m}\subseteq N$, which acts orbitwise on $M$. Then, $\RS^3\times \RS^3$ acts on $M/\Gamma$ giving the same diagram as \eqref{D:Dim_7_GS_8}. Therefore $Q$ is a good orbifold. 
 
 If the orbit of codimension $2$ has $\Z_2$ as its local group, we  consider Diagram~\eqref{D:Dim_7_Manifold_2} with $n=1$ and $q=0$, and choose  $\Gamma=\Delta D^*_{2m}.\langle (-1, 1)\rangle \subseteq N$. Thus $\RS^3\times \RS^3$ acts on $M/\Gamma$ with the same diagram as \eqref{D:Dim_7_GS_8}, i.e. $Q$ is good. 
 \\

 
  \noindent \textbf{Case 7.10.} The  diagram in this case is
 \begin{align}\label{D:Dim_7_GS_10}
 (\RS^3\times \RS^3, \pm\Delta\mathbb{Z}_k, \{(e^{ip\theta}, e^{iq\theta})\}, \pm\Delta\RS^3), 
 \end{align}
 where $p$ is even, $q$ is odd, and consequently $k$ is odd.  
 If the local group  at the orbit of codimension $2$  is trivial, then by Proposition~\ref{P:Orbifold fundamental group}, $\pi_1^{orb}(Q)=0$, that is, $Q$ is simply-connected and, therefore, bad. Assume that the local group is $\Z_k$. Then choose $\Delta \Z_k\subseteq N\cap \{(e^{ip\theta}, e^{iq\theta})\}$, which acts on $M$ orbitwise. As a result, $\RS^3\times \RS^3$ acts on $M/\Gamma$ giving Diagram~\eqref{D:Dim_7_GS_10}. 
 
 Note that if $q=p+1$, then the universal cover is the Eschenburg space 
\[ 
E^7_p=\diag(z, z, z^p)\setminus\SU(3)/\diag(1, 1,  \bar{z}^{p+2}),
\]
where $\SU(2)\times\SU(2)$ acts on $E^7_p$ with the first factor acting on the left and the
second on the right, both as the upper $\SU(2)$ block in $\SU(3)$ (see \cite[Page~175]{H}).
\\


 \noindent \textbf{Case 7.11.} The diagram here is
 \begin{align}\label{D:Dim_7_GS_11}
 (\RS^3\times \RS^3, \pm\Delta\mathbb{Z}_k, \{(e^{ip\theta}, e^{-iq\theta})\}, \pm\Delta\RS^3), 
 \end{align}
  where $p$ is even, $q$ is odd, and consequently $k$ is odd. This diagram is similar to Diagram \eqref{D:Dim_7_GS_10} in the previous case, and the results follow from a similar argument. 
\\

  
   \noindent \textbf{Case 7.12.} We consider the diagram
 \begin{align}\label{D:Dim_7_GS_12}
 (\RS^3\times \RS^3, \pm\Delta D^*_{2m}, \{(e^{ip\theta}, e^{iq\theta})\}\cup \{(je^{ip\theta}, je^{iq\theta})\}, \pm\Delta\RS^3)
 \end{align}
 with $p$ even, $q$  odd, and $m$ odd. 
If the local group  at the orbit of codimension $2$  is trivial, then by Proposition~\ref{P:Orbifold fundamental group}, $\pi_1^{orb}(Q)=\Z_2$. Since  $\pm\Delta D^*_{2m}$ does not inject into $\Z_2$, the orbifold $Q$ is bad.

 Assume the local group is $\Z_m$. Then choose $\Delta D^*_{2m}\subset N\cap\{(e^{ip\theta}, e^{iq\theta})\}\cup \{(je^{ip\theta}, je^{iq\theta})\}$ which acts on $M$ orbitwise. As a result, $\RS^3\times \RS^3$ acts on $M/\Gamma$ giving Diagram~\eqref{D:Dim_7_GS_12}. 
 
As in the Case 7.10, if $q=p+1$, then the universal cover of $Q$ is the Eschenburg space 
\[
 E^7_p=\diag(z, z, z^p)\setminus\SU(3)/\diag(1, 1,  \bar{z}^{p+2}).
 \]


\noindent \textbf{Case 7.13.} This case is similar to Case 7.12, and the same argument gives the result. The diagram under consideration is
 \begin{align}\label{D:Dim_7_GS_13}
 (\RS^3\times \RS^3, \pm\Delta\D^*_{2m}, \{(e^{ip\theta}, e^{-iq\theta})\}\cup \{(je^{ip\theta}, je^{-iq\theta})\}, \pm\Delta\RS^3), 
 \end{align}
  where $p$ is even, $q$ is odd, and consequently $m$ is odd.
  \\
  
  
  \noindent \textbf{Case 7.14.} In this case, we have the following diagram
 \begin{align}\label{D:Dim_7_GS_14}
 (\RS^3\times \RS^3,  \Delta\Gamma,  \RS^3\times \Gamma, \Delta\RS^3),
 \end{align}
   where $\Gamma$ is a finite, non-trivial subgroup of $\RS^3$. Let the local group  at the orbit of codimension $2$ be trivial. Then by Proposition~\ref{P:Orbifold fundamental group}, $\pi_1^{orb}(Q)=\Gamma$. Consider the following diagram from the smooth classification (see \cite[Page~172, $Q_A^7$]{H}):
  \begin{align}\label{D:Dim_7_Manifold_3}
 (\RS^3\times \RS^3,  1,  \RS^3\times 1, \Delta\RS^3).
 \end{align}
This is the diagram of the following action of $\RS^3\times\RS^3$ on $\SP^7\subseteq \mathbb{H}\times\mathbb{H}$:
 \begin{align*}
 (\RS^3\times \RS^3)\times \SP^7&\to \SP^7\\
 ((g, h), (x, y))&\mapsto (gxh^{-1}, hy).
 \end{align*}
 Let $c\colon [-1,1] \to \SP^7$ be a normal geodesic between the non-principal orbits with $c(1)=(1, 0)$. Since $G_{(0, 1)}=\RS^3\times 1$, we have $c(-1)=(0, h_-)=(g_-, h_-)\cdot(0, 1)$, for some $(g_-, h_-)\in  \RS^3\times \RS^3$. Observe that $G_{(\sqrt{2}/2, \sqrt{2}/2)}= 1$. Then, for some $-1<t_0<1$, we have 
 \[
 c(t_0)=((\sqrt{2}/2) g_0h_0^{-1}, (\sqrt{2}/2) h_0)=(g_0, h_0)\cdot(\sqrt{2}/2, \sqrt{2}/2),
 \]
for some $(g_0, h_0)\in  (\RS^3\times \RS^3)$. 
Now define an action of  $\Gamma\subseteq 1\times \RS^3$  on $\SP^7$ as follows:
 \begin{align*}
 ((1, \gamma), (x, y))&\mapsto (x, y\gamma^{-1}).
 \end{align*}
This action commutes with the action of $\RS^3\times \RS^3$ on $\SP^7$, which leads to an $(\RS^3\times \RS^3)$-action on $\SP^7/\Gamma$. By the above explanation, we obtain a diagram as \eqref{D:Dim_7_GS_11} for the normal geodesic $\pi\circ c$, where $\pi: M\to \M/\Gamma$ is the natural projection. Thus $\SP^7$ is the orbifold universal cover of $Q$ and hence $Q$ is good. 
 \\
 
  
  \noindent \textbf{Case 7.15.} In this case, we have the following diagram:
 \begin{align}\label{D:Dim_7_GS_15}
(\RS^3\times \RS^3,  \pm\Delta\Gamma,  \pm\Delta_z\RS^3, \pm\Delta\RS^3),                                                                                                      
 \end{align}
 where $\Gamma=\Z_k$, $4|k$ and $z=j$, or $\Gamma=\D^*_{2m}$, $2|m$ and $z=j$, or $\Gamma=D^*_{2m}$ and $z=i$. 
 
 We consider the following diagram from the smooth classification (see \cite[Case~$Q^7_B$]{Hoelscher}):
  \begin{align}\label{D:Dim_7_Manifold_15}
(\RS^3\times \RS^3,  1,  \Delta_z\RS^3, \Delta\RS^3),                                                                                                      
 \end{align}
 where $z$ is chosen accordingly as above. Note that since in Diagram~\eqref{D:Dim_7_Manifold_15}, $N_{\RS^3\times \RS^3}(1)=\RS^3\times \RS^3$, we can conjugate each isotropy subgroup by $(1,g)$, for any $g\in S^3$,  and still have an equivalent diagram. Now by Proposition~\ref{P:Orbifold fundamental group}, $\pi_1^{orb}(Q)=\pi_1(R_{Q})=\Gamma\times \Z_2$. Then choose $\pm\Delta\Gamma\subset N=\pm\Delta S^3\cap \pm\Delta_z S^3$. Note that  the conditions on $\Gamma$ and $z$ imply that $\pm\Delta\Gamma$ is in fact a subset of $N$.  As before, $\pm\Delta\Gamma\subset N$ acts on $M$ orbitwise and commutes with the $(\RS^3\times \RS^3)$-action on $M$. Therefore,  $\RS^3\times \RS^3$ acts on $M/\pm\Delta\Gamma$ and gives Diagram~\eqref{D:Dim_7_GS_15}. Hence $Q$ is a good orbifold.
 \\

 
 \noindent \textbf{Case 7.16.} Finally, we have the following diagram:
 \begin{align}\label{D:Dim_7_GS_16}
(\RS^3\times \RS^3\times \RS^1, N_{\Delta\RS^3}(\Delta\RS^1)\cdot\mathbb{Z}_k, \T^2\cup (j, j, 1)\T^2,  \Delta\RS^3\cdot\mathbb{Z}_k),
 \end{align}
where $\T^2=\{(e^{i\phi}, e^{i\phi}e^{ip\theta}, e^{i\theta})\}$,
$\Z_k\subseteq \{1, e^{p\theta i}, e^{\theta i}\}$ for $k=1, 2$. Consider the following diagram from the smooth classification:(see \cite[Case~$2_7A2$]{Hoelscher}):
 \begin{align}\label{D:Dim_7_Manifold_4}
 (\RS^3\times \RS^3\times \RS^1, \Delta\RS^1\cdot\mathbb{Z}_k, \T^2,  \Delta\RS^3.\mathbb{Z}_k),
 \end{align}
with the same conditions as \eqref{D:Dim_7_GS_16}. This diagram is in fact the diagram of the cohomogeneity one action of $\RS^3\times \RS^3\times \RS^1$ on a Brieskorn variety $B^7_d$, where $d=p$ in the case that $k=2$ and $d=2p$ in the case that $k=1$ (see \cite[Section~5.2]{Hoelscher}). Now we let the local group at the orbit of codimension $2$ be trivial. Then, by Proposition~\ref{P:Orbifold fundamental group}, $\pi_1^{orb}(Q)=\Z_2$. Choose $(j, j, 1)\in N$, so that $(j, j, 1)\Delta \RS^1$ acts on $M$ orbitwise, via 
\[
(j, j, 1)\Delta \RS^1\cdot gL=g(-j, -j, 1)L,
\]
where $gL$ is a coset in the orbit. The action is well-defined and commutes with the $\RS^3\times \RS^3\times \RS^1$ action on $B^7_d$. Therefore, $\RS^3\times \RS^3\times \RS^1$ acts on $B^7_d/\Z_2$, giving the same diagram as \eqref{D:Dim_7_GS_16}. Thus $B^7_d$ is the orbifold universal cover of $Q$.



\hfill\\
\vspace{1cm}



\begin{table}[htbp]
\begin{center}
\small{
\begin{tabular}{p{9cm}lc}\toprule
Diagram  & Space \hspace{1.5cm} & Orbifold\\
\midrule
$(\RS^3,  \Gamma,  \RS^3, \RS^3)$	&  $\Susp(\RS^3/\Gamma)$ & Yes \\ [.1cm] \mytableextraspace
$(\RS^3,  \Z_k,  \RS^1, \RS^3)$	&  $\CP^2/\Z_k$ & Yes \\ [.1cm] \mytableextraspace
$(\RS^3,  \D_{2m}^*,  N_{\RS^3}(\RS^1), \RS^3)$	&  $\CP^2/ \D_{2m}^*$ & Yes \\ [.1cm] \mytableextraspace
$(\RS^3,  \langle i\rangle,  \{e^{j\theta}\cup i\{e^{j\theta}\}\},	\RS^3)$&  $\CP^2/\langle i \rangle$ & Yes \\ [.1cm] \mytableextraspace
$(\RS^3\times \RS^1, \N_{\RS^3}(\RS^1)\times \Z_k, \N_{\RS^3}(\RS^1)\times \RS^1, \RS^3\times \Z_k)$&  $\RP^2\ast \SP^1$ & Yes \\ 
[.1cm] \mytableextraspace
$(\Spin(4), \N_{\Spin(4)}(\Spin(3)), \Spin(4), \Spin(4))$&  $\Susp(\RP^3)$ & Yes \\
 [.1cm] \mytableextraspace
\bottomrule
\end{tabular}
}
\end{center}
\caption{\label{TB:GROUP_DGMS_D4}Group diagrams in dimension $4$}
\end{table}

\hfill\\
\vspace{2cm}



\begin{table}[htbp]
\begin{center}
\small{
\begin{tabular}{p{9cm}lc}\toprule
Diagram  & Space  & Orbifold\\
\midrule
$(\RS^3\times \RS^1,  \Gamma\times \mathbb{Z}_k,  \Gamma\times \RS^1, \RS^3\times \mathbb{Z}_k)$	&  $(\RS^3/\Gamma)\ast \SP^1$ & Yes \\ [.1cm] \mytableextraspace
$(\RS^3\times \RS^1, \{(e^{\frac{2\pi i}{k}\frac{lk+mps}{m}}, e^{\frac{2\pi i}{k}qs})\}, (\mathbb{Z}_m\times 1)\K^-_0, \RS^3\times \mathbb{Z}_{k/q}),$\newline
$\K^-_0=\{(e^{ip\theta}, e^{iq\theta}\}$, $q|(m, k), (p, k)=1$ 
	&  $\SP^5/\Z_m$ & Yes\\ 
[.1cm] \mytableextraspace
$(\RS^3\times \RS^3, \N_{\RS^3}(\RS^1)\times \RS^1, \N_{\RS^3}(\RS^1)\times \RS^1, \RS^3\times \RS^1)$&  $\RP^2\ast \SP^2$ & Yes \\ 
[.1cm] \mytableextraspace
 $(\RS^3\times \RS^3, \N_{\RS^3}(\RS^1)\times \RS^1, \RS^3\times \RS^1, \RS^3\times \RS^1)$&  $\Susp(\RP^2)\times\SP^2$ & Yes \\ 
[.1cm] \mytableextraspace
$(\RS^3\times \RS^3, \N_{\RS^3}(\RS^1)\times \N_{\RS^3}(\RS^1), \N_{\RS^3}(\RS^1)\times \RS^3, \RS^3\times \N_{\RS^3}(\RS^1))$&  $\RP^2\ast \RP^2$ & Yes \\
[.1cm] \mytableextraspace
$(\SU(3), \UU(2), \SU(3), \SU(3))$& $\Susp(\CP^2)$ & No \\
[.1cm] \mytableextraspace
$(\Spin(5), \N_{\Spin(5)}(\Spin(4)), \Spin(5), \Spin(5))$&  $\Susp(\RP^4)$ & Yes \\
 [.1cm] \mytableextraspace
\bottomrule
\end{tabular}
}
\end{center}
\caption{\label{TB:GROUP_DGMS_D5}Group diagrams in dimension $5$}
\end{table}

\afterpage{%
          \centering 


\begin{table}[p]
\begin{center}
\small{
\begin{tabular}{p{11.2cm}lc}\toprule
Diagram &  Space & Orbifold\\ \mytableextraspace
\midrule
$(\RS^3\times \RS^3, \N_{\RS^3}(\RS^1)\times 1, \N_{\RS^3}(\RS^1)\times \RS^3, \RS^3\times 1)$	&  $\RP^2\ast \SP^3$ & Yes \\ 
[.1cm] \mytableextraspace
$(\RS^3\times \RS^3, \N_{\RS^3}(\RS^1)\times \mathbb{Z}_k, \N_{\RS^3}(\RS^1)\times \RS^1, \RS^3\times \mathbb{Z}_k)$
	&  $(\RP^2\ast\SP^1)$-bundle over   $\SP^2$ & Yes\\ 
[.1cm] \mytableextraspace
$(\RS^3\times \RS^3, \N_{\RS^3}(\RS^1)\times 1, \RS^3\times 1, \RS^3\times 1)$&  $\Susp(\RP^2)\times \SP^3$ & Yes \\ 
[.1cm] \mytableextraspace
 $(\RS^3\times \RS^3, \N_{\RS^3}(\RS^1)\times \Gamma, \N_{\RS^3}(\RS^1)\times \RS^3, \RS^3\times \Gamma)$&  $\RP^2\ast(\RS^3/\Gamma)$ & Yes \\ 
[.1cm] \mytableextraspace
$(\RS^3\times \RS^3, \Delta\RS^1\cup(j, j)\Delta\RS^1, \RS^3\times\N_{\RS^3}(\RS^1), \Delta\RS^3)$&  $\RP^2\ast \SP^3$ & Yes \\
[.1cm] \mytableextraspace
$(\RS^3\times \RS^3, \Delta\RS^1\cup(j, j)\Delta\RS^1, \T^2\cup (j, j)\T^2, \Delta\RS^3)$&  $(\SO(5)/(\SO(2)\SO(3)))/\Z_2$& Yes \\ 
[.1cm] \mytableextraspace
$(\RS^3\times \RS^3, \pm\Delta\RS^1\cup(j,\pm j)\Delta\RS^1, \T^2\cup (j, j)\T^2, \pm\Delta\RS^3)$& $\CP^3/\Z_2$ & Yes \\ 
[.1cm] \mytableextraspace
 $(\RS^3\times \RS^3, \pm\Delta\RS^1\cup(j,\pm j)\Delta\RS^1, \RS^3\times\N_{\RS^3}(\RS^1), \pm\Delta\RS^3)$&  $\RP^2\ast \RP^3$ & Yes \\ 
[.1cm] \mytableextraspace
$(\RS^3\times \RS^3, \Delta\RS^1\cup(j, j)\Delta\RS^1, \Delta\RS^3, \Delta\RS^3)$& $\Susp(\RP^2)\times\SP^3$ & Yes\\
[.1cm] \mytableextraspace
$(\RS^3\times \RS^3, \D^*_{2m}\times\RS^1, \N_{\RS^3}(\RS^1)\times \RS^1, \RS^3\times \RS^1)$&  $\SP^2\times (\CP^2/\D^*_{2m})$ & Yes\\
[.1cm] \mytableextraspace
%
%
$(\RS^3\times \RS^3, \{(e^{ip\theta}\lambda, e^{i\theta})\mid \theta\in \mathbb{R}, \lambda\in\mathbb{Z}_k\}, \T^2, \RS^3\times \RS^1)$&  $(\CP^2/\Z_k)$-bundle over $\SP^2$ & Yes\\ 
[.1cm] \mytableextraspace
$(\RS^3\times \RS^3, \Gamma\times \RS^1, \Gamma\times \RS^3, \RS^3\times \RS^1)$&  $(\RS^3/\Gamma)\ast\SP^2$ & Yes \\ 
[.1cm] \mytableextraspace
 $(\RS^3\times \RS^3, \pm\Delta\RS^1\cup(j,\pm j)\Delta\RS^1, \RS^3\times\N_{\RS^3}(\RS^1), \pm\Delta\RS^3)$&  $\RP^2\ast \RP^3$ & Yes \\ 
[.1cm] \mytableextraspace
$(\RS^3\times \RS^3, \mathbb{Z}_k\times \RS^1, \mathbb{Z}_k\times \RS^3, \RS^3\times \RS^1)$&  $(\RS^3/\mathbb{Z}_k)\ast\SP^2$ & Yes \\
[.1cm] \mytableextraspace
$(\RS^3\times \RS^3, \pm\Delta\RS^1, \pm\Delta\RS^3, \RS^3\times \RS^1)$&  $\RP^3\ast\SP^2$  & Yes \\ 
[.1cm] \mytableextraspace
 $(\RS^3\times \RS^3, \Gamma\times \RS^1, \RS^3\times \RS^1, \RS^3\times \RS^1)$&  $\Susp(\RS^3/\Gamma)\times\SP^2$ & Yes \\ 
[.1cm] \mytableextraspace
$(\RS^3\times \RS^3, \{(e^{ip\theta}\lambda, e^{i\theta})\}, \RS^3\times \RS^1, \RS^3\times \RS^1)$&  $(\SP^4/\Z_k)$-bundle over $\SP^2$ & Yes \\ 
[.1cm] \mytableextraspace
 $(\RS^3\times \RS^3, \{(e^{i\theta}\lambda, e^{i\theta})\}, \RS^1\times \RS^3, \RS^3\times \RS^1)$&  $\CP^3/\Z_k$ & Yes \\ 
[.1cm] \mytableextraspace
$(\RS^3\times \RS^3\times \RS^1, \N_{\RS^3}(\RS^1)\times \RS^1\times \mathbb{Z}_k , \N_{\RS^3}(\RS^1)\times \RS^1\times \RS^1, \RS^3\times \RS^1\times \mathbb{Z}_k)$ & $(\RP^2\ast\SP^1)\times \SP^2$ & Yes \\ 
[.1cm] \mytableextraspace
$(\SU(3), \mathrm{S}(\UU(2)\mathbb{Z}_k) , \UU(2), \SU(3))$&  $\CP^3/\Z_k$ & Yes \\ 
 [.1cm] \mytableextraspace
$(\SU(3), \mathrm{S}(\UU(2)\mathbb{Z}_k) , \SU(3), \SU(3))$&  $\Susp(\SP^5/\mathbb{Z}_k)$ & Yes\\ 
[.1cm] \mytableextraspace
 $(\SU(3)\times\RS^1, \UU(2)\times\mathbb{Z}_k, \UU(2)\times \RS^1, \SU(3)\times\mathbb{Z}_k)$&  $\CP^2\ast\SP^1$ & No \\ 
[.1cm] \mytableextraspace
$(\Sp(2)\times\RS^1, \N_{\Sp(2)}(\Sp(1)\Sp(1))\times\mathbb{Z}_k, \N_{\Sp(2)}(\Sp(1)\Sp(1)) \times\RS^1, \Sp(2)\times\mathbb{Z}_k)$ & $\RP^4\ast\SP^1$ & Yes\\
[.1cm] \mytableextraspace
$(\Spin(6), \N_{\Spin(6)}(\Spin(5)), \Spin(6), \Spin(6))$&  $\Susp(\RP^5)$& Yes\\
 [.1cm] \mytableextraspace
\bottomrule
\end{tabular}
}
\end{center}
\caption{\label{TB:GROUP_DGMS_D6}  Group diagrams in dimension $6$}
\end{table}

    \clearpage
}

\afterpage{%
          \centering 


\begin{table}[tbp]
\begin{center}
\small{
\begin{tabular}{p{10cm}lc}\toprule
 Diagram & Space & Orbifold \\ \mytableextraspace
\midrule 
$(\RS^3\times \RS^3, \Gamma\times\mathbb{Z}_k, \Gamma\times \RS^1, \RS^3\times \mathbb{Z}_k)$	&  $((\RS^3/\Gamma)\ast\SP^1)$-bundle over   $\SP^2$ & Yes\\ 
[.1cm] \mytableextraspace
$(\RS^3\times \RS^3, \mathbb{Z}_k\times 1, \RS^1\times 1, \RS^3\times 1)$	& $(\CP^2/\Z_k)\times\SP^3$  & Yes\\ 
[.1cm] \mytableextraspace
$(\RS^3\times \RS^3, \D^{*}_{2m}\times 1, N_{\RS^3}(\RS^1)\times 1, \RS^3\times 1)$	& $(\CP^2/\Z_m)\times\SP^3$ & Yes \\ 
[.1cm] \mytableextraspace
$(\RS^3\times \RS^3, \{(e^{\frac{lk+mps}{km}2\pi i}, e^{\frac{2\pi qsi}{k}})\}, (\mathbb{Z}_m\times 1)\K^-_0, \RS^3\times \mathbb{Z}_{k/(k,q)})$
	&  $((\SP^5/\Z_q)/\Z_m\Z_k)$-bundle over $\SP^2$ & Yes\\ 
[.1cm] \mytableextraspace
$(\RS^3\times \RS^3, \Gamma\times 1, \Gamma\times \RS^3, \RS^3\times 1)$&  $\RS^3\ast (\RS^3/\Gamma)$ & Yes \\ 
[.1cm] \mytableextraspace
 $(\RS^3\times \RS^3,\mathbb{Z}_2\times 1, \pm \Delta\RS^3, \RS^3\times 1)$&  $\mathbb{R}P^3\ast\SP^3$ & Yes \\ 
[.1cm] \mytableextraspace
$(\RS^3\times \RS^3, \Gamma\times 1, \RS^3\times 1, \RS^3\times 1)$& $\Susp(\RS^3/\Gamma)\times \RS^3$ & Yes \\ 
[.1cm] \mytableextraspace
$(\RS^3\times \RS^3, \Gamma\times \Lambda, \Gamma\times \RS^3, \RS^3\times \Lambda)$& $(\RS^3/\Gamma)\ast(\RS^3/\Lambda)$ & Yes \\ 
[.1cm] \mytableextraspace
$(\RS^3\times \RS^3, \pm \Delta\Lambda, \pm\Delta\RS^3, \RS^3\times \Lambda )$& $(\mathbb{R}P^3)\ast(\RS^3/\Lambda)$ & Yes\\
[.1cm] \mytableextraspace
$(\RS^3\times \RS^3, \Delta\mathbb{Z}_k, \RS^1\times \mathbb{Z}_k, \Delta\RS^3)$&  -- & Yes \\ 
[.1cm] \mytableextraspace
$(\RS^3\times \RS^3, \Delta\D^*_{2m}, (\RS^1\times 1)\Delta\D^*_{2m}, \Delta\RS^3)$& -- & Yes\\ 
[.1cm] \mytableextraspace
$(\RS^3\times \RS^3, \Delta\mathbb{Z}_k, \{(e^{ip\theta}, e^{iq\theta})\}, \Delta\RS^3)$ \newline where $k|(p-q)$ and if $k$ is even $p$, $q$ are odd & -- & Yes  \\ 
[.1cm] \mytableextraspace
$(\RS^3\times \RS^3, \Delta\mathbb{Z}_k, \{(e^{ip\theta}, e^{-iq\theta})\}, \Delta\RS^3)$ \newline 
where $k|(p+q)$ and if $k$ is even $p$, $q$ are odd. & -- & Yes\\ 
 [.1cm] \mytableextraspace
$(\RS^3\times \RS^3, \Delta\D^*_{2m}, \{(e^{ip\theta}, e^{iq\theta})\}\cup \{(je^{ip\theta}, je^{iq\theta})\}, \Delta\RS^3)$ \newline where $m|(p-q)$ and if $m$ is even $p$, $q$ are odd
& -- & Yes \\ 
[.1cm] \mytableextraspace
$(\RS^3\times \RS^3, \Delta\D^*_{2k}, \{(e^{ip\theta}, e^{-iq\theta})\}\cup \{(je^{ip\theta}, je^{-iq\theta})\}, \Delta\RS^3)$ \newline where $k|(p+q)$ and if $k$ is even $p$, $q$ are odd
&  -- & Yes\\ 
[.1cm] \mytableextraspace
$(\RS^3\times \RS^3, \pm\Delta\mathbb{Z}_k, \RS^1\times \mathbb{Z}_k, \pm\Delta\RS^3)$& -- & Yes\\ 
[.1cm] \mytableextraspace
$(\RS^3\times \RS^3, \pm\Delta\D^*_{2m}, (\RS^1\times 1)\Delta\D^*_{2m}, \pm\Delta\RS^3) 
$ &-- & Yes\\
[.1cm] \mytableextraspace
$(\RS^3\times \RS^3, \pm\Delta\mathbb{Z}_k, \{(e^{ip\theta}, e^{iq\theta})\}, \pm\Delta\RS^3)$  &-- & Yes\\
[.1cm] \mytableextraspace
$ (\RS^3\times \RS^3, \pm\Delta\mathbb{Z}_k, \{(e^{ip\theta}, e^{-iq\theta})\}, \pm\Delta\RS^3)$ 
& -- & Yes\\
[.1cm] \mytableextraspace
$(\RS^3\times \RS^3, \pm\Delta\D^*_{2k}, \{(e^{ip\theta}, e^{iq\theta})\}\cup \{(je^{ip\theta}, je^{iq\theta})\}, \pm\Delta\RS^3)
$ & -- & Yes \\
[.1cm] \mytableextraspace
$(\RS^3\times \RS^3, \pm\Delta\D^*_{2k}, \{(e^{ip\theta}, e^{-iq\theta})\}\cup \{(je^{ip\theta}, je^{-iq\theta})\}, \pm\Delta\RS^3)$& -- & Yes\\
[.1cm] \mytableextraspace
$(\RS^3\times \RS^3,  \Delta\Gamma,  \RS^3\times \Gamma, \Delta\RS^3)$ & $\SP^7/\Gamma$ & Yes\\
[.1cm] \mytableextraspace
$(\RS^3\times \RS^3,  \Delta\Gamma,  \Delta\RS^3, \Delta\RS^3)$& $\Susp(\RS^3/\Gamma)\times \SP^3$ & Yes\\
 [.1cm] \mytableextraspace
$(\RS^3\times \RS^3,  \pm\Delta\Z_{2k},  \pm\Delta_j\RS^3, \pm\Delta\RS^3)$&  -- & Yes\\  
 [.1cm] \mytableextraspace
$(\RS^3\times \RS^3,  \pm\Delta\D^*_{2k},  \pm\Delta_j\RS^3, \pm\Delta\RS^3)$&  -- & Yes\\  
 [.1cm] \mytableextraspace
 $(\RS^3\times \RS^3,  \pm\Delta\D^*_{2k},  \pm\Delta_i\RS^3, \pm\Delta\RS^3)$&  -- & Yes\\ 
 [.1cm] \mytableextraspace
\bottomrule
\end{tabular}
}
\end{center}
\caption{\label{TB:GROUP_DGMS_D7_S3xS3} Group diagrams for $\RS^3\times \RS^3$   in dimension $7$ }
\end{table}

    \clearpage
}

\afterpage{
          \centering 

\begin{table}[tbp]
\begin{center}
\small{
\begin{tabular}{p{10cm}lc}
\toprule
 Diagram & Space & Orbifold \\ \mytableextraspace
\midrule 

$(\RS^3\times \RS^3\times \RS^1, \N(\RS^1)\times\mathbb{Z}_k, \N(\RS^1)\times\T^1, \RS^3\times \mathbb{Z}_k)$ & $(\RP^2\ast\SP^1)$-bundle over $\SP^2$  & Yes  
\\ 
[.1cm] \mytableextraspace
 $(\RS^3\times \RS^3\times \RS^1, \N_{\Delta\RS^3}(\Delta\RS^1)\mathbb{Z}_k, \T^2\cup (j, j, 1)\T^2,  \Delta\RS^3.\mathbb{Z}_k)$&  $B^7_d/\Z_2$ & Yes\\ 
[.1cm] \mytableextraspace
$(\RS^3\times \RS^3\times \RS^1, \Gamma\times\RS^1\times\mathbb{Z}_k, \Gamma\times\T^2,  \RS^3\times\RS^1\times\mathbb{Z}_k)$ &$(\RS^3/\Gamma\ast\SP^1)\times\SP^2$ & Yes\\
[.1cm] \mytableextraspace
$(\RS^3\times \RS^3\times \RS^3, \N_{\RS^3}(\RS^1)\times \T^2, \N_{\RS^3}(\RS^1)\times \RS^3\times \RS^1,  \RS^3\times\T^2)$& $(\RP^2\ast\SP^2)\times\SP^2$ & Yes\\
[.1cm] \mytableextraspace
$(\RS^3\times \RS^3\times \RS^3, \N_{\RS^3}(\RS^1)\times \T^2, \RS^3\times\T^2,  \RS^3\times\T^2)$& $\Susp(\RP^2)\times(\SP^2\times \SP^2)$ & Yes\\
[.1cm] \mytableextraspace
$(\RS^3\times \RS^3\times \RS^3, \N_{\RS^3}(\RS^1)\times\N_{\RS^3}(\RS^1)\times\RS^1, \N_{\RS^3}(\RS^1)\times \RS^3\times \RS^1,  \RS^3\times\N_{\RS^3}(\RS^1)\times \RS^1)$&  $(\RP^2\ast\RP^2)\times\SP^2$ & Yes\\
[.1cm] \mytableextraspace
\bottomrule
\end{tabular}
}
\end{center}
\caption{\label{TB:GROUP_DGMS_D7_S3xS3xS1_S3xS3xS3} Group diagrams for $\RS^3\times \RS^3\times \RS^1$ and $\RS^3\times \RS^3\times \RS^3$ in dimension $7$Œ}
\end{table}

\begin{table}[tbp]
\begin{center}
\small{
\begin{tabular}{p{10cm}lc}
\toprule
 Diagram & Space & Orbifold\\ \mytableextraspace
\midrule 
$(\SU(3), \T^2, \SU(3), \SU(3))$ &$\Susp(\W^6)$  & No\\
[.1cm] \mytableextraspace
$(\SU(3), \T^2\mathbb{Z}_2, \SU(3), \SU(3))$& $\Susp(\W^6/\mathbb{Z}_2)$ & No\\
[.1cm] \mytableextraspace
$(\SU(3), \T^2, \U(2), \SU(3))$& -- & No\\
[.1cm] \mytableextraspace
$(\SU(3), \T^2\mathbb{Z}_2, \U(2), \SU(3))$&  -- & No\\
[.1cm] \mytableextraspace
$(\SU(3), \T^2\mathbb{Z}_2, \U(2), \U(2))$&  $\Susp(\RP^2)$-bundle over $\mathbb{C}P^2$& Yes\\
[.1cm] \mytableextraspace
$(\SU(3)\times \RS^1, \RS(U(2).\mathbb{Z}_k)\times \mathbb{Z}_l, \RS(\UU(2).\mathbb{Z}_k)\times \RS^1, \SU^3\times \mathbb{Z}_l)$&  $(\mathbb{\RS}^5/\Z_k)\ast\SP^1$ & Yes\\
[.1cm] \mytableextraspace
$(\SU(3)\times \RS^3, \UU(2)\times \RS^1, \SU(3)\times\RS^1, \SU(3)\times\RS^1)$ & $\Susp(\CP^2)\times\SP^2$ & No\\
[.1cm] \mytableextraspace
$(\SU(3)\times \RS^3, \UU(2)\times \RS^1, \SU(3)\times\RS^1, \UU(2)\times \RS^3)$& $\mathbb{C}P^2\ast\SP^2$ & No\\
[.1cm] \mytableextraspace
$(\SU(3)\times \RS^3, \UU(2)\times \N_{\RS^3}(\RS^1), \SU(3)\times\N_{\RS^3}(\RS^1), \UU(2)\times \RS^3)$& $\mathbb{C}P^2\ast\RP^2$ & No\\
[.1cm] \mytableextraspace
$(\SU(3)\times \RS^3, \UU(2)\times \N_{\RS^3}(\RS^1), \UU(2)\times \RS^3, \UU(2)\times \RS^3)$&  $\Susp(\RP^2)\times\CP^2$ & Yes\\
[.1cm] \mytableextraspace
\bottomrule
\end{tabular}
}
\end{center}
\caption{\label{TB:GROUP_DGMS_D7_SU(3)_SU(3)xS3}Group diagrams for $\SU(3)$ and $\SU(3)\times \RS^3$  in dimension $7$Œ}
\end{table}
    \clearpage

}

\afterpage{
          \centering 

\begin{table}[p]
\begin{center}
\small{
\begin{tabular}{p{10cm}lc}
\toprule
 Diagram & Space & Orbifold \\ \mytableextraspace
\midrule 
$(\Sp(2), \Sp(1)\SO(2), \Sp(2), \Sp(2))$& $\Susp(\mathbb{C}P^3)$ & No\\
[.1cm] \mytableextraspace
$(\Sp(2), \Sp(1)\SO(2)\mathbb{Z}_2, \Sp(2), \Sp(2))$& $\Susp(\mathbb{C}P^3/\mathbb{Z}_2)$ & No\\
[.1cm] \mytableextraspace
$(\Sp(2), \Sp(1)\SO(2), \Sp(1)\Sp(1), \Sp(2))$&  -- & No\\
[.1cm] \mytableextraspace
$(\Sp(2), \Sp(1)\SO(2)\mathbb{Z}_2, \Sp(1)\Sp(1), \Sp(2))$ &-- & No\\
[.1cm] \mytableextraspace
$(\Sp(2), \Sp(1)\SO(2)\mathbb{Z}_2, \Sp(1)\Sp(1), \Sp(1)\Sp(1))$& $\Susp(\RP^2)$-bundle over $\SP^4$ & Yes\\
[.1cm] \mytableextraspace
$(\Sp(2)\times\RS^3, \Sp(1)\Sp(1)\times\N_{\RS^3}(\RS^1), \Sp(1)\Sp(1)\times\RS^3, \Sp(1)\Sp(1)\times\RS^3)$& $\Susp(\RP^2)\times\SP^4$ & Yes \\
[.1cm] \mytableextraspace
$(\Sp(2)\times\RS^3, \Sp(1)\Sp(1)\times\N_{\RS^3}(\RS^1), \Sp(1)\Sp(1)\times\RS^3, \Sp(2)\times\N_{\RS^3}(\RS^1))$&  $\SP^4\ast\RP^2$ & Yes\\
[.1cm] \mytableextraspace
$(\Sp(2)\times\RS^3, \Sp(1)\Sp(1)\mathbb{Z}_2\times\RS^1, \Sp(1)\Sp(1)\mathbb{Z}_2\times\RS^3, \Sp(2)\times\RS^1)$& $\SP^2\ast\RP^4$ & Yes\\
[.1cm] \mytableextraspace
$(\Sp(2)\times\RS^3, \Sp(1)\Sp(1)\mathbb{Z}_2\times\N_{\RS^3}(\RS^1), \Sp(1)\Sp(1)\mathbb{Z}_2\times\RS^3, \Sp(2)\times\N_{\RS^3}(\RS^1))$& $\RP^4\ast\RP^2$ & Yes\\
[.1cm] \mytableextraspace
$(\Sp(2)\times\RS^3, \Sp(1)\Sp(1)\mathbb{Z}_2\times\RS^1, \Sp(2)\times \RS^1, \Sp(2)\times \RS^1)$& $\Susp(\RP^4)\times\SP^2$ & Yes\\
[.1cm] \mytableextraspace
\bottomrule
\end{tabular}
}
\end{center}
\caption{\label{TB:GROUP_DGMS_D7_Sp(2)_Sp(2)xS3}Group diagrams for $\Sp(2)$ and $\Sp(2)\times \RS^3$  in dimension $7$Œ}
\end{table}

\begin{table}[p]
\begin{center}
\small{
\begin{tabular}{p{10cm}lc}
\toprule
 Diagram & Space\hspace{3.2cm} & Orbifold \\ \mytableextraspace
\midrule 
$(\mathrm{G}_2, \N_{\mathrm{G}_2}(\SU(3)), \mathrm{G}_2, \mathrm{G}_2)$ & $\Susp(\RP^6)$  & Yes\\
[.1cm] \mytableextraspace
$(\SU(4), \UU(3), \SU(4), \SU(4))$& $\Susp(\mathbb{C}P^3)$ & No\\
[.1cm] \mytableextraspace
$(\SU(4)\times \RS^1, \Sp(2)\mathbb{Z}_2\times\mathbb{Z}_k, \Sp(2)\mathbb{Z}_2\times\RS^1, \SU(4)\times\mathbb{Z}_k)$& $\RP^5\ast\SP^1$ & Yes\\
[.1cm] \mytableextraspace
$(\Spin(7), \N_{\Spin(7)}(\Spin(6)), \Spin(7), \Spin(7))$& $\Susp(\RP^6)$ & Yes\\
\bottomrule
\end{tabular}
}
\end{center}
\caption{\label{TB:GROUP_DGMS_D7_G2_SU(4)_SU(3)xS1_Spin(7)} Group diagrams for $\mathrm{G}_2$, $\SU(4)$, $\SU(4)\times \RS^1$ and $\Spin(7)$  in dimension $7$}
\end{table}

    \clearpage
}


\newpage
\bibliography{C1A_Current}
\bibliographystyle{amsplain}

\end{document}